\documentclass{amsart}
\usepackage[all]{xy}
\usepackage{tikz}
\usetikzlibrary{shapes.geometric}
\usetikzlibrary{shapes.multipart} %Circle split
\usetikzlibrary{arrows}
\usetikzlibrary{cd}
\usetikzlibrary{snakes}

    \tikzstyle{int}=[circle, draw,fill=black,outer sep=0,minimum size=3pt, inner sep=0]
    \tikzstyle{ext}=[circle, draw=black,outer sep=0,inner sep=1pt]
    \tikzstyle{black}=[circle,fill=black!,inner sep=0pt,minimum size=2mm]
    \tikzstyle{white}=[circle,inner sep=0pt,draw=black,minimum size=2mm]
    \tikzstyle{invisible}=[coordinate,inner sep=0pt,minimum size=2mm]
    \tikzstyle{biw}=[circle split,inner sep=1pt,draw=black,minimum size=8mm]
    \tikzstyle{biwb}=[rectangle,inner sep=1pt,draw=black,minimum size=8mm]
    \tikzstyle{uniw}=[circle,inner sep=1pt,draw=black,minimum size=8mm]
    \tikzstyle{unimin}=[circle,inner sep=1pt,draw=black,minimum size=5mm]

\usepackage{xcolor}
    \definecolor{mygray}{gray}{0.9}
\usepackage{latexsym}

\usepackage{amssymb}
\usepackage{amsmath}
\usepackage{amsthm}
\usepackage{amscd}
\usepackage{fancyhdr}
\usepackage{fullpage}
\usepackage{mathrsfs}

\xyoption{arc}
\usepackage[center]{caption}

\newcommand{\GC}{\mathsf{GC}}

\newcommand{\dGC}{\mathsf{dGC}}

\newcommand{\wGC}{\mathsf{wGC}}
\newcommand{\OGC}{\mathsf{oGC}}
\newcommand{\oGC}{\mathsf{oGC}}

\newcommand{\fWGC}{\mathsf{fwGC}}
\newcommand{\fwGC}{\mathsf{fwGC}}
\newcommand{\fpWGC}{\mathsf{fwPGC}}
\newcommand{\fqWGC}{\mathsf{fwQGC}}
\newcommand{\fpwGC}{\mathsf{fwPGC}}
\newcommand{\fqwGC}{\mathsf{fwQGC}}

\newcommand{\pwGC}{\mathsf{wPGC}}
\newcommand{\qwGC}{\mathsf{wQGC}}
\newcommand{\OWGC}{\mathsf{owGC}}

\newcommand{\pOWGC}{\mathsf{owPGC}}
\newcommand{\owGC}{\mathsf{owGC}}
\newcommand{\oqwGC}{\mathsf{owQGC}}
\newcommand{\opwGC}{\mathsf{owPGC}}
\newcommand{\fowGC}{\mathsf{fowGC}}
\newcommand{\foqwGC}{\mathsf{fowQGC}}
\newcommand{\fopwGC}{\mathsf{fowPGC}}

\newcommand{\qGC}{\mathsf{qGC}}
\newcommand{\qqGC}{\mathsf{qQGC}}
\newcommand{\pqGC}{\mathsf{qPGC}}
\newcommand{\oqGC}{\mathsf{oqGC}}
\newcommand{\oqqGC}{\mathsf{oqQGC}}
\newcommand{\opqGC}{\mathsf{oqPGC}}

\newcommand{\tGC}{\mathsf{tGC}}
\newcommand{\qtGC}{\mathsf{tQGC}}
\newcommand{\ptGC}{\mathsf{tPGC}}

\newcommand{\qmGC}{\mathsf{mQGC}}
\newcommand{\pmGC}{\mathsf{mPGC}}

 \newcommand{\bu}{\bullet}
 \newcommand{\ad}{{\mathrm a\mathrm d}}

\theoremstyle{plain}
\swapnumbers
\newtheorem{theorem}{Theorem}[subsection]
\newtheorem{corollary}[theorem]{Corollary}

\newtheorem{lemma}[theorem]{Lemma}
\newtheorem{proposition}[theorem]{Proposition}

\newtheorem{prop-def}[theorem]{Proposition-definition}

\newtheorem{main-theorem}{Main~Theorem}[section]
\newtheorem{section-theorem}{Theorem}[section]
\newtheorem{section-corollary}{Corollary}[section]

\theoremstyle{definition}
\newtheorem{example}[theorem]{Example}
\newtheorem{remark}[theorem]{Remark}
\newtheorem{definition}[theorem]{Definition}

\usetikzlibrary{decorations.markings}

\def\MarkLt{6pt}
\def\MarkSep{3pt}

\tikzset{
  TwoMarks/.style={
    postaction={decorate,
      decoration={
        markings,
        mark=at position #1 with
          {
              \begin{scope}[xslant=0.2]
              \draw[line width=\MarkSep,white,-] (0pt,-\MarkLt) -- (0pt,\MarkLt) ;
              \draw[-] (-0.5*\MarkSep,-\MarkLt) -- (-0.5*\MarkSep,\MarkLt) ;
              \draw[-] (0.5*\MarkSep,-\MarkLt) -- (0.5*\MarkSep,\MarkLt) ;
              \end{scope}
          }
       }
    }
  },
  TwoMarks/.default={0.5},
  OneMark/.style={
    postaction={decorate,
      decoration={
        markings,
        mark=at position #1 with
          {
              \draw[-,solid] (0,-\MarkLt) -- (0,\MarkLt) ;
          }
       }
    }
  },
  OneMark/.default={0.5},
  NMark/.style={
    postaction={decorate,
      decoration={
        markings,
        mark=at position #1 with
          {
              \draw[-,solid] (-3pt,-\MarkLt) -- (-3pt,\MarkLt) ;
              \draw[-,solid] (3pt,-\MarkLt) -- (3pt,\MarkLt) ;
              
          }
       }
    }
  },
  NMark/.default={0.5}
}

\long\def\symbolfootnote[#1]#2{\begingroup%
\def\thefootnote{\fnsymbol{footnote}}\footnote[#1]{#2}\endgroup}

\author{Oskar Frost}
\address{Mathematics Research Unit, University of Luxembourg, Maison du Nombre, 6 Avenue de la Fonte,
 L-4364 Esch-sur-Alzette, Grand Duchy of Luxembourg }
\email{oskar.frost@uni.lu}
\title{Graph complexes and Deformation theories of the (wheeled) properads of quasi- and pseudo-Lie bialgebras}

\begin{document}
\begin{abstract}
    Quasi-Lie bialgebras are natural extensions of Lie-bialgebras, where the cobracket satisfies the co-Jacobi relation up to some natural obstruction controlled by a skew-symmetric $3$-tensor $\phi$. This structure was introduced by Drinfeld while studying deformation theory of universal enveloping algebras and has since seen many other applications in algebra and geometry. In this paper we study the derivation complex of strongly homotopy quasi-Lie bialgebra, both in the unwheeled (i.e standard) and wheeled case, and compute its cohomology in terms of Kontsevich graph complexes.
\end{abstract}
\maketitle

{\large
\section{Introduction and outline}
}
\label{sec:introduction}
\subsection{Introduction}
A Lie bialgebra is a vector space equipped with a bracket $[-,-]$ and a cobracket $\delta$, such that $(V,[-,-])$ is a Lie algebra, $(V,\delta)$ is a Lie coalgebra and that the bracket and cobracket satisfy the Drinfeld compability relation (that $\delta$ is a 1-cocycle). First introduced by V. Drinfeld \cite{D1} in the theory of Yang Baxter equations and quantum groups, this notion has later found many other applications in a wide range of areas of pure mathematics and mathematical physics.\\

A quasi-Lie bialgebra is a vector space equipped with a bracket $[-,-]$, a cobracket $\delta$ and a skew-symmetric element $\phi\in V\otimes V\otimes V$ such that $(V,[-,-])$ is a Lie algebra, and $\delta$ satisfies co-Jacobi relation up to some natural relation with $\phi$. These structures were also introduced by V. Drinfeld \cite{D2} and their quantization theory was developed later by B. Enriques and G. Halbout \cite{EH}. Their strongly homotopy versions have been studied in \cite{G,Kr}. Quasi-Lie bialgebras appear in many different areas. V. Turaev has found such structures when studying quasi-surfaces \cite{Tu}.
Recently quasi-Lie bialgebra structures have been found in the theory of twisting of properads \cite{M3} and the theory of cohomology groups of the moduli spaces $\mathcal{M}_{g,n}$ of genus $g$ algebraic curves with $n$ punctures \cite{M4}.
Our main purpose in this paper is to study the deformation theory of the properad $\mathcal{QL}ieb$ governing quasi-Lie bialgebras, as well as the deformation of the wheeled closure of its minimal resolution $\mathcal{QH}olieb$, first constructed in \cite{G}.
This work complements a similar study of the properad of Lie bialgebras $\mathcal{L}ieb$ and of the wheeled closure of its minimal resolution $\mathcal{H}olieb$ given in \cite{MW1} and \cite{F} respectively.\\

Graph complexes is a collective notion for chain complexes whose generators are linear combinations of various types of graphs. The type of graphs considered most often include undirected graphs, directed graphs, hairy graphs, ribbon graphs, and so on. The most famous series of graph complexes is given by the Kontsevich graph complexes $\GC_k$ indexed by an integer $k$ which are generated by undirected graphs. The Kontsevich graph complex carries, like many other graph complexes, a Lie algebra structure compatible with its differential. T. Willwacher proved in \cite{W1} that the zeroth cohomology of $\GC_2$ is equal to the Lie algebra $\mathfrak{grt}_1$ of the pro-unipotent Grothendieck-Teichmuller group $GRT_1$ introduced by Drinfeld in \cite{D3}. The directed graph complex $\dGC_k$ consisting of directed graphs is quasi-isomorhpic to $\GC_d$ and has many interesting subcomplexes. For example, the oriented graph complex $\OGC_{k+1}$ is the subcomplex of $\dGC_{k+1}$ spanned by graphs containing no closed directed paths. The targeted graph complex $\dGC_{d+1}^t$ is the subcomplex spanned graphs with at least one target vertex (i.e a vertex with no outgoing directed edges attached). Both of these subcomplexes of $\dGC_{k+1}$ are quasi-isomorphic to $\GC_{k}$ \cite{W2,Z2}.\\

In this paper we study the deformation theory of homotopy Lie bialgebras and homotopy quasi-Lie bialgebras in larger generality. All the above mentioned properads have degree shifted versions $\mathcal{H}olieb_{c,d}$ and $\mathcal{QH}olieb_{c,d}$ in which the Lie bracket has degree $1-d$ and the Lie co-bracket has degree $1-c$, so that the case $c=d=1$ corresponds to the usual (quasi-) Lie bialgebras. We study in this paper $\mathcal{QH}olieb_{c,d}$ for arbitrary $c$ and $d$.
In \cite{MW1} a quasi-isomorphism $G:\OGC_{c+d+1}\rightarrow Der(\mathcal{H}olieb_{c,d})$  (up to a rescaling class) has been constructed. In this paper we will give an alternative proof of this statement, and extend it to the deformation complex of the homotopy quasi-Lie bialgebra properad.
\begin{theorem}[Main Theorem I]
    For any $c,d\in\mathbb{Z}$, there is an explicit morphism of dg Lie algebras
    \begin{equation*}
        G:\OGC_{c+d+1}\rightarrow Der(\mathcal{QH}olieb_{c,d})
    \end{equation*}
    which is a quasi-isomorphism up to a rescaling class.
\end{theorem}
The map is defined by mapping an oriented graph to a graph of the same shape, summing over all possible ways of attaching hairs to the vertices. We immediately get the following corollary.
\begin{corollary} \
$H^0(Der(\mathcal{QH}olieb_{1,1}))=H^0(\OGC_3)=\mathfrak{grt}_1$
\end{corollary}
The non-wheeled properad $\mathcal{H}olieb_{c,d}$ (respectively $\mathcal{QH}olieb_{c,d}$) is spanned by directed graphs with at least one incoming and at least one outgoing leg (respectively with at least one outgoing leg). Once we consider their wheeled extensions $\mathcal{H}olieb_{c,d}^\circlearrowleft$ (respectively $\mathcal{QH}olieb_{c,d}^\circlearrowleft$), we have to work with a much larger family of graphs which might not have incoming or outgoing legs at all. Therefore, when studying the deformation theory of Lie- or quasi-Lie bialgebras, it is natural to expect that generic deformations of original structures might produce new homotopy operations of "curvature" type, i.e with no inputs at all, or with no outputs, or both. Hence in the wheeled case, it is natural to study several deformation theories of the original Lie (respectively quasi-Lie) bialgebra structure, the one which allows deformation which create new curvature type homotopy operations, and the other one where such a creation is prohibited. Denote the first type derivation complex by $Der^*(\mathcal{H}olieb_{c,d}^\circlearrowleft)$ (respectively $Der^*(\mathcal{QH}olieb_{c,d}^\circlearrowleft)$), and the second type derivation complex by $Der(\mathcal{H}olieb_{c,d}^\circlearrowleft)$ (respectively $Der(\mathcal{H}olieb_{c,d}^\circlearrowleft)$). We computed the cohomologies of both $Der^*(\mathcal{H}olieb_{c,d}^\circlearrowleft)$ and $Der(\mathcal{H}olieb_{c,d}^\circlearrowleft)$ in our previous paper \cite{F} and related their cohomologies to Kontsevich graph complexes. In this paper we compute the cohomology of $Der^*(\mathcal{QH}olieb_{c,d}^\circlearrowleft)$ and $Der(\mathcal{QH}olieb_{c,d}^\circlearrowleft)$. We show complete answers for both complexes; not very surprisingly, the cohomology groups of the two complexes turn out to be rather different.
\begin{theorem}[Main Theorem II]
    For any $c,d\in\mathbb{Z}$, there is an explicit morphism of dg Lie algebras
    \begin{equation*}
        G:\dGC_{c+d+1}^t\rightarrow Der(\mathcal{QH}olieb_{c,d}^\circlearrowleft)
    \end{equation*}
    which is a quasi-isomorphism up to a rescaling class.
\end{theorem}
\begin{theorem}[Main Theorem III]\label{thm:Der(QH*)}
    For any $c,d\in\mathbb{Z}$, there is an explicit morphism of dg Lie algebras
    \begin{equation*}
        G:\dGC_{c+d+1}\oplus \dGC_{c+d+1}^{\geq3,no \ t}\rightarrow Der^*(\mathcal{QH}olieb_{c,d}^\circlearrowleft)
    \end{equation*}
    which is a quasi-isomorphism up to a rescaling class. Here $\dGC_{c+d+1}^{\geq3,no \ t}$ is the quotient complex spanned by directed graphs with no bivalent vertices nor target vertices. 
\end{theorem}
Let us remark that the direct summand $\dGC_{c+d+1}^{\geq3,no \ t}$ arises in the above correspondence for a "trivial" reason. The direct summand correspond precisely to the direct summand of the complex $Der(\mathcal{QH}olieb_{c,d}^\circlearrowleft)$ with no in- or out-legs. The isomorphism of these complexes is immediate, so the non-trivial part of the above theorem established an isomorphism between $\dGC_{c+d+1}$ and the complement in $Der(\mathcal{QH}olieb_{c,d}^\circlearrowleft)$ to the above trivial subcomplex.
To explain this new situation with the $Der(\mathcal{QH}olieb_{c,d}^\circlearrowleft)$ and $Der^*(\mathcal{QH}olieb_{c,d}^\circlearrowleft)$ complexes better, it is useful to consider a further extension of the notion of Lie bialgebras. An extension where one allows not only the curvature term $\phi\in\wedge^3 V$ (as in the case of $\mathcal{QL}ieb_{c,d}$), but also its dual curvature term $\psi\in\wedge^3 V^*$ satisfying natural compability relations. This extension has been introduced and studied by J. Granåker in \cite{G} and we denote associated the degree shifted properad (respectively homotopy properad) by $\mathcal{PL}ieb_{c,d}$ (respectively $\mathcal{PH}olieb_{c,d}$).
\begin{theorem}[Main Theorem IV]
    For any $c,d\in\mathbb{Z}$, there is an explicit dg morphism
\begin{align*}
    G:\OGC_{c+d+1}\rightarrow Der(\mathcal{PH}olieb_{c,d})
\end{align*}
which is a quasi-isomorphism up to a rescaling class.
\end{theorem}
\begin{theorem}[Main Theorem V]
    There is an explicit dg morphism
\begin{align*}
    G:\dGC_{c+d+1}\oplus\dGC_{c+d+1}^{\geq3}\rightarrow Der(\mathcal{PH}olieb_{c,d}^\circlearrowleft)
\end{align*}
which is a quasi-isomorphism up to a rescaling class. Here $\dGC_{c+d+1}^{\geq3}$ denotes the quotient complex spanned by directed graphs with no bivalent vertices.
\end{theorem}
Similar to the remark for Theorem \ref{thm:Der(QH*)}, the direct summand $\dGC_k^{\geq3}$ trivially correspond to the direct summand of the derivation complex with no in- and out-legs.
\subsection{Outline}
    In section two we introduce the necessary background behind the theory of quasi-, pseudo- and normal Lie bialgebras as well as their properads and their deformation complexes.
    In section three we give a brief reminder about the Kontsevich graph complexes and their interrelations. We also define a new auxiliary graph complex $\wGC_k$ of bi-weighted graphs, as well as the corresponding quasi and pseudo versions $\qwGC_k$ and $\pwGC_k$, and show that they are isomorphic to the above mentioned derivation complexes of (wheeled) properads.
    In section four we consider two filtrations on the complexes $\wGC_k$, $\qwGC_k$ and $\pwGC_k$ by special-in and special-out vertices, and find the "much smaller" quasi-isomorphic subcomplexes $\qGC_k$, $\qqGC_k$ and $\pqGC_k$ which are generated by graphs with vertices of four kinds.
    In section five we compute the cohomologies of these "small" complexes and show the main theorems.
\subsection{Notation}

Let $\mathbb{S}_n$ denote the permutation group of the set $\{1,2,...,n\}$. The one dimensional sign representation of $\mathbb{S}_n$ is denoted by $sgn_n$. All vector spaces are assumed to be $\mathbb{Z}$ graded over a field $\mathbb{K}$ of characteristic zero. If $V=\bigoplus_{i\in\mathbb{Z}} V^i$ is a graded vector space, then $V[k]$ denotes the graded vector space where $V[k]^i=V^{i+k}$. For a properad $\mathcal{P}$ we denote by $\mathcal{P}\{k\}$ the properad being uniquely defined by the property: for any graded vector space $V$, a representation of $\mathcal{P}\{k\}$ in $V$ is identical to a representation of $\mathcal{P}$ in $V[k]$. Also let $\mathrm{Hom}^k$(V,W) denote the morphisms of degree $k$ from $V$ to the graded vector space $W$.

\subsection*{Acknowledgements}
I would like to thank Sergei Merkulov for giving the motivation to study this problem and his helpful guidance along the way.

\section{Lie bialgebras and deformation complexes of properads}
\subsection{Pseudo- and quasi-Lie bialgebras}
\begin{definition}
    Let $n\in\mathbb{Z}$. A \textit{Lie }$n$\textit{-bialgebra} is a $\mathbb{Z}$ graded vector space $V$ together with two operations $[\,,]:V\wedge V\rightarrow V$ and $\delta:V\rightarrow V\wedge V$. $[\,,]$ is a Lie-bracket of degree $-n$, $\delta$ is a Lie coproduct of degree $0$ and the operations satisfy the relation
    \begin{equation*}
        \delta([a,b])=(\ad_a\otimes 1 + 1\otimes\ad_a)\delta(b)-(\ad_b\otimes 1 + 1\otimes\ad_b)\delta(a)
    \end{equation*}
    for any $a,b\in V$ and $\ad_a(b)=[a,b]$.
    Let $n\in\mathbb{Z}$. A \textit{pseudo-Lie $n$-bialgebra} is a 5-tuple $(V,[-,-],\delta,\phi,\eta)$ where $V$ is a $\mathbb{Z}$ graded vector space and the other symbols denote morphisms $\phi\in\mathrm{Hom}_V^{n}(\mathbb{K},V^{\wedge 3})$, $\delta\in\mathrm{Hom}_V^0(V,V^{\wedge 2})$, $[-,-]\in\mathrm{Hom}_V^{-n}(V^{\wedge 2},V)$ and $ \eta\in\mathrm{Hom}_V^{-2n}(V^{\wedge 3},\mathbb{K})$ such that
    \begin{align*}
        \delta[a,b]&=(\ad_a\otimes 1 + 1\otimes\ad_a)\delta(b)-(\ad_b\otimes 1 + 1\otimes\ad_b)\delta(a) + \sum_\phi \phi_1\otimes\phi_2 (\eta(\phi_3\otimes a\otimes b))\\
        \dfrac{1}{2}\mathrm{Alt}_3(\delta\otimes\mathrm{id})(\delta(a))&=[a\otimes1\otimes1+1\otimes1\otimes1+1\otimes1\otimes a,\phi]\\
        [[a,b],c]+[[c,a],b]+[[b,c],a]&=\eta(a,b,\delta_1(c))\delta_2(c)+\eta(c,a,\delta_1(b))\delta_2(b)+\eta(b,c,\delta_1(a))\delta_2(a)\\
        0&=\mathrm{Alt}_4(\delta\otimes\mathrm{id}\otimes\mathrm{id})(\phi)\\
        0&=\eta([x,y],z,w)+\eta([x,z],y,w)+\eta([x,w],y,z)\\
        & \ \ \ +\eta([y,z],x,w)+\eta([y,w],x,z)+\eta([z,w],x,y)
    \end{align*}
    where $\mathrm{Alt}_k:V^{\otimes k}\rightarrow V^{\otimes k}$ denotes the operator $\sum_{\sigma\in\mathbb{S}_k}\mathrm{sgn}(\sigma)\sigma$. A \textit{quasi-Lie bialgebra} is a pseudo-Lie algebra where $\eta=0$. When $\eta=0$ and $\phi=0$, the definition reduces to that of a Lie bi-algebra.
\end{definition}
\begin{definition}
    We denote by $\mathcal{PL}ieb_n$ the properad of pseudo-Lie bialgebras. Let $p,q\in\mathbb{Z}$. Let $\mathcal{PL}ieb_{p,q}:=\mathcal{PL}ieb_{p+q-2}\{1-p\}$ be the degree shifted pseudo-Lie bialgebra properad where the bracket has degree $1-q$ and the cobracket has degree $1-p$. Further the trace operation has degree $5-2p-3q$ and the constant tri-tensor has degree $p+2q-3$. This properad is quadratic generated by the $\mathbb{S}$-bimodule
    \begin{equation*}
        E(m,n)=
        \begin{cases}
            (\mathbf{1}_0\otimes\mathrm{sgn}_3^{\otimes 5-2p-3q})[-5+2p+3q]=

            & (m,n)=(3,0)
        \end{cases}
    \end{equation*}
    modulo the ideal generated by
    \begin{equation*}
    \mathcal{R}(m,n)=
        \begin{cases}
        % ( 0 , 4 ) CASE 1
        %
,
        & (m,n)=(4,0)
        \end{cases}
\end{equation*}
    Similarly one can define the properads $\mathcal{QL}ieb_{c,d}$ and $\mathcal{L}ieb_{c,d}$ using appropriate restrictions on the generators and relations above. It has been shown that the properads $\mathcal{L}ieb$, $\mathcal{QL}ieb$ and $\mathcal{PL}ieb$ admit minimal resolutions. We denote the respective minimal resolution by $\mathcal{H}olieb_{c,d}$ \cite{V}, $\mathcal{QH}olieb_{c,d}$ \cite{G,Kr} and $\mathcal{PH}olieb_{c,d}$ \cite{G}. Each of them is a free properad generated by $\mathbb{S}$-modules where in each arity $(m,n)$ they are either zero or the span of one element $o_{m,n}=$\begin{tikzpicture}
        [shorten >=1pt,node distance=1.2cm,auto,baseline=-0.1cm,scale=0.6]
        \node[black] (c) at (0,0) {};
        \node[] (u1) at (-1.4,0.6) {$\scriptstyle 1$}
            edge [-] (c);
        \node[] (u2) at (-0.8,0.6) {}
            edge [-] (c);
        \node[] (u3) at (0,0.5) {$\scriptstyle \cdots$};
        \node[] (u4) at (0.8,0.6) {}
            edge [-] (c);
        \node[] (u5) at (1.4,0.6) {$\scriptstyle m$}
            edge [-] (c);
        \node[] (l1) at (-1.4,-0.6) {$\scriptstyle 1$}
            edge [-] (c);
        \node[] (l2) at (-0.8,-0.6) {}
            edge [-] (c);
        \node[] (l3) at (0,-0.5) {$\scriptstyle \cdots$};
        \node[] (l4) at (0.8,-0.6) {}
            edge [-] (c);
        \node[] (l5) at (1.4,-0.6) {$\scriptstyle n$}
            edge [-] (c);
    \end{tikzpicture} of degree $1-p(m-1)-q(n-1)$ satisfying $\sigma_m\tau_no_{m,n}=(-1)^{p|\sigma_m|+q|\tau_n|}o_{m,n}$ for all $\sigma_m\in\mathbb{S}_m$ and $\tau_n\in\mathbb{S}_n$. That is
    $\begin{tikzpicture}
        [shorten >=1pt,node distance=1.2cm,auto,baseline=-0.1cm,scale=0.6]
        \node[black] (c) at (0,0) {};
        \node[] (u1) at (-1.4,0.6) {$\scriptstyle \sigma(1)$}
            edge [-] (c);
        \node[] (u2) at (-0.8,0.6) {}
            edge [-] (c);
        \node[] (u3) at (0,0.5) {$\scriptstyle \cdots$};
        \node[] (u4) at (0.8,0.6) {}
            edge [-] (c);
        \node[] (u5) at (1.4,0.6) {$\scriptstyle \sigma(m)$}
            edge [-] (c);
        \node[] (l1) at (-1.4,-0.6) {$\scriptstyle \tau(1)$}
            edge [-] (c);
        \node[] (l2) at (-0.8,-0.6) {}
            edge [-] (c);
        \node[] (l3) at (0,-0.5) {$\scriptstyle \cdots$};
        \node[] (l4) at (0.8,-0.6) {}
            edge [-] (c);
        \node[] (l5) at (1.4,-0.6) {$\scriptstyle \tau(n)$}
            edge [-] (c);
    \end{tikzpicture}
    =(-1)^{p|\sigma|+q|\tau|}
    \begin{tikzpicture}
        [shorten >=1pt,node distance=1.2cm,auto,baseline=-0.1cm,scale=0.6]
        \node[black] (c) at (0,0) {};
        \node[] (u1) at (-1.4,0.6) {$\scriptstyle 1$}
            edge [-] (c);
        \node[] (u2) at (-0.8,0.6) {}
            edge [-] (c);
        \node[] (u3) at (0,0.5) {$\scriptstyle \cdots$};
        \node[] (u4) at (0.8,0.6) {}
            edge [-] (c);
        \node[] (u5) at (1.4,0.6) {$\scriptstyle m$}
            edge [-] (c);
        \node[] (l1) at (-1.4,-0.6) {$\scriptstyle 1$}
            edge [-] (c);
        \node[] (l2) at (-0.8,-0.6) {}
            edge [-] (c);
        \node[] (l3) at (0,-0.5) {$\scriptstyle \cdots$};
        \node[] (l4) at (0.8,-0.6) {}
            edge [-] (c);
        \node[] (l5) at (1.4,-0.6) {$\scriptstyle n$}
            edge [-] (c);
    \end{tikzpicture}$.
    More specifically let $\mathcal{H}olieb_{c,d}$, $\mathcal{QH}olieb_{c,d}$ and $\mathcal{PH}olieb_{c,d}$ be the free properads generated by the $\mathbb{S}$-bimodules $\{E(m,n)\}_{m,n\in\mathbb{N}}$, $\{QE(m,n)\}_{m,n\in\mathbb{N}}$ and $\{PE(m,n)\}_{m,n\in\mathbb{N}}$ defined by
\begin{align*}
    E(m,n)&=\begin{cases}
        \langle o_{m,n}\rangle & \text{if }m+n\geq 3\text{ and }m,n\geq 1\\
        0 & \text{else.}
    \end{cases}\\
    QE(m,n)&=\begin{cases}
       \langle o_{m,n}\rangle & \text{if }m+n\geq 3\text{ and }m\geq 1\\
        0 & \text{else.}
    \end{cases}\\
    PE(m,n)&=\begin{cases}
        \langle o_{m,n}\rangle & \text{if }m+n\geq 3\\
        0 & \text{else.}
    \end{cases}
\end{align*}
The differential $d$ acts on generators by splitting the vertex in two and summing over all possible reattachments of the in- and out-legs:
\begin{equation}\label{eq:holiebdiff)}
d
\resizebox{14mm}{!}{\begin{xy}
 <0mm,0mm>*{\circ};<0mm,0mm>*{}**@{},
 <-0.6mm,0.44mm>*{};<-8mm,5mm>*{}**@{-},
 <-0.4mm,0.7mm>*{};<-4.5mm,5mm>*{}**@{-},
 <0mm,0mm>*{};<-1mm,5mm>*{\ldots}**@{},
 <0.4mm,0.7mm>*{};<4.5mm,5mm>*{}**@{-},
 <0.6mm,0.44mm>*{};<8mm,5mm>*{}**@{-},
   <0mm,0mm>*{};<-8.5mm,5.5mm>*{^1}**@{},
   <0mm,0mm>*{};<-5mm,5.5mm>*{^2}**@{},
   <0mm,0mm>*{};<4.5mm,5.5mm>*{^{m\hspace{-0.5mm}-\hspace{-0.5mm}1}}**@{},
   <0mm,0mm>*{};<9.0mm,5.5mm>*{^m}**@{},
 <-0.6mm,-0.44
 mm>*{};<-8mm,-5mm>*{}**@{-},
 <-0.4mm,-0.7mm>*{};<-4.5mm,-5mm>*{}**@{-},
 <0mm,0mm>*{};<-1mm,-5mm>*{\ldots}**@{},
 <0.4mm,-0.7mm>*{};<4.5mm,-5mm>*{}**@{-},
 <0.6mm,-0.44mm>*{};<8mm,-5mm>*{}**@{-},
   <0mm,0mm>*{};<-8.5mm,-6.9mm>*{^1}**@{},
   <0mm,0mm>*{};<-5mm,-6.9mm>*{^2}**@{},
   <0mm,0mm>*{};<4.5mm,-6.9mm>*{^{n\hspace{-0.5mm}-\hspace{-0.5mm}1}}**@{},
   <0mm,0mm>*{};<9.0mm,-6.9mm>*{^n}**@{},
 \end{xy}}
\ \ = \ \
 \sum_{\substack{[1,\ldots,m]=I_1\sqcup I_2\\
 {|I_1|\geq 0, |I_2|\geq 1}}}
 \sum_{\substack{[1,\ldots,n]=J_1\sqcup J_2 \\
 {|J_1|\geq 1, |J_2|\geq 1} }
}\hspace{0mm}
\pm
\resizebox{22mm}{!}{ \begin{xy}
 <0mm,0mm>*{\circ};<0mm,0mm>*{}**@{},
 <-0.6mm,0.44mm>*{};<-8mm,5mm>*{}**@{-},
 <-0.4mm,0.7mm>*{};<-4.5mm,5mm>*{}**@{-},
 <0mm,0mm>*{};<0mm,5mm>*{\ldots}**@{},
 <0.4mm,0.7mm>*{};<4.5mm,5mm>*{}**@{-},
 <0.6mm,0.44mm>*{};<12.4mm,4.8mm>*{}**@{-},
     <0mm,0mm>*{};<-2mm,7mm>*{\overbrace{\ \ \ \ \ \ \ \ \ \ \ \ }}**@{},
     <0mm,0mm>*{};<-2mm,9mm>*{^{I_1}}**@{},
 <-0.6mm,-0.44mm>*{};<-8mm,-5mm>*{}**@{-},
 <-0.4mm,-0.7mm>*{};<-4.5mm,-5mm>*{}**@{-},
 <0mm,0mm>*{};<-1mm,-5mm>*{\ldots}**@{},
 <0.4mm,-0.7mm>*{};<4.5mm,-5mm>*{}**@{-},
 <0.6mm,-0.44mm>*{};<8mm,-5mm>*{}**@{-},
      <0mm,0mm>*{};<0mm,-7mm>*{\underbrace{\ \ \ \ \ \ \ \ \ \ \ \ \ \ \
      }}**@{},
      <0mm,0mm>*{};<0mm,-10.6mm>*{_{J_1}}**@{},
 <13mm,5mm>*{};<13mm,5mm>*{\circ}**@{},
 <12.6mm,5.44mm>*{};<5mm,10mm>*{}**@{-},
 <12.6mm,5.7mm>*{};<8.5mm,10mm>*{}**@{-},
 <13mm,5mm>*{};<13mm,10mm>*{\ldots}**@{},
 <13.4mm,5.7mm>*{};<16.5mm,10mm>*{}**@{-},
 <13.6mm,5.44mm>*{};<20mm,10mm>*{}**@{-},
      <13mm,5mm>*{};<13mm,12mm>*{\overbrace{\ \ \ \ \ \ \ \ \ \ \ \ \ \ }}**@{},
      <13mm,5mm>*{};<13mm,14mm>*{^{I_2}}**@{},
 <12.4mm,4.3mm>*{};<8mm,0mm>*{}**@{-},
 <12.6mm,4.3mm>*{};<12mm,0mm>*{\ldots}**@{},
 <13.4mm,4.5mm>*{};<16.5mm,0mm>*{}**@{-},
 <13.6mm,4.8mm>*{};<20mm,0mm>*{}**@{-},
     <13mm,5mm>*{};<14.3mm,-2mm>*{\underbrace{\ \ \ \ \ \ \ \ \ \ \ }}**@{},
     <13mm,5mm>*{};<14.3mm,-4.5mm>*{_{J_2}}**@{},
 \end{xy}}.
\end{equation}
The properad $\mathcal{H}olieb_{c,d}$ is generated by corollas with at least one output, one input and where the total amount of in- and outputs is greater than or equal to three. The properad $\mathcal{QH}olieb_{c,d}$ is generated by corollas with at least one output and at least three inputs and outputs combined, while generators $\mathcal{PH}olieb_{c,d}$ only need to have at least three inputs and outputs in total. So while $\mathcal{H}olieb_{c,d}$ is only generated by \textit{generic} corollas, $\mathcal{QH}olieb_{c,d}$ is additionally generated by corollas that are at least trivalent sources, and $\mathcal{PH}olieb_{c,d}$ is additionally generated by corollas that are at least trivalent sourced or targets.
\end{definition}
\subsection{Properad extensions}
Let $\mathcal{H}olieb^\bullet_{c,d}$ be the free properad generated by the operation $o_{m,n}$ in each arity. Set $\mathcal{H}olieb_{c,d}^+$, $\mathcal{QH}olieb_{c,d}^+$ and $\mathcal{PH}olieb_{c,d}^+$ to be the free extensions of $\mathcal{H}olieb_{c,d}$, $\mathcal{QH}olieb_{c,d}$ and $\mathcal{PH}olieb_{c,d}$ by adding the following generators:
\begin{align*}
    \mathcal{H}olieb_{c,d}^+: &\quad
    \begin{tikzpicture}
        [shorten >=1pt,node distance=1.2cm,auto,baseline=-0.1cm,scale=0.8]
        \node[black] (c) at (0,0) {};
        \node[] (u1) at (0,0.6) {}
            edge [-] (c);
        \node[] (l1) at (0,-0.6) {}
            edge [-] (c);
    \end{tikzpicture}\\
    \mathcal{QH}olieb_{c,d}^+: &\quad
    \begin{tikzpicture}
        [shorten >=1pt,node distance=1.2cm,auto,baseline=-0.1cm,scale=0.8]
        \node[black] (c) at (0,0) {};
        \node[] (u1) at (0,0.6) {}
            edge [-] (c);
        \node[] (l1) at (0,-0.6) {}
            edge [-] (c);
    \end{tikzpicture}\ \ ,
    \begin{tikzpicture}
        [shorten >=1pt,node distance=1.2cm,auto,baseline=-0.1cm,scale=0.8]
        \node[black] (c) at (0,0) {};
        \node[] (u1) at (-0.25,0.6) {}
            edge [-] (c);
        \node[] (u2) at (0.25,0.6) {}
            edge [-] (c);
    \end{tikzpicture}, \ 
    \begin{tikzpicture}
        [shorten >=1pt,node distance=1.2cm,auto,baseline=-0.1cm,scale=0.8]
        \node[black] (c) at (0,0) {};
        \node[] (u1) at (0,0.6) {}
            edge [-] (c);
    \end{tikzpicture}\\
    \mathcal{PH}olieb_{c,d}^+: &\quad
    \begin{tikzpicture}
        [shorten >=1pt,node distance=1.2cm,auto,baseline=-0.1cm,scale=0.8]
        \node[black] (c) at (0,0) {};
        \node[] (u1) at (0,0.6) {}
            edge [-] (c);
        \node[] (l1) at (0,-0.6) {}
            edge [-] (c);
    \end{tikzpicture}\ \ ,
    \begin{tikzpicture}
        [shorten >=1pt,node distance=1.2cm,auto,baseline=-0.1cm,scale=0.8]
        \node[black] (c) at (0,0) {};
        \node[] (u1) at (-0.25,0.6) {}
            edge [-] (c);
        \node[] (u2) at (0.25,0.6) {}
            edge [-] (c);
    \end{tikzpicture}\ , \ 
    \begin{tikzpicture}
        [shorten >=1pt,node distance=1.2cm,auto,baseline=-0.1cm,scale=0.8]
        \node[black] (c) at (0,0) {};
        \node[] (u1) at (0,0.6) {}
            edge [-] (c);
    \end{tikzpicture}\ \ ,
    \begin{tikzpicture}
        [shorten >=1pt,node distance=1.2cm,auto,baseline=-0.1cm,scale=0.8]
        \node[black] (c) at (0,0) {};
        \node[] (u1) at (-0.25,-0.6) {}
            edge [-] (c);
        \node[] (u2) at (0.25,-0.6) {}
            edge [-] (c);
    \end{tikzpicture}\  , \
    \begin{tikzpicture}
        [shorten >=1pt,node distance=1.2cm,auto,baseline=-0.1cm,scale=0.8]
        \node[black] (c) at (0,0) {};
        \node[] (u1) at (0,-0.6) {}
            edge [-] (c);
    \end{tikzpicture}\ , \ 
    \begin{tikzpicture}
        [shorten >=1pt,node distance=1.2cm,auto,baseline=-0.1cm,scale=0.8]
        \node[black] (c) at (0,0) {};
    \end{tikzpicture}\\
\end{align*}
The differential extends naturally in each case. We immediately remark that $\mathcal{H}olieb_{c,d}^\bullet=\mathcal{PH}olieb_{c,d}^+$.
We have projections from $\mathcal{H}olieb_{c,d}^\bullet$ into $\mathcal{H}olieb_{c,d}$, $\mathcal{H}olqieb_{c,d}$ and $\mathcal{PH}olieb_{c,d}$, which factor through their corresponding extended properad by similar projections. That is
\begin{align*}
        \mathcal{H}olieb_{c,d}^\bullet \twoheadrightarrow \mathcal{H}olieb_{c,d}^+ \twoheadrightarrow \mathcal{H}olieb_{c,d}\\
        \mathcal{H}olieb_{c,d}^\bullet \twoheadrightarrow \mathcal{QH}olieb_{c,d}^+ \twoheadrightarrow \mathcal{QH}olieb_{c,d}\\
        \mathcal{H}olieb_{c,d}^\bullet \twoheadrightarrow \mathcal{PH}olieb_{c,d}^+ \twoheadrightarrow \mathcal{PH}olieb_{c,d}
    \end{align*}

When studying deformation complexes of free properads, we always work with their genus completion as in \cite{MW1} by default.

\subsection{Wheelification}
Combinatorially, a properad $\mathcal{P}$ is an $\mathbb{S}$-module $\{\mathcal{P}(m,n)\}_{m,n\in\mathbb{N}}$ together with a horizontal composition
\begin{equation*}
    \otimes:\mathcal{P}(m_1,n_1)\otimes...\otimes \mathcal{P}(m_k,n_k)\rightarrow P(m_1+...+m_k,n_1+...+n_k)
\end{equation*}
and vertical composition \cite{V}:
\begin{equation*}
    \circ: \mathcal{P}(m,n)\otimes \mathcal{P}(n,k)\rightarrow \mathcal{P}(m,k)
\end{equation*}
satisfying certain axioms.
A \textit{wheeled properad} is a properad equipped with additional contraction maps
\begin{equation*}
    \xi^i_j:\mathcal{P}(m,n)\rightarrow \mathcal{P}(m-1,n-1)\quad\text{ where }\quad 1\leq m,n,\, 1\leq i\leq m \text{ and } 1\leq j\leq n
\end{equation*}
satisfying some further axioms (see \cite{M2,MMS} for full details).
Consider the endomorphism properad $\mathcal{E}nd_V=\{Hom(V^m,V^n)\}_{m,n\in\mathbb{N}}$ for some finite dimensional vector space $V$. This properad has a natural contraction map $\xi_{i,j}:Hom(V^m,V^n)\rightarrow Hom(V^{m-1},V^{n-1})$ induced by the standard trace map $Hom(V,V)\rightarrow\mathbb{K}$.
Wheeled properads form the category $Properad^\circlearrowleft$, and there is a natural forgetful functor $\Box:Properad^\circlearrowleft\rightarrow Properad$. This functor has a left adjoint functor $(-)^\circlearrowleft:Properad\rightarrow Properad^\circlearrowleft$ called the \textit{wheelification functor}.
Suppose that $\mathcal{F}$ is a free properad. The elements of $\mathcal{F}$ are the operations from the generating set composed freely. Any such element can be viewed as a directed graph with no closed paths, having input- and output-hairs attached to the vertices. The elements of wheelification $\mathcal{F}^\circlearrowleft$ of $\mathcal{F}$ can be viewed in the same way, with the difference that the graphs can contain closed paths. Sometimes we write $\mathcal{P}^\uparrow$ for a properad $\mathcal{P}$ that is not wheeled.
\subsection{Derivation complexes}
\begin{definition}
    Let $f:\mathcal{P}\rightarrow\mathcal{Q}$ be a morphism of dg properads. A \textit{derivation} of $f$ is a morphism of graded $\mathcal{S}$-modules $D:\mathcal{P}\rightarrow\mathcal{Q}$ of degree $-1$ such that 
    \begin{equation*}
        D(x\circ y)=D(x)\circ f(y) - (-1)^{|x|}f(x)\circ D(y).
    \end{equation*}
    When the map $f$ is implied we write the space of derivations by $Der(\mathcal{P},\mathcal{Q})$. The space of deformations is the space of derivations shiftet by $-1$, i.e $Def(\mathcal{P},\mathcal{Q}):=Der(\mathcal{P},\mathcal{Q})[-1]$. There is a natural differential $\partial$ on the deformation complex defined by
    \begin{equation*}
        \partial(D):=d_b\circ D + f \circ D + (-1)^{|f|}D\circ f
    \end{equation*}
    for $D\in Def(\mathcal{P},\mathcal{Q})$. We will only work with deformation complexes from free properads making the study of the deformation complex easier since $Der(\mathcal{F}ree(E),\mathcal{Q})\cong Hom(E,\mathcal{Q})[-1]$. That is any derivation is uniquely determined by where it maps the generators. 
\end{definition}
\begin{definition}
    Let $Der^\bullet(\mathcal{H}olieb_{p,q}^\uparrow)$ be the complex of derivations of the projection map $\mathcal{H}olieb_{p,q}^\bullet\rightarrow\mathcal{H}olieb_{p,q}$. Then
    \begin{equation*}
        Der^\bullet(\mathcal{H}olieb_{p,q}^\uparrow)\cong\prod_{m,n\geq0}(\mathcal{H}olieb_{p,q}^\uparrow(m,n)\otimes sgn_m^{\otimes |p|}\otimes sgn_n^{\otimes |q|})^{\mathbb{S}_m\times\mathbb{S}_n}[1+p(1-m)+q(1-n)]
    \end{equation*}
    where $\mathcal{H}olieb_{p,q}^\uparrow(m,n)$ is the vector space of elements with $m$ outputs and $n$ inputs in $Der^\bullet(\mathcal{H}olieb_{p,q}^\uparrow)$.
    Similarly let $Der(\mathcal{H}olieb_{p,q}^{\uparrow})$ be the derivation complex of the projection map $\mathcal{H}olieb_{p,q}^+\rightarrow\mathcal{H}olieb_{p,q}$. Then
    \begin{equation*}
        Der(\mathcal{H}olieb_{p,q}^\uparrow)\cong\prod_{m,n\geq1}(\mathcal{H}olieb_{p,q}^\uparrow(m,n)\otimes sgn_m^{\otimes |p|}\otimes sgn_n^{\otimes |q|})^{\mathbb{S}_m\times\mathbb{S}_n}[1+p(1-m)+q(1-n)].
    \end{equation*}
    We remark that $Der^\bullet(\mathcal{H}olieb_{p,q}^\uparrow)=Der(\mathcal{H}olieb_{p,q}^\uparrow)$ since $\mathcal{H}olieb_{p,q}^\uparrow(m,n)=0$ if $m=0$ or $n=0$, and that the elements of $\mathcal{H}olieb_{p,q}^\uparrow(m,n)$ can be represented as hairy graphs with no closed paths, and hence contain at least one source and one target vertex. These vertices must have at least one incoming hair and one outgoing hair respectively.
    Similarly define $Der(\mathcal{QH}olieb_{p,q}^\uparrow)$ and $Der(\mathcal{PH}olieb_{p,q}^\uparrow)$ with the associated quasi Lie-bialgebra and pseudo Lie-bialgebra properads. We gather
    \begin{align*}
        Der(\mathcal{QH}olieb^\uparrow_{p,q})&\cong\prod_{m,\geq 1,n\geq 0}(\mathcal{QH}olieb^\uparrow_{p,q}(m,n)\otimes sgn_m^{\otimes |p|}\otimes sgn_n^{\otimes |q|})^{\mathbb{S}_m\times\mathbb{S}_n}[1+p(1-m)+q(1-n)]\\
        Der(\mathcal{PH}olieb^\uparrow_{p,q})&\cong\prod_{m,n\geq 0}(\mathcal{PH}olieb^\uparrow_{p,q}(m,n)\otimes sgn_m^{\otimes |p|}\otimes sgn_n^{\otimes |q|})^{\mathbb{S}_m\times\mathbb{S}_n}[1+p(1-m)+q(1-n)].
    \end{align*}
The differential $\partial$ on the derivation complexes is given by vertex splitting in the respective properad, with the addition that one also attaches $(m,n)$ corollas to every hair for all integers $m,n$ compatible with the valency conditions of the given properad.
\begin{equation}
\partial \Gamma =
 d\Gamma
  \pm
  \sum_{m,n}
 \resizebox{14mm}{!}{\begin{xy}
 <0mm,0mm>*{\bu};<0mm,0mm>*{}**@{},
 <-0.6mm,0.44mm>*{};<-8mm,5mm>*{}**@{-},
 <-0.4mm,0.7mm>*{};<-4.5mm,5mm>*{}**@{-},
 <0mm,0mm>*{};<-1mm,5mm>*{\ldots}**@{},
 <0.4mm,0.7mm>*{};<4.5mm,5mm>*{}**@{-},
 <0.6mm,0.44mm>*{};<10mm,6mm>*{}**@{-},
   <0mm,0mm>*{};<12.0mm,7.5mm>*{\Gamma}**@{},
 <-0.6mm,-0.44mm>*{};<-8mm,-5mm>*{}**@{-},
 <-0.4mm,-0.7mm>*{};<-4.5mm,-5mm>*{}**@{-},
 <0mm,0mm>*{};<-1mm,-5mm>*{\ldots}**@{},
 <0.4mm,-0.7mm>*{};<4.5mm,-5mm>*{}**@{-},
 <0.6mm,-0.44mm>*{};<8mm,-5mm>*{}**@{-},
 \end{xy}}
   \mp
  \sum_{m,n}
 \resizebox{14mm}{!}{\begin{xy}
 <0mm,0mm>*{\bu};<0mm,0mm>*{}**@{},
 <-0.6mm,0.44mm>*{};<-8mm,5mm>*{}**@{-},
 <-0.4mm,0.7mm>*{};<-4.5mm,5mm>*{}**@{-},
 <0mm,0mm>*{};<-1mm,5mm>*{\ldots}**@{},
 <0.4mm,0.7mm>*{};<4.5mm,5mm>*{}**@{-},
 <0.6mm,0.44mm>*{};<-10mm,-6mm>*{}**@{-},
   <0mm,0mm>*{};<-12.0mm,-7.5mm>*{\Gamma}**@{},
 <-0.6mm,-0.44mm>*{};<8mm,5mm>*{}**@{-},
 <-0.4mm,-0.7mm>*{};<-4.5mm,-5mm>*{}**@{-},
 <0mm,0mm>*{};<-1mm,-5mm>*{\ldots}**@{},
 <0.4mm,-0.7mm>*{};<4.5mm,-5mm>*{}**@{-},
 <0.6mm,-0.44mm>*{};<8mm,-5mm>*{}**@{-},
 \end{xy}}
 \end{equation}
The sign rule for the formula is described in \cite{MW1}.
\end{definition}
\begin{definition}
    Let $Der^\bullet(\mathcal{H}olieb_{p,q}^\circlearrowleft)$ be the complex of derivations of the projection map $\mathcal{H}olieb_{p,q}^{\circlearrowleft,\bullet}\rightarrow\mathcal{H}olieb^{\circlearrowleft}_{p,q}$, and similarly let $Der(\mathcal{H}olieb_{p,q}^\circlearrowleft)$ be the derivation complex of the projection $\mathcal{H}olieb_{p,q}^{\circlearrowleft,+}\rightarrow\mathcal{H}olieb^{\circlearrowleft}_{p,q}$. Then
    \begin{align*}
        Der^\bullet(\mathcal{H}olieb_{p,q}^\circlearrowleft)&\cong\prod_{m,n\geq0}(\mathcal{H}olieb_{p,q}^\circlearrowleft(m,n)\otimes sgn_m^{\otimes |p|}\otimes sgn_n^{\otimes |q|})^{\mathbb{S}_m\times\mathbb{S}_n}[1+p(1-m)+q(1-n)]\\
        Der(\mathcal{H}olieb_{p,q}^\circlearrowleft)&\cong\prod_{m,n\geq1}(\mathcal{H}olieb_{p,q}^\circlearrowleft(m,n)\otimes sgn_m^{\otimes |p|}\otimes sgn_n^{\otimes |q|})^{\mathbb{S}_m\times\mathbb{S}_n}[1+p(1-m)+q(1-n)]
    \end{align*}
    where $\mathcal{H}olieb_{p,q}^\circlearrowleft(m,n)$ is the subspace of graphs in $\mathcal{H}olieb_{p,q}^{\bullet,\circlearrowleft}$ of graphs with $m$ outputs and $n$ inputs. The differential is the same as in the unwheeled derivation complexes.
    We similarly define the analogous derivation complexes in the case of the quasi Lie-bialgebra for which we have
    \begin{align*}
        Der^\bullet(\mathcal{QH}olieb_{p,q}^\circlearrowleft)&\cong\prod_{m,n\geq0}(\mathcal{QH}olieb_{p,q}^\circlearrowleft(m,n)\otimes sgn_m^{\otimes |p|}\otimes sgn_n^{\otimes |q|})^{\mathbb{S}_m\times\mathbb{S}_n}[1+p(1-m)+q(1-n)]\\
        Der(\mathcal{QH}olieb_{p,q}^\circlearrowleft)&\cong\prod_{m\geq1,\geq0}(\mathcal{QH}olieb_{p,q}^\circlearrowleft(m,n)\otimes sgn_m^{\otimes |p|}\otimes sgn_n^{\otimes |q|})^{\mathbb{S}_m\times\mathbb{S}_n}[1+p(1-m)+q(1-n)]
    \end{align*}
    Since $\mathcal{H}olieb_{p,q}^\bullet=\mathcal{PH}olieb_{p,q}^+$, the similarly defined derivation complexes $Der^\bullet(\mathcal{H}olieb_{p,q}^\circlearrowleft)$ and $Der(\mathcal{H}olieb_{p,q}^\circlearrowleft)$ are equal, and can be described as
    \begin{equation*}
        Der(\mathcal{PH}olieb_{p,q}^\circlearrowleft)\cong\prod_{m,n\geq0}(\mathcal{PH}olieb_{p,q}^\circlearrowleft(m,n)\otimes sgn_m^{\otimes |p|}\otimes sgn_n^{\otimes |q|})^{\mathbb{S}_m\times\mathbb{S}_n}[1+p(1-m)+q(1-n)].
    \end{equation*}
\end{definition}

\section{Graph complexes}
\subsection{The directed Kontsevich's graph complex}
Let $\overline{\mathrm{V}}_v\overline{\mathrm{E}}_e\mathrm{cgra}$ be the set of connected directed graphs with $e$ edges and $v$ vertices. The edges and vertices are labeled from $1$ to $e$ and $1$ to $v$ respectively. Both tadpoles and multiple edges are allowed in the graphs. Let $k\in\mathbb{Z}$ and let $\mathrm{V}_v\mathrm{E}_e\mathrm{GC}_k$ be the graded $\mathbb{K}$ vector space concentrated in degree $(v-1)k+(1-k)e$ generated by $\overline{\mathrm{V}}_v\overline{\mathrm{E}}_e\mathrm{cgra}$. There is a natural right action of $\mathbb{S}_v\times\mathbb{S}_e$ on the vector space permuting the labels of vertices and edges respectively. The full connected directed graph complex $(\mathsf{cfdGC}_k,d)$ is a chain complex where
\begin{equation*}
        \mathsf{cfdGC}_k:=
        \begin{cases}
        \prod_{e,v} \Big(\mathrm{\Bar{V}}_v\mathrm{\Bar{E}}_e\mathrm{GC}_k \otimes \mathrm{sgn}_e \Big)_{\mathbb{S}_v\times \mathbb{S}_e} & \text{ for }k \text{ even,}\\
        \prod_{e,v} \Big(\mathrm{\Bar{V}}_v\mathrm{\Bar{E}}_e\mathrm{GC}_k \otimes \mathrm{sgn}_v\Big)_{\mathbb{S}_v\times \mathbb{S}_e} & \text{ for }k \text{ odd.}
        \end{cases}
\end{equation*}
The subscript denotes taking the space of coinvariants under the group actions. The differential of degree 1 is defined on graphs $\Gamma$ as
\begin{equation*}
    d(\Gamma):=\delta(\Gamma)-\delta'(\Gamma)-\delta''(\Gamma)=\sum_{x\in V(\Gamma)}\delta_x(\Gamma)-\delta_x'(\Gamma)-\delta_x''(\Gamma)
\end{equation*}
where $V(\Gamma)$ denotes the set of vertices of $\Gamma$. The term $\delta_x(\Gamma)$ denotes the sum of graphs obtained by replacing the vertex $x$ with two vertices and a new edge between these two vertices. The summation is over all possible ways of reattaching the edges connected to $x$ in $\Gamma$. The term $\delta'_x(\Gamma)$ is the graph where a new univalent vertex has been added to the graph connected via an incoming edge to $x$. Similarly $\delta''_x(\Gamma)$ is the graph where a new univalent vertex is connected via an outgoing edge to $x$. Given a representative of $\Gamma$, the signs of the resulting graphs is determined by the labelling its vertices and edges. The new edge is labelled $e+1$, while the source vertex of this edge is labelled by $x$ and the target vertex by $v+1$. Note that no new univalent vertices are created by the differential.
The term full refers to that we have no restrictions on the graphs we consider in the complex with respects to the valency of vertices in the graphs.
The \textit{loop number} of a graph is the number $b=e-v+1$ where $e$ is the number of edges and $v$ the number of vertices. The loop number is invariant under the differential.
Let $\dGC_k$ be the subcomplex of $\mathsf{fcdGC}_k$ of graphs with no univalent nor passing vertices (bivalent vertices that have an incoming and one outgoing edge).
\begin{proposition}
    The inclusion $\dGC_k\hookrightarrow\mathsf{cfdGC}_k$ is a quasi-isomorphism.
\end{proposition}
\begin{proof}
    See \cite{W1}
\end{proof}
Up to the parity of $k$, the complexes $\dGC_k$ are quasi-isomorphic up to degree shifts. Hence if one wishes to study the cohomology of $\dGC_k$, it is sufficient to study the complexes $\dGC_2$ and $\dGC_3$.
\begin{proposition}
    The cohomology of $\dGC_2$ vanish in negative degrees, i.e $H^l(\dGC_2)=0$ for $l<0$.
    The cohomlogy of $\dGC_3$ vanish in degrees higher than $-4$, i.e $H^l(\dGC_3)=0$ for $l>-4$.
\end{proposition}
\begin{proof}
    See \cite{F}.
\end{proof}
Thomas Willwacher proved the following important result in \cite{W1}.
\begin{proposition}
    The zeroth cohomology group of $\dGC_2$ is isomorphic to the Grothendieck-Teichmüller Lie-algebra $\mathfrak{grt}$.
\end{proposition}
\subsection{Subcomplexes of $\dGC_k$}

%Notable/Important subcomplexes
%- Oriented graph complex
%Sourced, targeted, sourced and targeted, sourced or targeted
%
There is a family of subcomplexes of $\dGC_k$:
\begin{equation*}
    \begin{tikzcd}
        & & \dGC_k^s \arrow[rd, hook] & &\\
        \OGC_k \arrow[r, hook] & \dGC_k^{st} \arrow[ru, hook] \arrow[rd, hook] & & \dGC^{s+t}_k \arrow[r,hook] & \dGC_k\\
        & & \dGC_k^t \arrow[ru, hook] & &
    \end{tikzcd}
\end{equation*}
where
\begin{itemize}
    \item $\dGC_k^{s+t}$ is the \textit{sourced or targeted graph complex} generated by graphs with at least one source vertex or one target vertex.
    \item $\dGC_k^s$ is the \textit{sourced graph complex} generated by graphs with at least one source vertex.
    \item $\dGC_k^t$ is the \textit{targeted graph complex} generated by graphs with at least one target vertex.
    \item $\dGC_k^{st}$ is the \textit{sourced and targeted graph complex} generated by graphs with at least one source and one target vertex.
    \item $\OGC_k$ is the \textit{oriented graph complex} generated by graphs not containing any closed path.
\end{itemize}
The cohomology of these complexes have been studied in several papers and have been related to each other. It is an easy observation that $\dGC^s_k\cong\dGC_k^t$ by reversing all the direction of the edges in a graph. M. Zivkovic showed in \cite{Z2} that $H^l(\OGC_k)\cong H^l(\dGC^s)\cong H^l(\dGC^t)$ by explicit quasi-isomorphisms. T. Willwacher showed that $H^l(\dGC_k)\cong H^l(\OGC_{k+1})$ \cite{W2}, and M. Zivkovic later found explicit quasi-isomorphisms from $\dGC_k$ to $\oGC_{k+1},\ \dGC_{k+1}^s$ and $\dGC_{k+1}^t$.
The complexes $\dGC^{st}$ and $\dGC^{s+t}$ have been studied in \cite{Z3} and \cite{M3}, and are best understood through two short exact sequences.
\begin{equation*}
        \xymatrix{
        0 \ar[r] & \dGC^{st}_k \ar[r] & \dGC^s_k\oplus\dGC^t_k \ar[r] & \dGC^{s+t}_k \ar[r] & 0\\
            & \Gamma \ar@{|->}[r] & (\Gamma,\Gamma)\\
            & & (\Gamma_1,\Gamma_2) \ar@{|->}[r] & \Gamma_1-\Gamma_2 }
\end{equation*}
\begin{equation*}
        \xymatrix{0 \ar[r] & \dGC^{s+t}_k \ar[r] & \dGC_k \ar[r] & \dGC^{\circlearrowleft}_k \ar[r] & 0}
\end{equation*}
where $\dGC_k^\circlearrowleft$ is the quotient complex $\dGC_k/\dGC_k^{s+t}$ of graphs with no sources nor targets. By considering a decomposition of these complexes over loop number, one can find certain conditions depending on $k$ and degrees of graphs for when these complexes are acyclic (see \cite{Z2} for the full statement.) The most interesting case occurs when $k=3$ where one finds that $H^l(\dGC_3^{s+t})=0$ for $l\leq 1$. This in turn implies that $H^0(\dGC^{st}_3)=H^0(\dGC^s_3)\oplus H^0(\dGC_3^t)\cong\mathfrak{grt}\oplus\mathfrak{grt}$, and $H^0(\dGC^\circlearrowleft_3)=H^0(\dGC_3)$.

\subsection{The bi-weighted graph complex}
In \cite{F}, we defined the \textit{bi-weighted graph complex} $\fWGC_k$ as a tool to easier compute the cohomology of $Der(\mathcal{H}olieb_{c,d}^\circlearrowleft)$. Here we follow the same idea and define two new complexes $\fqWGC_k$ and $\fpWGC_k$ to study the derivations of the quasi- and pseudo-Lie bialgebra properads. We refer to our previous paper for full details.
\begin{definition}
    Consider a directed graph $\Gamma\in\dGC_k$ (which might or might not be zero due to symmetries) and let $x$ be a vertex of $\Gamma$. A \textit{bi-weight} on $x$ is a pair of non-negative integers $(w_x^{out},w_x^{in})$ such that $w_x^{out}+|x|_{out}\geq 1$, $w_x^{in}+|x|_{in}\geq 1$, and $w_x^{out}+w_x^{in}+|x|_{out}+|x|_{in}\geq 3$. A \textit{bi-weighted graph} is a graph $\Gamma\in\dGC_k$ with a bi-weight on each vertex. Let $\fwGC_k$ be the vector space spanned by all bi-weighted graphs. We usually write a bi-weighted vertex together with its bi-weights as $\begin{tikzpicture}[baseline={([yshift=-.5ex]current bounding box.center)}]
    \node[biw] (v) at (0,0) {$\scriptstyle{w_x^{out}}$\nodepart{lower}$\scriptstyle{w_x^{in}}$};
\end{tikzpicture}$.
    We similarly define a \textit{quasi bi-weight} $(w_x^{out},w_x^{in})_q$ and a \textit{pseudo bi-weight} $(w_x^{out},w_x^{in})_p$ as a pair of non-positive integers satisfying
\begin{equation*}
    \begin{tabular}{|c|c|}
        \hline
        $(w_x^{out},w_x^{in})_q$ & $(w_x^{out},w_x^{in})_p$ \\
        \hline
        $\begin{array}{l}
             w_x^{out}+|x|_{out}\geq 1  \\
             w_x^{in}+|x|_{in}\geq 0  \\
             w_x^{out}+w_x^{in}+|x|_{out}+|x|_{in}\geq 3
        \end{array}$ & 
        $\begin{array}{l}
             w_x^{out}+|x|_{out}\geq 0  \\
             w_x^{in}+|x|_{in}\geq 0  \\
             w_x^{out}+w_x^{in}+|x|_{out}+|x|_{in}\geq 3
        \end{array}$  \\
        \hline
    \end{tabular}
\end{equation*}
We similarly define a quasi bi-weighted and a pseudo bi-weighted graph. Quasi- and psuedo bi-weighted vertices are represented in the same manner as bi-weighted vertices. Let $\fqwGC_k$ and $\fpwGC_k$ be the vector spaces of quasi- and pseudo bi-weighted graphs.
The differential on all of three complexes is defined similarly as the differential in $\dGC_k$ by splitting vertices and attaching univalent vertices. Hence in all three cases the differential act on a graph $\Gamma$ as
\begin{equation*}
    d(\Gamma):=\delta(\Gamma)-\delta'(\Gamma)-\delta''(\Gamma)=\sum_{x\in V(\Gamma)}\delta_x(\Gamma)-\delta_x'(\Gamma)-\delta_x''(\Gamma).
\end{equation*}
where $V(\Gamma)$ denotes the set of vertices of $\Gamma$. The map $\delta_x(\Gamma)$ denotes vertex splitting of the vertex $x$ similar as in $\dGC$ with the addition that one also sum over all possible re-distributions of the bi-weights of $x$ on the two new vertices. The $\delta'_x$ is the graph where a new univalent vertex has been added to the graph connected via an incoming edge to $x$. The out-weight of $x$ is decreased by one, and one sums over all possible bi-weights on the new univalent vertex. Similarly $\delta''_x$ is a graph where a new univalent vertex is connected via an outgoing edge to $x$ where the in-weight of $x$ is decreased by one and one sums over all possible bi-weights assigned to the new univalent vertex. The differentials differ between the complexes only due to the different conditions regarding the bi-weights. 
Pictorially we can represent $\delta_x(\Gamma)-\delta_x'(\Gamma)-\delta_x''(\Gamma)$ in all of these three complexes as
\begin{equation*}
    d_x\Big(\ 
    \begin{tikzpicture}[shorten >=1pt,>=angle 90,baseline={([yshift=-.5ex]current bounding box.center)}]
        \node[biw] (v0) at (0,0) {$m$\nodepart{lower}$n$};
        \node[invisible] (v1) at (-0.4,0.8) {}
            edge [<-] (v0);
        \node[] at (0,0.6) {$\scriptstyle\cdots$};
        %\node[invisible] (v2) at (0,0.9) {};
        %    edge [<-] (v0);
        \node[invisible] (v3) at (0.4,0.8) {}
            edge [<-] (v0);
        \node[invisible] (v4) at (-0.4,-0.8) {}
            edge [->] (v0);
        \node[] at (0,-0.6) {$\scriptstyle\cdots$};
        %\node[invisible] (v5) at (0,-0.9) {}
        %    edge [->] (v0);
        \node[invisible] (v6) at (0.4,-0.8) {}
            edge [->] (v0);
    \end{tikzpicture}\ \Big) \ 
    = \ \sum_{\substack{\scriptstyle{m=m_1+m_2}\\ \scriptstyle{n=n_1+n_2}}}
    \begin{tikzpicture}[shorten >=1pt,>=angle 90,baseline={([yshift=-.5ex]current bounding box.center)}]
        \node[ellipse,
            draw = black,
            minimum width = 3cm, 
            minimum height = 1.8cm,
            dotted] (e) at (0,0) {};
        \node[biw] (vLeft) at (-0.8,0) {$m_1$\nodepart{lower}$n_1$};
        \node[biw] (vRight) at (0.8,0) {$m_2$\nodepart{lower}$n_2$}
            edge [<-] (vLeft);
        \node[invisible] (v1) at (-0.6,1.4) {}
            edge [<-] (e);
        \node[] at (0,1.2) {$\cdots$};
        %\node[invisible] (v2) at (0,1.5) {}
        %    edge [<-] (e);
        \node[invisible] (v3) at (0.6,1.4) {}
            edge [<-] (e);
        \node[invisible] (v4) at (-0.6,-1.4) {}
            edge [->] (e);
        \node[] at (0,-1.2) {$\cdots$};
        %\node[invisible] (v5) at (0,-1.5) {}
        %    edge [->] (e);
        \node[invisible] (v6) at (0.6,-1.4) {}
            edge [->] (e);
    \end{tikzpicture}
    \ -\ \sum_{\substack{i\geq 1,j\geq 0\\
    i+j\geq 2}}\ 
    \begin{tikzpicture}[shorten >=1pt,>=angle 90,baseline={(4ex,-0.5ex)}]
        \node[biw] (v0) at (0,0) {$\scriptstyle m-1$\nodepart{lower}$n$};
        \node[biw] (new) at (1.5,0.7) {$i$\nodepart{lower}$j$}
            edge [<-] (v0);
        \node[invisible] (v1) at (-0.4,0.8) {}
            edge [<-] (v0);
        \node[] at (0,0.6) {$\scriptstyle\cdots$};
        %\node[invisible] (v2) at (0,0.9) {}
        %    edge [<-] (v0);
        \node[invisible] (v3) at (0.4,0.8) {}
            edge [<-] (v0);
        \node[invisible] (v4) at (-0.4,-0.8) {}
            edge [->] (v0);
        \node[] at (0,-0.6) {$\scriptstyle\cdots$};
        %\node[invisible] (v5) at (0,-0.9) {}
        %    edge [->] (v0);
        \node[invisible] (v6) at (0.4,-0.8) {}
            edge [->] (v0);
    \end{tikzpicture}
    \ -\ \sum_{\substack{i\geq 0,j\geq 1\\
    i+j\geq 2}}\ 
    \begin{tikzpicture}[shorten >=1pt,>=angle 90,baseline={(4ex,-0.5ex)}]
        \node[biw] (v0) at (0,0) {$\scriptstyle m$\nodepart{lower}$\scriptstyle{n-1}$};
        \node[biw] (new) at (1.5,-0.7) {$i$\nodepart{lower}$j$}
            edge [->] (v0);
        \node[invisible] (v1) at (-0.4,0.8) {}
            edge [<-] (v0);
        \node[] at (0,0.6) {$\scriptstyle\cdots$};
        %\node[invisible] (v2) at (0,0.9) {}
        %    edge [<-] (v0);
        \node[invisible] (v3) at (0.4,0.8) {}
            edge [<-] (v0);
        \node[invisible] (v4) at (-0.4,-0.8) {}
            edge [->] (v0);
        \node[] at (0,-0.6) {$\scriptstyle\cdots$};
        %\node[invisible] (v5) at (0,-0.9) {}
        %    edge [->] (v0);
        \node[invisible] (v6) at (0.4,-0.8) {}
            edge [->] (v0);
    \end{tikzpicture}
\end{equation*}
We will also use the following notation for brecity when describing the splitting term $\delta_x$ of the differential:
\begin{equation*}
    \delta_x\Big(\ 
    \begin{tikzpicture}[shorten >=1pt,>=angle 90,baseline={([yshift=-.5ex]current bounding box.center)}]
        \node[biw] (v0) at (0,0) {$m$\nodepart{lower}$n$};
        \node[invisible] (v1) at (-0.4,0.8) {}
            edge [<-] (v0);
        \node[] at (0,0.6) {$\scriptstyle\cdots$};
        %\node[invisible] (v2) at (0,0.9) {};
        %    edge [<-] (v0);
        \node[invisible] (v3) at (0.4,0.8) {}
            edge [<-] (v0);
        \node[invisible] (v4) at (-0.4,-0.8) {}
            edge [->] (v0);
        \node[] at (0,-0.6) {$\scriptstyle\cdots$};
        %\node[invisible] (v5) at (0,-0.9) {}
        %    edge [->] (v0);
        \node[invisible] (v6) at (0.4,-0.8) {}
            edge [->] (v0);
    \end{tikzpicture}\ \Big) \ 
    = \ \sum_{\substack{\scriptstyle{m=m_1+m_2}\\ \scriptstyle{n=n_1+n_2}}} \Big(\frac{m_1}{n_1},\frac{m_2}{n_2}\Big)_x,\ \text{ where }
    \Big(\frac{m_1}{n_1},\frac{m_2}{n_2}\Big)_x:=
\begin{tikzpicture}[shorten >=1pt,>=angle 90,baseline={([yshift=-.5ex]current bounding box.center)}]
        \node[ellipse,
            draw = black,
            minimum width = 3cm, 
            minimum height = 1.8cm,
            dotted] (e) at (0,0) {};
        \node[biw] (vLeft) at (-0.8,0) {$m_1$\nodepart{lower}$n_1$};
        \node[biw] (vRight) at (0.8,0) {$m_2$\nodepart{lower}$n_2$}
            edge [<-] (vLeft);
        \node[invisible] (v1) at (-0.6,1.4) {}
            edge [<-] (e);
        \node[] at (0,1.2) {$\cdots$};
        %\node[invisible] (v2) at (0,1.5) {}
        %    edge [<-] (e);
        \node[invisible] (v3) at (0.6,1.4) {}
            edge [<-] (e);
        \node[invisible] (v4) at (-0.6,-1.4) {}
            edge [->] (e);
        \node[] at (0,-1.2) {$\cdots$};
        %\node[invisible] (v5) at (0,-1.5) {}
        %    edge [->] (e);
        \node[invisible] (v6) at (0.6,-1.4) {}
            edge [->] (e);
    \end{tikzpicture}.
\end{equation*}
Any invalid assignment of bi-weights to a vertex makes the whole graph zero. Note that this can be different in each of the three complexes:
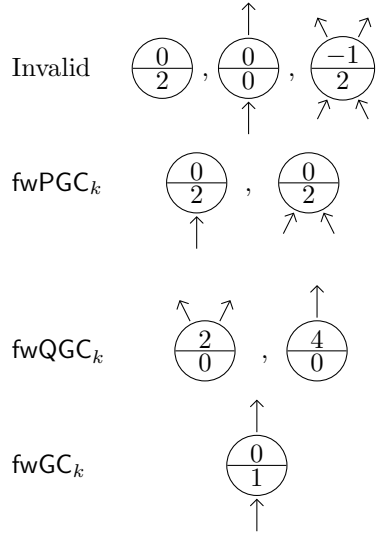
\begin{equation*}
\begin{tabular}{lc}
     Invalid &
        \begin{tikzpicture}[shorten >=1pt,node distance=1.2cm,auto,>=angle 90,baseline=-0.1cm]
                \node[biw] (a) at (0,0) {$0$\nodepart{lower}$2$};
        \end{tikzpicture}\ ,\
        \begin{tikzpicture}[shorten >=1pt,node distance=1.2cm,auto,>=angle 90,baseline=-0.1cm]
            \node[biw] (a) at (0,0) {$0$\nodepart{lower}$0$};
            \node[] (up) at (0,1) {}
                edge [<-] (a);
            \node[] (down) at (0,-1) {}
                edge [->] (a);
        \end{tikzpicture}\ ,\
        \begin{tikzpicture}[shorten >=1pt,node distance=1.2cm,auto,>=angle 90,baseline=-0.1cm]
            \node[biw] (a) at (0,0) {$-1$\nodepart{lower}$2$};
            \node[] (down1) at (-0.4,-0.8) {}
                edge [->] (a);
            %\node[] (down2) at (0,-0.9) {$\scriptstyle \geq 2$};
            \node[] (down3) at (0.4,-0.8) {}
                edge [->] (a);
            \node[] (up1) at (-0.4,0.8) {}
                edge [<-] (a);
            %\node[] (up2) at (0,0.9) {$\scriptstyle \geq 2$};
            \node[] (up3) at (0.4,0.8) {}
                edge [<-] (a);
        \end{tikzpicture}
        \\
     $\fpwGC_k$ & 
     \begin{tikzpicture}[shorten >=1pt,node distance=1.2cm,auto,>=angle 90,baseline=-0.1cm]
            \node[biw] (a) at (0,0) {$0$\nodepart{lower}$2$};
            \node[] (down) at (0,-1) {}
                edge [->] (a);
    \end{tikzpicture} \ , \
     \begin{tikzpicture}[shorten >=1pt,node distance=1.2cm,auto,>=angle 90,baseline=-0.1cm]
            \node[biw] (a) at (0,0) {$0$\nodepart{lower}$2$};
            \node[] (up1) at (-0.4,-0.8) {}
                edge [->] (a);
            \node[] (up3) at (0.4,-0.8) {}
                edge [->] (a);
    \end{tikzpicture}\\
     $\fqwGC_k$ & 
     \begin{tikzpicture}[shorten >=1pt,node distance=1.2cm,auto,>=angle 90,baseline=-0.1cm]
            \node[biw] (a) at (0,0) {$2$\nodepart{lower}$0$};
            \node[] (up1) at (-0.4,0.8) {}
                edge [<-] (a);
            %\node[] (up2) at (0,0.9) {$\scriptstyle \geq 2$};
            \node[] (up3) at (0.4,0.8) {}
                edge [<-] (a);
        \end{tikzpicture} \ , \ 
        \begin{tikzpicture}[shorten >=1pt,node distance=1.2cm,auto,>=angle 90,baseline=-0.1cm]
            \node[biw] (a) at (0,0) {$4$\nodepart{lower}$0$};
            \node[] (up) at (0,1) {}
                edge [<-] (a);
        \end{tikzpicture}
     \\
     $\fwGC_k$ & 
     \begin{tikzpicture}[shorten >=1pt,node distance=1.2cm,auto,>=angle 90,baseline=-0.1cm]
            \node[biw] (a) at (0,0) {$0$\nodepart{lower}$1$};
            \node[] (up) at (0,1) {}
                edge [<-] (a);
            \node[] (down) at (0,-1) {}
                edge [->] (a);
        \end{tikzpicture}
\end{tabular}
\end{equation*}
We call $\fwGC_k$ the \textit{bi-weighted graph complex}, $\fqwGC_k$ the \textit{quasi bi-weighted graph complex} and $\fpwGC_k$ the \textit{psuedo bi-weighted graph complex}.
\end{definition}
\begin{remark}
    Unlike the differential in $\dGC_k$, the differential does not cancel the creation of new univalent vertices in any of the bi-weighted complexes.
\end{remark}
\begin{definition}
We introduce a notion for bi-weights that will be useful in the remainder of this paper. Let $r$ be a natural number.  We let the symbol $\infty_r$ when used as an in or out-weight denote the sum of graphs
\begin{equation*}
    \begin{tikzpicture}[baseline={([yshift=-.5ex]current bounding box.center)}]
        \node[biw] (v0) at (0,0) {$\infty_r$\nodepart{lower}$n$};
        \node[invisible] (v1) at (-0.4,0.8) {}
            edge [<-] (v0);
        \node[] at (0,0.6) {$\scriptstyle\cdots$};
        %\node[invisible] (v2) at (0,0.9) {};
        %    edge [<-] (v0);
        \node[invisible] (v3) at (0.4,0.8) {}
            edge [<-] (v0);
        \node[invisible] (v4) at (-0.4,-0.8) {}
            edge [->] (v0);
        \node[] at (0,-0.6) {$\scriptstyle\cdots$};
        %\node[invisible] (v5) at (0,-0.9) {}
        %    edge [->] (v0);
        \node[invisible] (v6) at (0.4,-0.8) {}
            edge [->] (v0);
    \end{tikzpicture}  
\ = \ \sum_{i\geq r}\ 
    \begin{tikzpicture}[baseline={([yshift=-.5ex]current bounding box.center)}]
        \node[biw] (v0) at (0,0) {$i$\nodepart{lower}$n$};
        \node[invisible] (v1) at (-0.4,0.8) {}
            edge [<-] (v0);
        \node[] at (0,0.6) {$\scriptstyle\cdots$};
        %\node[invisible] (v2) at (0,0.9) {};
        %    edge [<-] (v0);
        \node[invisible] (v3) at (0.4,0.8) {}
            edge [<-] (v0);
        \node[invisible] (v4) at (-0.4,-0.8) {}
            edge [->] (v0);
        \node[] at (0,-0.6) {$\scriptstyle\cdots$};
        %\node[invisible] (v5) at (0,-0.9) {}
        %    edge [->] (v0);
        \node[invisible] (v6) at (0.4,-0.8) {}
            edge [->] (v0);
    \end{tikzpicture} 
\ \qquad,\qquad\
    \begin{tikzpicture}[baseline={([yshift=-.5ex]current bounding box.center)}]
        \node[biw] (v0) at (0,0) {$m$\nodepart{lower}$\infty_r$};
        \node[invisible] (v1) at (-0.4,0.8) {}
            edge [<-] (v0);
        \node[] at (0,0.6) {$\scriptstyle\cdots$};
        %\node[invisible] (v2) at (0,0.9) {};
        %    edge [<-] (v0);
        \node[invisible] (v3) at (0.4,0.8) {}
            edge [<-] (v0);
        \node[invisible] (v4) at (-0.4,-0.8) {}
            edge [->] (v0);
        \node[] at (0,-0.6) {$\scriptstyle\cdots$};
        %\node[invisible] (v5) at (0,-0.9) {}
        %    edge [->] (v0);
        \node[invisible] (v6) at (0.4,-0.8) {}
            edge [->] (v0);
    \end{tikzpicture}  
\ = \ \sum_{i\geq r} \ 
    \begin{tikzpicture}[baseline={([yshift=-.5ex]current bounding box.center)}]
        \node[biw] (v0) at (0,0) {$m$\nodepart{lower}$i$};
        \node[invisible] (v1) at (-0.4,0.8) {}
            edge [<-] (v0);
        \node[] at (0,0.6) {$\scriptstyle\cdots$};
        %\node[invisible] (v2) at (0,0.9) {};
        %    edge [<-] (v0);
        \node[invisible] (v3) at (0.4,0.8) {}
            edge [<-] (v0);
        \node[invisible] (v4) at (-0.4,-0.8) {}
            edge [->] (v0);
        \node[] at (0,-0.6) {$\scriptstyle\cdots$};
        %\node[invisible] (v5) at (0,-0.9) {}
        %    edge [->] (v0);
        \node[invisible] (v6) at (0.4,-0.8) {}
            edge [->] (v0);
    \end{tikzpicture}
\end{equation*}
which distributes like a tensor over the vertices of the graph when more then one vertex is decorated with these symbols.
\end{definition}
\begin{remark}
Note that vertices $\begin{tikzpicture}[baseline={([yshift=-.5ex]current bounding box.center)}]
\node[biw] (v) at (0,0) {$\scriptstyle{\infty_0}$\nodepart{lower}$\scriptstyle{\infty_0}$};
\end{tikzpicture}$ do not create any new univalent vertices under the action of the differential in any of the three complexes. Further note that the differential with this convention can be described as
\begin{equation*}
    d_x\Big(\ 
    \begin{tikzpicture}[shorten >=1pt,>=angle 90,baseline={([yshift=-.5ex]current bounding box.center)}]
        \node[biw] (v0) at (0,0) {$m$\nodepart{lower}$n$};
        \node[invisible] (v1) at (-0.4,0.8) {}
            edge [<-] (v0);
        \node[] at (0,0.6) {$\scriptstyle\cdots$};
        %\node[invisible] (v2) at (0,0.9) {};
        %    edge [<-] (v0);
        \node[invisible] (v3) at (0.4,0.8) {}
            edge [<-] (v0);
        \node[invisible] (v4) at (-0.4,-0.8) {}
            edge [->] (v0);
        \node[] at (0,-0.6) {$\scriptstyle\cdots$};
        %\node[invisible] (v5) at (0,-0.9) {}
        %    edge [->] (v0);
        \node[invisible] (v6) at (0.4,-0.8) {}
            edge [->] (v0);
    \end{tikzpicture}\ \Big) \ 
    = \ \sum_{\substack{\scriptstyle{m=m_1+m_2}\\ \scriptstyle{n=n_1+n_2}}}
    \begin{tikzpicture}[shorten >=1pt,>=angle 90,baseline={([yshift=-.5ex]current bounding box.center)}]
        \node[ellipse,
            draw = black,
            minimum width = 3cm, 
            minimum height = 1.8cm,
            dotted] (e) at (0,0) {};
        \node[biw] (vLeft) at (-0.8,0) {$m_1$\nodepart{lower}$n_1$};
        \node[biw] (vRight) at (0.8,0) {$m_2$\nodepart{lower}$n_2$}
            edge [<-] (vLeft);
        \node[invisible] (v1) at (-0.6,1.4) {}
            edge [<-] (e);
        \node[] at (0,1.2) {$\cdots$};
        %\node[invisible] (v2) at (0,1.5) {}
        %    edge [<-] (e);
        \node[invisible] (v3) at (0.6,1.4) {}
            edge [<-] (e);
        \node[invisible] (v4) at (-0.6,-1.4) {}
            edge [->] (e);
        \node[] at (0,-1.2) {$\cdots$};
        %\node[invisible] (v5) at (0,-1.5) {}
        %    edge [->] (e);
        \node[invisible] (v6) at (0.6,-1.4) {}
            edge [->] (e);
    \end{tikzpicture}
    \ -\
    \begin{tikzpicture}[shorten >=1pt,>=angle 90,baseline={(4ex,-0.5ex)}]
        \node[biw] (v0) at (0,0) {$\scriptstyle m-1$\nodepart{lower}$n$};
        \node[biw] (new) at (1.5,0.7) {$\infty_0$\nodepart{lower}$\infty_0$}
            edge [<-] (v0);
        \node[invisible] (v1) at (-0.4,0.8) {}
            edge [<-] (v0);
        \node[] at (0,0.6) {$\scriptstyle\cdots$};
        %\node[invisible] (v2) at (0,0.9) {}
        %    edge [<-] (v0);
        \node[invisible] (v3) at (0.4,0.8) {}
            edge [<-] (v0);
        \node[invisible] (v4) at (-0.4,-0.8) {}
            edge [->] (v0);
        \node[] at (0,-0.6) {$\scriptstyle\cdots$};
        %\node[invisible] (v5) at (0,-0.9) {}
        %    edge [->] (v0);
        \node[invisible] (v6) at (0.4,-0.8) {}
            edge [->] (v0);
    \end{tikzpicture}
    \ -\ 
    \begin{tikzpicture}[shorten >=1pt,>=angle 90,baseline={(4ex,-0.5ex)}]
        \node[biw] (v0) at (0,0) {$\scriptstyle m$\nodepart{lower}$\scriptstyle{n-1}$};
        \node[biw] (new) at (1.5,-0.7) {$\infty_0$\nodepart{lower}$\infty_0$}
            edge [->] (v0);
        \node[invisible] (v1) at (-0.4,0.8) {}
            edge [<-] (v0);
        \node[] at (0,0.6) {$\scriptstyle\cdots$};
        %\node[invisible] (v2) at (0,0.9) {}
        %    edge [<-] (v0);
        \node[invisible] (v3) at (0.4,0.8) {}
            edge [<-] (v0);
        \node[invisible] (v4) at (-0.4,-0.8) {}
            edge [->] (v0);
        \node[] at (0,-0.6) {$\scriptstyle\cdots$};
        %\node[invisible] (v5) at (0,-0.9) {}
        %    edge [->] (v0);
        \node[invisible] (v6) at (0.4,-0.8) {}
            edge [->] (v0);
    \end{tikzpicture}
\end{equation*}
\end{remark}
\begin{proposition}\label{prop:Holieb-iso}
    Let the maps
    \begin{align*}
    F&:Der^\bullet(\mathcal{H}olieb_{p,q}^\circlearrowleft)\rightarrow\fwGC_{p+q+1} \\
    qF&:Der^\bullet(\mathcal{QH}olieb_{p,q}^{\circlearrowleft})\rightarrow\fqwGC_{p+q+1} \\
    pF&:Der^\bullet(\mathcal{PH}olieb_{p,q}^{\circlearrowleft})\rightarrow\fpwGC_{p+q+1}
    \end{align*}
be defined by mapping the graph representation of an element with out-hairs and in-hairs to a bi-weighted graphs of the same shape where the hairs have been interpreted as bi-weights. Then these maps are chain maps of degree $0$, and furthermore they are isomorphisms.
\end{proposition}
\begin{proof}
    The proof is just untwisting of definitions of both complexes, and showing that the maps is of degree 0. This is already shown in \cite{F}.
\end{proof}
\subsection{Subcomplexes of $\fwGC_k$, $\fqwGC_k$ and $\fpwGC_k$}
\begin{definition}
    Let $\fwGC_k^+$ be the subcomplex of $\fwGC_k$ of graphs having at least one vertex with positive out-weight and at least one vertex with a positive in-weight.
    Let $\fqwGC_k^+$ be the subcomplex of $\fqwGC_k$ of graphs having at least one vertex with positive out-weight.
\end{definition}
\begin{proposition}
    The isomorphisms from proposition \ref{prop:Holieb-iso} restrict to isomorphisms
\begin{align*}
    F:&Der(\mathcal{H}olieb_{p,q}^\circlearrowleft)\rightarrow \fwGC^+_{p+q+1}\\
    qF:&Der(\mathcal{QH}olieb_{p,q}^\circlearrowleft)\rightarrow \fqwGC^+_{p+q+1}
\end{align*}
\end{proposition}
\begin{proof}
    This follows by inspection.
\end{proof}
\begin{definition}
    Let $\mathsf{f}\owGC_k$ be the subcomplex of $\fwGC_k$ of graphs containing no closed path. Similarly define $\mathsf{f}\oqwGC_k$ and $\mathsf{f}\opwGC_k$.
\end{definition}
\begin{proposition}
    The isomorphisms from proposition \ref{prop:Holieb-iso} restrict to isomorphisms
    \begin{align*}
    F&:Der(\mathcal{H}olieb_{p,q}^\uparrow)\rightarrow\fowGC_{p+q+1} \\
    qF&:Der(\mathcal{QH}olieb_{p,q}^{\uparrow})\rightarrow\foqwGC_{p+q+1} \\
    pF&:Der(\mathcal{PH}olieb_{p,q}^{\uparrow})\rightarrow\fopwGC_{p+q+1}
    \end{align*}
\end{proposition}
\begin{proof}
    This follows by inspection.
\end{proof}
\subsection{Loop number zero}
All graph complexes above split over their loop number $b=e-v+1$. Let $\mathsf{b}_0\wGC_k$ be the subcomplex of $\fwGC_k$ of graphs with loop number $b=0$ and let $\wGC_k$ denote the subcomplex of graphs with looper number one and higher. Hence $\fwGC_k=\mathsf{b}_0\wGC_k\oplus\wGC_k$.
Similarly let $\mathsf{b_0}\qwGC_k$ and $\mathsf{b_0}\pwGC_k$ be the subcomplexes of graphs with loop number $b=0$, and $\qwGC_k$ and $\pwGC_k$ the subcomplexes of graphs with loop number greater than zero.
\begin{proposition}
    The cohomology of $\mathsf{b}_0\wGC_k$, $\mathsf{b}_0\qwGC_k$ and $\mathsf{b}_0\pwGC_k$ are generated by the single vertex graph
    \begin{equation*}
    \sum_{\substack{i,j\geq0\\i+j\geq 3}}(i+j-2)\ 
    \tikz[baseline=-2.07cm]
    \node[biw] at (0,-2) {$i$\nodepart{lower}$j$};\
\end{equation*}
where any graph containing a vertex of invalid bi-weight is zero.
\end{proposition}
\begin{proof}
    The proof has already been done for the complex $\mathsf{b}_0\wGC_k$. The proof naturally extends to the quasi- and pseudo cases.
\end{proof}
The oriented graph complexes decompose in a similar manner, where the loop number zero component are equivalent to the complexes above. Denote their subcomplexes of graphs with loop number greater than zero by $\owGC_k$, $\oqwGC_k$ and $\opwGC_k$ respectively.

\section{Special in-vertices and special out-vertices}
\subsection{On the convergence of spectral sequences}
    For most of the arguments in this paper, we will consider filtrations of complexes and then consider the associated spectral sequences, and so we need to make sure that these spectral sequences converge. This can be assured by considering a trick of shifting degrees as seen in \cite{W1} and \cite{F}, where we let the degree of a graph to be $k(v-1)-(k-1)e+(k-\frac{1}{2})(e-v)=\frac{1}{2}(v+e)-k$. With this new grading, the number of graphs of a specific degree is always finite. Hence any filtration considered over the number of specific vertices independent of their bi-weight will be bounded, giving us convergence.
\subsection{Special in-vertices}
\begin{definition}
    Let $\Gamma$ be a normal, quasi or psuedo bi-weighted graph. A
    vertex $x$ of $\Gamma$ is a \textit{special in-vertex} if
    \begin{enumerate}
        \item[i)] either $x$ is a univalent vertex with one outgoing edge and out-weight zero, i.e on the form $\begin{tikzpicture}[shorten >=1pt,node distance=1.2cm,auto,>=angle 90,baseline=-0.1cm]
            \node[biw] (a) at (0,0) {$0$\nodepart{lower}$n$};
            \node[] (d) at (0,1) {}
                edge [<-] (a);
        \end{tikzpicture}$
        \item[ii)] or $x$ becomes a univalent vertex of type i) after recursively removing the special-in vertices of type i) from $\Gamma$.
    \end{enumerate}
    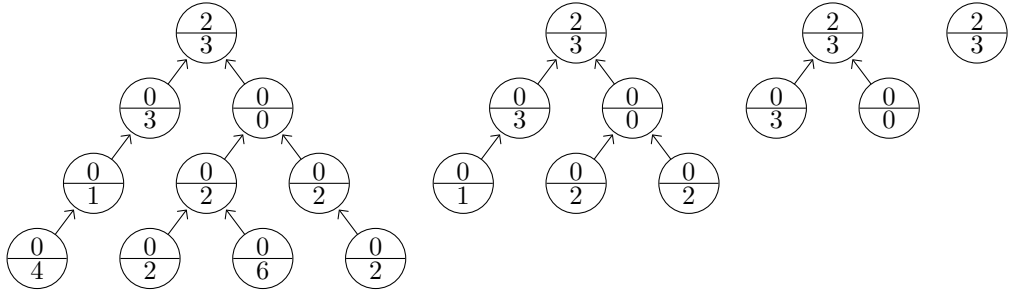
\begin{figure}[h]
    \centering
\begin{tikzpicture}[shorten >=1pt,>=angle 90,baseline={([yshift={-\ht\strutbox}]current bounding box.north)},outer sep=0pt,inner sep=0pt]
    \node[biw] (r1) at (0,0) {$2$\nodepart{lower}$3$};
    %Left leg
    \node[biw] (g1v1) at (-0.75,-1) {$0$\nodepart{lower}$3$}
        edge [->] (r1);
    \node[biw] (g1v2) at (-1.5,-2) {$0$\nodepart{lower}$1$}
        edge [->] (g1v1);
    \node[biw] (g1v3) at (-2.25,-3) {$0$\nodepart{lower}$4$}
        edge [->] (g1v2);
    %Right leg
    \node[biw] (g1v4) at (0.75,-1) {$0$\nodepart{lower}$0$}
        edge [->] (r1);
    \node[biw] (g1v5) at (1.5,-2) {$0$\nodepart{lower}$2$}
        edge [->] (g1v4);
    \node[biw] (g1v6) at (2.25,-3) {$0$\nodepart{lower}$2$}
        edge [->] (g1v5);
    %Right-left leg
    \node[biw] (g1v7) at (0,-2) {$0$\nodepart{lower}$2$}
        edge [->] (g1v4);
    \node[biw] (g1v8) at (-0.75,-3) {$0$\nodepart{lower}$2$}
        edge [->] (g1v7);
    \node[biw] (g1v9) at (0.75,-3) {$0$\nodepart{lower}$6$}
        edge [->] (g1v7);
\end{tikzpicture}
\ \ 
\begin{tikzpicture}[shorten >=1pt,>=angle 90,baseline={([yshift={-\ht\strutbox}]current bounding box.north)},outer sep=0pt,inner sep=0pt]
    \node[biw] (r1) at (0,0) {$2$\nodepart{lower}$3$};
    %Left leg
    \node[biw] (g1v1) at (-0.75,-1) {$0$\nodepart{lower}$3$}
        edge [->] (r1);
    \node[biw] (g1v2) at (-1.5,-2) {$0$\nodepart{lower}$1$}
        edge [->] (g1v1);
    %Right leg
    \node[biw] (g1v4) at (0.75,-1) {$0$\nodepart{lower}$0$}
        edge [->] (r1);
    \node[biw] (g1v5) at (1.5,-2) {$0$\nodepart{lower}$2$}
        edge [->] (g1v4);
    %Right-left leg
    \node[biw] (g1v7) at (0,-2) {$0$\nodepart{lower}$2$}
        edge [->] (g1v4);
\end{tikzpicture}
\ \ 
\begin{tikzpicture}[shorten >=1pt,>=angle 90,baseline={([yshift={-\ht\strutbox}]current bounding box.north)},outer sep=0pt,inner sep=0pt]
    \node[biw] (r1) at (0,0) {$2$\nodepart{lower}$3$};
    %Left leg
    \node[biw] (g1v1) at (-0.75,-1) {$0$\nodepart{lower}$3$}
        edge [->] (r1);
    %Right leg
    \node[biw] (g1v4) at (0.75,-1) {$0$\nodepart{lower}$0$}
        edge [->] (r1);
    %Right-left leg
\end{tikzpicture}
\ \ 
\begin{tikzpicture}[baseline={([yshift={-\ht\strutbox}]current bounding box.north)},outer sep=0pt,inner sep=0pt]
    \node[biw] (r1) at (0,0) {$2$\nodepart{lower}$3$};
    %Left leg
    %Right leg
    %Right-left leg
\end{tikzpicture}    \caption{Example of recursive removal of special-in vertices of the rightmost graph. All vertices except the top one are special in-vertices.}
\end{figure}
We say that a vertex that is not a special in-vertex is an \textit{in-core vertex}. The special in-vertices of any graph form clusters of trees where the direction of the edges always point towards a single root vertex that is an in-core vertex.
To any graph $\Gamma$ we let the associated \textit{in-core graph} $\gamma$ be the graph spanned by the in-core vertices, not considering their in-weight. Hence each vertex of an in-core graph has three integer parameters $|x|_{out}$, $|x|_{in}$ and $w_x^{out}$ associated to it.
\end{definition}

\begin{definition}\label{def:swGC}
    Let $\mathsf{S}\wGC_k$ be the subcomplex of $\wGC_k$ of graphs $\Gamma$ whose vertices $V(\Gamma)$ are independently decorated by the modified bi-weights $\frac{m}{\infty_1}$ and $\frac{m}{0}$ for $m\in\mathbb{N}$ subject to the following conditions:
    \begin{enumerate}
    \item If $x\in V(\Gamma)$ is a source, then
        \begin{equation*}
        x=\begin{tikzpicture}[shorten >=1pt,node distance=1.2cm,auto,>=angle 90,baseline=-0.1cm]
            \node[biw] (a) at (0,0) {$m$\nodepart{lower}$\infty_1$};
            \node[] (up1) at (-0.4,0.8) {}
                edge [<-] (a);
            \node[] (up2) at (0,0.6) {$\scriptstyle\cdots$};
            \node[] (up2) at (0,0.9) {$\scriptstyle \geq 1$};
            \node[] (up3) at (0.4,0.8) {}
                edge [<-] (a);
        \end{tikzpicture}
        \text{with } m+|x|_{out}\geq 2
    \end{equation*}
    \item If $x\in V(\Gamma)$ is a univalent target, then
        \begin{equation*}
        x=\begin{tikzpicture}[shorten >=1pt,node distance=1.2cm,auto,>=angle 90,baseline=-0.1cm]
            \node[biw] (a) at (0,0) {$m$\nodepart{lower}$\infty_1$};
            \node[] (d) at (0,-1) {}
                edge [->] (a);
        \end{tikzpicture}
        \text{ with }m\geq 1,\text{ or }
        x=\begin{tikzpicture}[shorten >=1pt,node distance=1.2cm,auto,>=angle 90,baseline=-0.1cm]
            \node[biw] (a) at (0,0) {$m$\nodepart{lower}$0$};
            \node[] (d) at (0,-1) {}
                edge [->] (a);
        \end{tikzpicture}
        \text{ with }m\geq 2
    \end{equation*}
    \item If $x\in V(\Gamma)$ is a target with at least two in-edges, then
    \begin{equation*}
        x=\begin{tikzpicture}[shorten >=1pt,node distance=1.2cm,auto,>=angle 90,baseline=-0.1cm]
            \node[biw] (a) at (0,0) {$m$\nodepart{lower}$\infty_1$};
            \node[] (up1) at (-0.4,-0.8) {}
                edge [->] (a);
            \node[] (up2) at (0,-0.6) {$\scriptstyle\cdots$};
            \node[] (up2) at (0,-0.9) {$\scriptstyle \geq 2$};
            \node[] (up3) at (0.4,-0.8) {}
                edge [->] (a);
        \end{tikzpicture}
        \text{with } m\geq 1,\text{ or }
        x=\begin{tikzpicture}[shorten >=1pt,node distance=1.2cm,auto,>=angle 90,baseline=-0.1cm]
            \node[biw] (a) at (0,0) {$m$\nodepart{lower}$0$};
            \node[] (up1) at (-0.4,-0.8) {}
                edge [->] (a);
            \node[] (up2) at (0,-0.6) {$\scriptstyle\cdots$};
            \node[] (up2) at (0,-0.9) {$\scriptstyle \geq 2$};
            \node[] (up3) at (0.4,-0.8) {}
                edge [->] (a);
        \end{tikzpicture}
        \text{with } m\geq 1
    \end{equation*}
    \item If $x\in V(\Gamma)$ is passing (one in-edge and one out-edge), then
    \begin{equation*}
        x=\begin{tikzpicture}[shorten >=1pt,node distance=1.2cm,auto,>=angle 90,baseline=-0.1cm]
            \node[biw] (a) at (0,0) {$m$\nodepart{lower}$\infty_1$};
            \node[] (up) at (0,1) {}
                edge [<-] (a);
            \node[] (down) at (0,-1) {}
                edge [->] (a);
        \end{tikzpicture}
        \text{ with }m\geq 0,\text{ or }
        x=\begin{tikzpicture}[shorten >=1pt,node distance=1.2cm,auto,>=angle 90,baseline=-0.1cm]
            \node[biw] (a) at (0,0) {$m$\nodepart{lower}$0$};
            \node[] (up) at (0,1) {}
                edge [<-] (a);
            \node[] (down) at (0,-1) {}
                edge [->] (a);
        \end{tikzpicture}
        \text{ with }m\geq 1
    \end{equation*}
    \item If $x\in V(\Gamma)$ is of none of the types above (i.e $x$ is at least trivalent and has at least one in-edge and at least one out-edge), then
    \begin{equation*}
        x=\begin{tikzpicture}[shorten >=1pt,node distance=1.2cm,auto,>=angle 90,baseline=-0.1cm]
            \node[biw] (a) at (0,0) {$m$\nodepart{lower}$\infty_1$};
            \node[] (down1) at (-0.4,-0.8) {}
                edge [->] (a);
            \node[] (down2) at (0,-0.6) {$\scriptstyle\cdots$};
            %\node[] (down2) at (0,-0.9) {$\scriptstyle \geq 2$};
            \node[] (down3) at (0.4,-0.8) {}
                edge [->] (a);
            \node[] (up1) at (-0.4,0.8) {}
                edge [<-] (a);
            \node[] (up2) at (0,0.6) {$\scriptstyle\cdots$};
            %\node[] (up2) at (0,0.9) {$\scriptstyle \geq 2$};
            \node[] (up3) at (0.4,0.8) {}
                edge [<-] (a);
        \end{tikzpicture}
        \text{with } m\geq 0,\text{ or }
        x=\begin{tikzpicture}[shorten >=1pt,node distance=1.2cm,auto,>=angle 90,baseline=-0.1cm]
            \node[biw] (a) at (0,0) {$m$\nodepart{lower}$0$};
            \node[] (down1) at (-0.4,-0.8) {}
                edge [->] (a);
            \node[] (down2) at (0,-0.6) {$\scriptstyle\cdots$};
            %\node[] (down2) at (0,-0.9) {$\scriptstyle \geq 2$};
            \node[] (down3) at (0.4,-0.8) {}
                edge [->] (a);
            \node[] (up1) at (-0.4,0.8) {}
                edge [<-] (a);
            \node[] (up2) at (0,0.6) {$\scriptstyle\cdots$};
            %\node[] (up2) at (0,0.9) {$\scriptstyle \geq 2$};
            \node[] (up3) at (0.4,0.8) {}
                edge [<-] (a);
        \end{tikzpicture}
        \text{with } m\geq 0
    \end{equation*}
\end{enumerate}
    Note that the graphs of $\mathsf{S}\wGC_k$ do not contain special-in vertices, and that no such vertices are created by the differential. The differential acts on graphs $\Gamma$ with vertices of the types (1)-(5) above as $d(\Gamma)=\sum_{x\in V(\Gamma)}d_x(\Gamma)$. The map $d_x$ act on vertices with in-weight $\infty_1$ and $0$ respectively as
    \begin{align*}
    d_x\Big(\ 
    \begin{tikzpicture}[shorten >=1pt,>=angle 90,baseline={([yshift=-.5ex]current bounding box.center)}]
        \node[biw] (v0) at (0,0) {$m$\nodepart{lower}$\infty_1$};
        \node[invisible] (v1) at (-0.4,0.8) {}
            edge [<-] (v0);
        \node[] at (0,0.6) {$\scriptstyle\cdots$};
        %\node[invisible] (v2) at (0,0.9) {};
        %    edge [<-] (v0);
        \node[invisible] (v3) at (0.4,0.8) {}
            edge [<-] (v0);
        \node[invisible] (v4) at (-0.4,-0.8) {}
            edge [->] (v0);
        \node[] at (0,-0.6) {$\scriptstyle\cdots$};
        %\node[invisible] (v5) at (0,-0.9) {}
        %    edge [->] (v0);
        \node[invisible] (v6) at (0.4,-0.8) {}
            edge [->] (v0);
    \end{tikzpicture}\ \Big) \ 
    &= \ \sum_{\scriptstyle{m=m_1+m_2}}
    \Bigg(
    \begin{tikzpicture}[shorten >=1pt,>=angle 90,baseline={([yshift=-.5ex]current bounding box.center)}]
        \node[ellipse,
            draw = black,
            minimum width = 3cm, 
            minimum height = 1.8cm,
            dotted] (e) at (0,0) {};
        \node[biw] (vLeft) at (-0.8,0) {$m_1$\nodepart{lower}$\infty_1$};
        \node[biw] (vRight) at (0.8,0) {$m_2$\nodepart{lower}$\infty_1$}
            edge [<-] (vLeft);
        \node[invisible] (v1) at (-0.6,1.4) {}
            edge [<-] (e);
        \node[] at (0,1.2) {$\cdots$};
        %\node[invisible] (v2) at (0,1.5) {}
        %    edge [<-] (e);
        \node[invisible] (v3) at (0.6,1.4) {}
            edge [<-] (e);
        \node[invisible] (v4) at (-0.6,-1.4) {}
            edge [->] (e);
        \node[] at (0,-1.2) {$\cdots$};
        %\node[invisible] (v5) at (0,-1.5) {}
        %    edge [->] (e);
        \node[invisible] (v6) at (0.6,-1.4) {}
            edge [->] (e);
    \end{tikzpicture}
    \ + \
    \begin{tikzpicture}[shorten >=1pt,>=angle 90,baseline={([yshift=-.5ex]current bounding box.center)}]
        \node[ellipse,
            draw = black,
            minimum width = 3cm, 
            minimum height = 1.8cm,
            dotted] (e) at (0,0) {};
        \node[biw] (vLeft) at (-0.8,0) {$m_1$\nodepart{lower}$0$};
        \node[biw] (vRight) at (0.8,0) {$m_2$\nodepart{lower}$\infty_1$}
            edge [<-] (vLeft);
        \node[invisible] (v1) at (-0.6,1.4) {}
            edge [<-] (e);
        \node[] at (0,1.2) {$\cdots$};
        %\node[invisible] (v2) at (0,1.5) {}
        %    edge [<-] (e);
        \node[invisible] (v3) at (0.6,1.4) {}
            edge [<-] (e);
        \node[invisible] (v4) at (-0.6,-1.4) {}
            edge [->] (e);
        \node[] at (0,-1.2) {$\cdots$};
        %\node[invisible] (v5) at (0,-1.5) {}
        %    edge [->] (e);
        \node[invisible] (v6) at (0.6,-1.4) {}
            edge [->] (e);
    \end{tikzpicture}
    \ + \ 
    \begin{tikzpicture}[shorten >=1pt,>=angle 90,baseline={([yshift=-.5ex]current bounding box.center)}]
        \node[ellipse,
            draw = black,
            minimum width = 3cm, 
            minimum height = 1.8cm,
            dotted] (e) at (0,0) {};
        \node[biw] (vLeft) at (-0.8,0) {$m_1$\nodepart{lower}$\infty_1$};
        \node[biw] (vRight) at (0.8,0) {$m_2$\nodepart{lower}$0$}
            edge [<-] (vLeft);
        \node[invisible] (v1) at (-0.6,1.4) {}
            edge [<-] (e);
        \node[] at (0,1.2) {$\cdots$};
        %\node[invisible] (v2) at (0,1.5) {}
        %    edge [<-] (e);
        \node[invisible] (v3) at (0.6,1.4) {}
            edge [<-] (e);
        \node[invisible] (v4) at (-0.6,-1.4) {}
            edge [->] (e);
        \node[] at (0,-1.2) {$\cdots$};
        %\node[invisible] (v5) at (0,-1.5) {}
        %    edge [->] (e);
        \node[invisible] (v6) at (0.6,-1.4) {}
            edge [->] (e);
    \end{tikzpicture}
    \Bigg)\\
    &\qquad \ -\
    \begin{tikzpicture}[shorten >=1pt,>=angle 90,baseline={(4ex,-0.5ex)}]
        \node[biw] (v0) at (0,0) {$\scriptstyle m-1$\nodepart{lower}$\infty_1$};
        \node[biw] (new) at (1.5,0.7) {$\infty_1$\nodepart{lower}$\infty_0$}
            edge [<-] (v0);
        \node[invisible] (v1) at (-0.4,0.8) {}
            edge [<-] (v0);
        \node[] at (0,0.6) {$\scriptstyle\cdots$};
        %\node[invisible] (v2) at (0,0.9) {}
        %    edge [<-] (v0);
        \node[invisible] (v3) at (0.4,0.8) {}
            edge [<-] (v0);
        \node[invisible] (v4) at (-0.4,-0.8) {}
            edge [->] (v0);
        \node[] at (0,-0.6) {$\scriptstyle\cdots$};
        %\node[invisible] (v5) at (0,-0.9) {}
        %    edge [->] (v0);
        \node[invisible] (v6) at (0.4,-0.8) {}
            edge [->] (v0);
    \end{tikzpicture}
    \ -\ 
    \begin{tikzpicture}[shorten >=1pt,>=angle 90,baseline={(4ex,-0.5ex)}]
        \node[biw] (v0) at (0,0) {$m$\nodepart{lower}$\infty_1$};
        \node[biw] (new) at (1.5,-0.7) {$\infty_1$\nodepart{lower}$\infty_1$}
            edge [->] (v0);
        \node[invisible] (v1) at (-0.4,0.8) {}
            edge [<-] (v0);
        \node[] at (0,0.6) {$\scriptstyle\cdots$};
        %\node[invisible] (v2) at (0,0.9) {}
        %    edge [<-] (v0);
        \node[invisible] (v3) at (0.4,0.8) {}
            edge [<-] (v0);
        \node[invisible] (v4) at (-0.4,-0.8) {}
            edge [->] (v0);
        \node[] at (0,-0.6) {$\scriptstyle\cdots$};
        %\node[invisible] (v5) at (0,-0.9) {}
        %    edge [->] (v0);
        \node[invisible] (v6) at (0.4,-0.8) {}
            edge [->] (v0);
    \end{tikzpicture}
    \ -\ 
    \begin{tikzpicture}[shorten >=1pt,>=angle 90,baseline={(4ex,-0.5ex)}]
        \node[biw] (v0) at (0,0) {$m$\nodepart{lower}$0$};
        \node[biw] (new) at (1.5,-0.7) {$\infty_1$\nodepart{lower}$\infty_1$}
            edge [->] (v0);
        \node[invisible] (v1) at (-0.4,0.8) {}
            edge [<-] (v0);
        \node[] at (0,0.6) {$\scriptstyle\cdots$};
        %\node[invisible] (v2) at (0,0.9) {}
        %    edge [<-] (v0);
        \node[invisible] (v3) at (0.4,0.8) {}
            edge [<-] (v0);
        \node[invisible] (v4) at (-0.4,-0.8) {}
            edge [->] (v0);
        \node[] at (0,-0.6) {$\scriptstyle\cdots$};
        %\node[invisible] (v5) at (0,-0.9) {}
        %    edge [->] (v0);
        \node[invisible] (v6) at (0.4,-0.8) {}
            edge [->] (v0);
    \end{tikzpicture}\\
%%%%%%%%%%%%%%%%IN-WEIGHT ZERO%%%%%%%%%%%%%%%%
    d_x\Big(\ 
    \begin{tikzpicture}[shorten >=1pt,>=angle 90,baseline={([yshift=-.5ex]current bounding box.center)}]
        \node[biw] (v0) at (0,0) {$m$\nodepart{lower}$0$};
        \node[invisible] (v1) at (-0.4,0.8) {}
            edge [<-] (v0);
        \node[] at (0,0.6) {$\scriptstyle\cdots$};
        %\node[invisible] (v2) at (0,0.9) {};
        %    edge [<-] (v0);
        \node[invisible] (v3) at (0.4,0.8) {}
            edge [<-] (v0);
        \node[invisible] (v4) at (-0.4,-0.8) {}
            edge [->] (v0);
        \node[] at (0,-0.6) {$\scriptstyle\cdots$};
        %\node[invisible] (v5) at (0,-0.9) {}
        %    edge [->] (v0);
        \node[invisible] (v6) at (0.4,-0.8) {}
            edge [->] (v0);
    \end{tikzpicture}\ \Big) \ 
    &= \ \sum_{\scriptstyle{m=m_1+m_2}}
    \begin{tikzpicture}[shorten >=1pt,>=angle 90,baseline={([yshift=-.5ex]current bounding box.center)}]
        \node[ellipse,
            draw = black,
            minimum width = 3cm, 
            minimum height = 1.8cm,
            dotted] (e) at (0,0) {};
        \node[biw] (vLeft) at (-0.8,0) {$m_1$\nodepart{lower}$0$};
        \node[biw] (vRight) at (0.8,0) {$m_2$\nodepart{lower}$0$}
            edge [<-] (vLeft);
        \node[invisible] (v1) at (-0.6,1.4) {}
            edge [<-] (e);
        \node[] at (0,1.2) {$\cdots$};
        %\node[invisible] (v2) at (0,1.5) {}
        %    edge [<-] (e);
        \node[invisible] (v3) at (0.6,1.4) {}
            edge [<-] (e);
        \node[invisible] (v4) at (-0.6,-1.4) {}
            edge [->] (e);
        \node[] at (0,-1.2) {$\cdots$};
        %\node[invisible] (v5) at (0,-1.5) {}
        %    edge [->] (e);
        \node[invisible] (v6) at (0.6,-1.4) {}
            edge [->] (e);
    \end{tikzpicture}
    \ -\
    \begin{tikzpicture}[shorten >=1pt,>=angle 90,baseline={(4ex,-0.5ex)}]
        \node[biw] (v0) at (0,0) {$\scriptstyle m-1$\nodepart{lower}$0$};
        \node[biw] (new) at (1.5,0.7) {$\infty_1$\nodepart{lower}$\infty_0$}
            edge [<-] (v0);
        \node[invisible] (v1) at (-0.4,0.8) {}
            edge [<-] (v0);
        \node[] at (0,0.6) {$\scriptstyle\cdots$};
        %\node[invisible] (v2) at (0,0.9) {}
        %    edge [<-] (v0);
        \node[invisible] (v3) at (0.4,0.8) {}
            edge [<-] (v0);
        \node[invisible] (v4) at (-0.4,-0.8) {}
            edge [->] (v0);
        \node[] at (0,-0.6) {$\scriptstyle\cdots$};
        %\node[invisible] (v5) at (0,-0.9) {}
        %    edge [->] (v0);
        \node[invisible] (v6) at (0.4,-0.8) {}
            edge [->] (v0);
    \end{tikzpicture}
\end{align*}
Any term on the right hand side containing at least one vertex not of the type $(1)-(5)$ is set to zero.
\end{definition}
\begin{proposition}\label{prop:SinwGC}
    The inclusion $\mathsf{S}\wGC_k \hookrightarrow \wGC_k$ is a quasi-isomorphism.
\end{proposition}
\begin{proof}
    This proposition was already proven in \cite{F}, and here we only give a brief summary of the proof.
    Consider a filtration of $\wGC_k$ over the number of in-core vertices in a graph and let $gr\ \wGC_k$ be the associated graded complex. A similar filtration on $\mathsf{S}\wGC_k$ gives an associated graded complex with a trivial differential. we prove the statement by showing that the cohomology of $gr\ \wGC_k$ is isomorphic to $\mathsf{S}\wGC_k$. 
    The complex $gr\ \wGC_k$ decompose over in-core graphs $\gamma$ as $gr\ \wGC_k=\bigoplus_{\gamma}\mathsf{inCore}(\gamma)$ where $\mathsf{inCore}(\gamma)$ is the subspace of graphs with in-core $\gamma$. To each vertex in $V(\gamma)$, we can associate the complex $\mathcal{T}_x^{in}$ consisting of special-in trees attached to $x$. The complex further decompose as $\mathsf{inCore}(\gamma)=\Big(\bigotimes_{x\in V(\gamma)}\mathcal{T}_x^{in}\Big)^{\mathrm{Aut}(\gamma)}$.
    For two vertices $x,y\in V(\gamma)$, we gather that $\mathcal{T}^{in}_x\cong\mathcal{T}^{in}_{y}$ if $|x|_{in}=|y|_{in}$, $|x|_{out}=|y|_{out}$ and $w_x^{out}=w_y^{out}$, the latter being invariant under the differential.
    Hence we have an isomorphism $\mathcal{T}^{in}_x\cong\mathcal{T}^{in}_{|x|_{out},|x|_{in},w_x^{out}}$ where $\mathcal{T}^{in}_{a,b,c}$ is a general complex of special-in trees attached to a vertex with $a$ outgoing edges, $b$ incoming edges, and  whose core-vertex has out-weight $c$. The parameters need to satisfy $a+b\geq 1$ and $(a,b,c)\neq(1,0,0)$.
    We find that the only non-trivial cycles of $\mathcal{T}^{in}_{a,b,c}$ are exactly the decorations for vertices stated in the definition of $\mathsf{S}\wGC_k$, proving the statement.
\end{proof}
Let $\mathsf{S}\owGC_k$ be the subcomplex of $\mathsf{S}\wGC_k$ of oriented graphs. This complex is naturally a subcomplex of $\owGC_k$.
\begin{corollary}\label{cor:SinwGC}
    The inclusion $\mathsf{S}\owGC_k \hookrightarrow \owGC_k$ is a quasi-isomorphism.
\end{corollary}
\begin{proof}
    The proof is similar of proposition \ref{prop:SinwGC}. Let $gr \ \owGC_k$ be the assocated graded to the filtration of $\owGC_k$ over the number of core-vertcies in a graph. The complex decompose over oriented core graphs $\gamma$ as
    \begin{equation*}
        gr\ \owGC_k=\bigoplus_{\gamma}\mathsf{inCore}(\gamma)
    \end{equation*}
    where $\mathsf{inCore}(\gamma)$ is the subcomplex of graphs with in-core $\gamma$. These are the same complexes as in proposition \ref{prop:SinwGC}, and so the result follows.
\end{proof}
\begin{definition}
    Let $\mathsf{S}\qwGC_k$ be the subcomplex of $\qwGC_k$ of graphs $\Gamma$ whose vertices $V(\Gamma)$ are independently decorated by the modified bi-weights $\frac{m}{\infty_1}$ and $\frac{m}{0}$ for $m\in\mathbb{N}$ subject to the same conditions $(2)-(5)$ as in definition \ref{def:swGC} together with the additional modified condition
    \begin{itemize}
        \item[$(1')$] If $x\in V(\Gamma)$ is a source, then
        \begin{equation*}
            x=\begin{tikzpicture}[shorten >=1pt,node distance=1.2cm,auto,>=angle 90,baseline=-0.1cm]
                \node[biw] (a) at (0,0) {$m$\nodepart{lower}$\infty_1$};
                \node[] (up1) at (-0.4,0.8) {}
                    edge [<-] (a);
                \node[] (up2) at (0,0.6) {$\scriptstyle\cdots$};
                \node[] (up2) at (0,0.9) {$\scriptstyle \geq 1$};
                \node[] (up3) at (0.4,0.8) {}
                    edge [<-] (a);
            \end{tikzpicture}
            \text{with } m+|x|_{out}\geq 2, \text{ or }
            x=\begin{tikzpicture}[shorten >=1pt,node distance=1.2cm,auto,>=angle 90,baseline=-0.1cm]
                \node[biw] (a) at (0,0) {$m$\nodepart{lower}$0$};
                \node[] (up1) at (-0.4,0.8) {}
                    edge [<-] (a);
                \node[] (up2) at (0,0.6) {$\scriptstyle\cdots$};
                \node[] (up2) at (0,0.9) {$\scriptstyle \geq 1$};
                \node[] (up3) at (0.4,0.8) {}
                    edge [<-] (a);
            \end{tikzpicture}
            \text{with } m+|x|_{out}\geq 3.
    \end{equation*}
    \end{itemize}
    Note that this complex neither contain any special-in vertices nor create any such vertices under the differential.
    The differential act similar to that of $\mathsf{S}\wGC_k$, but where additional terms are allowed of graphs with vertices having these new decorations. That is, the differential acts on a graph $\Gamma\in \mathsf{S}\qwGC_k$ with vertices of the types $(1')$ and $(2)-(5)$ as $d(\Gamma)=\sum_{x\in V(\Gamma)}d_x(\Gamma)$. The map $d_x$ act on vertices with in-weight $\infty_1$ and $0$ respectively as
\begin{align*}
    d_x\Big(\ 
    \begin{tikzpicture}[shorten >=1pt,>=angle 90,baseline={([yshift=-.5ex]current bounding box.center)}]
        \node[biw] (v0) at (0,0) {$m$\nodepart{lower}$\infty_1$};
        \node[invisible] (v1) at (-0.4,0.8) {}
            edge [<-] (v0);
        \node[] at (0,0.6) {$\scriptstyle\cdots$};
        %\node[invisible] (v2) at (0,0.9) {};
        %    edge [<-] (v0);
        \node[invisible] (v3) at (0.4,0.8) {}
            edge [<-] (v0);
        \node[invisible] (v4) at (-0.4,-0.8) {}
            edge [->] (v0);
        \node[] at (0,-0.6) {$\scriptstyle\cdots$};
        %\node[invisible] (v5) at (0,-0.9) {}
        %    edge [->] (v0);
        \node[invisible] (v6) at (0.4,-0.8) {}
            edge [->] (v0);
    \end{tikzpicture}\ \Big) \ 
    &= \ \sum_{\scriptstyle{m=m_1+m_2}}
    \Bigg(
    \begin{tikzpicture}[shorten >=1pt,>=angle 90,baseline={([yshift=-.5ex]current bounding box.center)}]
        \node[ellipse,
            draw = black,
            minimum width = 3cm, 
            minimum height = 1.8cm,
            dotted] (e) at (0,0) {};
        \node[biw] (vLeft) at (-0.8,0) {$m_1$\nodepart{lower}$\infty_1$};
        \node[biw] (vRight) at (0.8,0) {$m_2$\nodepart{lower}$\infty_1$}
            edge [<-] (vLeft);
        \node[invisible] (v1) at (-0.6,1.4) {}
            edge [<-] (e);
        \node[] at (0,1.2) {$\cdots$};
        %\node[invisible] (v2) at (0,1.5) {}
        %    edge [<-] (e);
        \node[invisible] (v3) at (0.6,1.4) {}
            edge [<-] (e);
        \node[invisible] (v4) at (-0.6,-1.4) {}
            edge [->] (e);
        \node[] at (0,-1.2) {$\cdots$};
        %\node[invisible] (v5) at (0,-1.5) {}
        %    edge [->] (e);
        \node[invisible] (v6) at (0.6,-1.4) {}
            edge [->] (e);
    \end{tikzpicture}
    \ + \
    \begin{tikzpicture}[shorten >=1pt,>=angle 90,baseline={([yshift=-.5ex]current bounding box.center)}]
        \node[ellipse,
            draw = black,
            minimum width = 3cm, 
            minimum height = 1.8cm,
            dotted] (e) at (0,0) {};
        \node[biw] (vLeft) at (-0.8,0) {$m_1$\nodepart{lower}$0$};
        \node[biw] (vRight) at (0.8,0) {$m_2$\nodepart{lower}$\infty_1$}
            edge [<-] (vLeft);
        \node[invisible] (v1) at (-0.6,1.4) {}
            edge [<-] (e);
        \node[] at (0,1.2) {$\cdots$};
        %\node[invisible] (v2) at (0,1.5) {}
        %    edge [<-] (e);
        \node[invisible] (v3) at (0.6,1.4) {}
            edge [<-] (e);
        \node[invisible] (v4) at (-0.6,-1.4) {}
            edge [->] (e);
        \node[] at (0,-1.2) {$\cdots$};
        %\node[invisible] (v5) at (0,-1.5) {}
        %    edge [->] (e);
        \node[invisible] (v6) at (0.6,-1.4) {}
            edge [->] (e);
    \end{tikzpicture}
    \ + \ 
    \begin{tikzpicture}[shorten >=1pt,>=angle 90,baseline={([yshift=-.5ex]current bounding box.center)}]
        \node[ellipse,
            draw = black,
            minimum width = 3cm, 
            minimum height = 1.8cm,
            dotted] (e) at (0,0) {};
        \node[biw] (vLeft) at (-0.8,0) {$m_1$\nodepart{lower}$\infty_1$};
        \node[biw] (vRight) at (0.8,0) {$m_2$\nodepart{lower}$0$}
            edge [<-] (vLeft);
        \node[invisible] (v1) at (-0.6,1.4) {}
            edge [<-] (e);
        \node[] at (0,1.2) {$\cdots$};
        %\node[invisible] (v2) at (0,1.5) {}
        %    edge [<-] (e);
        \node[invisible] (v3) at (0.6,1.4) {}
            edge [<-] (e);
        \node[invisible] (v4) at (-0.6,-1.4) {}
            edge [->] (e);
        \node[] at (0,-1.2) {$\cdots$};
        %\node[invisible] (v5) at (0,-1.5) {}
        %    edge [->] (e);
        \node[invisible] (v6) at (0.6,-1.4) {}
            edge [->] (e);
    \end{tikzpicture}
    \Bigg)\\
    &\qquad \ -\ 
    \begin{tikzpicture}[shorten >=1pt,>=angle 90,baseline={(4ex,-0.5ex)}]
        \node[biw] (v0) at (0,0) {$\scriptstyle m-1$\nodepart{lower}$\infty_1$};
        \node[biw] (new) at (1.5,0.7) {$\infty_1$\nodepart{lower}$\infty_0$}
            edge [<-] (v0);
        \node[invisible] (v1) at (-0.4,0.8) {}
            edge [<-] (v0);
        \node[] at (0,0.6) {$\scriptstyle\cdots$};
        %\node[invisible] (v2) at (0,0.9) {}
        %    edge [<-] (v0);
        \node[invisible] (v3) at (0.4,0.8) {}
            edge [<-] (v0);
        \node[invisible] (v4) at (-0.4,-0.8) {}
            edge [->] (v0);
        \node[] at (0,-0.6) {$\scriptstyle\cdots$};
        %\node[invisible] (v5) at (0,-0.9) {}
        %    edge [->] (v0);
        \node[invisible] (v6) at (0.4,-0.8) {}
            edge [->] (v0);
    \end{tikzpicture} 
    \ - \
    \begin{tikzpicture}[shorten >=1pt,>=angle 90,baseline={(4ex,-0.5ex)}]
        \node[biw] (v0) at (0,0) {$m$\nodepart{lower}$\infty_1$};
        \node[biw] (new) at (1.5,-0.7) {$\infty_1$\nodepart{lower}$\infty_1$}
            edge [->] (v0);
        \node[invisible] (v1) at (-0.4,0.8) {}
            edge [<-] (v0);
        \node[] at (0,0.6) {$\scriptstyle\cdots$};
        %\node[invisible] (v2) at (0,0.9) {}
        %    edge [<-] (v0);
        \node[invisible] (v3) at (0.4,0.8) {}
            edge [<-] (v0);
        \node[invisible] (v4) at (-0.4,-0.8) {}
            edge [->] (v0);
        \node[] at (0,-0.6) {$\scriptstyle\cdots$};
        %\node[invisible] (v5) at (0,-0.9) {}
        %    edge [->] (v0);
        \node[invisible] (v6) at (0.4,-0.8) {}
            edge [->] (v0);
    \end{tikzpicture}
    \ -\ 
    \begin{tikzpicture}[shorten >=1pt,>=angle 90,baseline={(4ex,-0.5ex)}]
        \node[biw] (v0) at (0,0) {$m$\nodepart{lower}$0$};
        \node[biw] (new) at (1.5,-0.7) {$\infty_1$\nodepart{lower}$\infty_1$}
            edge [->] (v0);
        \node[invisible] (v1) at (-0.4,0.8) {}
            edge [<-] (v0);
        \node[] at (0,0.6) {$\scriptstyle\cdots$};
        %\node[invisible] (v2) at (0,0.9) {}
        %    edge [<-] (v0);
        \node[invisible] (v3) at (0.4,0.8) {}
            edge [<-] (v0);
        \node[invisible] (v4) at (-0.4,-0.8) {}
            edge [->] (v0);
        \node[] at (0,-0.6) {$\scriptstyle\cdots$};
        %\node[invisible] (v5) at (0,-0.9) {}
        %    edge [->] (v0);
        \node[invisible] (v6) at (0.4,-0.8) {}
            edge [->] (v0);
    \end{tikzpicture}\\
    &\qquad \ -\
    \begin{tikzpicture}[shorten >=1pt,>=angle 90,baseline={(4ex,-0.5ex)}]
        \node[biw] (v0) at (0,0) {$m$\nodepart{lower}$\infty_1$};
        \node[biw] (new) at (1.5,-0.7) {$\infty_2$\nodepart{lower}$0$}
            edge [->] (v0);
        \node[invisible] (v1) at (-0.4,0.8) {}
            edge [<-] (v0);
        \node[] at (0,0.6) {$\scriptstyle\cdots$};
        %\node[invisible] (v2) at (0,0.9) {}
        %    edge [<-] (v0);
        \node[invisible] (v3) at (0.4,0.8) {}
            edge [<-] (v0);
        \node[invisible] (v4) at (-0.4,-0.8) {}
            edge [->] (v0);
        \node[] at (0,-0.6) {$\scriptstyle\cdots$};
        %\node[invisible] (v5) at (0,-0.9) {}
        %    edge [->] (v0);
        \node[invisible] (v6) at (0.4,-0.8) {}
            edge [->] (v0);
    \end{tikzpicture}
    \ -\ 
    \begin{tikzpicture}[shorten >=1pt,>=angle 90,baseline={(4ex,-0.5ex)}]
        \node[biw] (v0) at (0,0) {$m$\nodepart{lower}$0$};
        \node[biw] (new) at (1.5,-0.7) {$\infty_2$\nodepart{lower}$0$}
            edge [->] (v0);
        \node[invisible] (v1) at (-0.4,0.8) {}
            edge [<-] (v0);
        \node[] at (0,0.6) {$\scriptstyle\cdots$};
        %\node[invisible] (v2) at (0,0.9) {}
        %    edge [<-] (v0);
        \node[invisible] (v3) at (0.4,0.8) {}
            edge [<-] (v0);
        \node[invisible] (v4) at (-0.4,-0.8) {}
            edge [->] (v0);
        \node[] at (0,-0.6) {$\scriptstyle\cdots$};
        %\node[invisible] (v5) at (0,-0.9) {}
        %    edge [->] (v0);
        \node[invisible] (v6) at (0.4,-0.8) {}
            edge [->] (v0);
    \end{tikzpicture}
     \\
%%%%%%%%%%%%%%%%IN-WEIGHT ZERO%%%%%%%%%%%%%%%%
    d_x\Big(\ 
    \begin{tikzpicture}[shorten >=1pt,>=angle 90,baseline={([yshift=-.5ex]current bounding box.center)}]
        \node[biw] (v0) at (0,0) {$m$\nodepart{lower}$0$};
        \node[invisible] (v1) at (-0.4,0.8) {}
            edge [<-] (v0);
        \node[] at (0,0.6) {$\scriptstyle\cdots$};
        %\node[invisible] (v2) at (0,0.9) {};
        %    edge [<-] (v0);
        \node[invisible] (v3) at (0.4,0.8) {}
            edge [<-] (v0);
        \node[invisible] (v4) at (-0.4,-0.8) {}
            edge [->] (v0);
        \node[] at (0,-0.6) {$\scriptstyle\cdots$};
        %\node[invisible] (v5) at (0,-0.9) {}
        %    edge [->] (v0);
        \node[invisible] (v6) at (0.4,-0.8) {}
            edge [->] (v0);
    \end{tikzpicture}\ \Big) \ 
    &= \ \sum_{\scriptstyle{m=m_1+m_2}}
    \begin{tikzpicture}[shorten >=1pt,>=angle 90,baseline={([yshift=-.5ex]current bounding box.center)}]
        \node[ellipse,
            draw = black,
            minimum width = 3cm, 
            minimum height = 1.8cm,
            dotted] (e) at (0,0) {};
        \node[biw] (vLeft) at (-0.8,0) {$m_1$\nodepart{lower}$0$};
        \node[biw] (vRight) at (0.8,0) {$m_2$\nodepart{lower}$0$}
            edge [<-] (vLeft);
        \node[invisible] (v1) at (-0.6,1.4) {}
            edge [<-] (e);
        \node[] at (0,1.2) {$\cdots$};
        %\node[invisible] (v2) at (0,1.5) {}
        %    edge [<-] (e);
        \node[invisible] (v3) at (0.6,1.4) {}
            edge [<-] (e);
        \node[invisible] (v4) at (-0.6,-1.4) {}
            edge [->] (e);
        \node[] at (0,-1.2) {$\cdots$};
        %\node[invisible] (v5) at (0,-1.5) {}
        %    edge [->] (e);
        \node[invisible] (v6) at (0.6,-1.4) {}
            edge [->] (e);
    \end{tikzpicture}
    \ -\
    \begin{tikzpicture}[shorten >=1pt,>=angle 90,baseline={(4ex,-0.5ex)}]
        \node[biw] (v0) at (0,0) {$\scriptstyle m-1$\nodepart{lower}$0$};
        \node[biw] (new) at (1.5,0.7) {$\infty_1$\nodepart{lower}$\infty_0$}
            edge [<-] (v0);
        \node[invisible] (v1) at (-0.4,0.8) {}
            edge [<-] (v0);
        \node[] at (0,0.6) {$\scriptstyle\cdots$};
        %\node[invisible] (v2) at (0,0.9) {}
        %    edge [<-] (v0);
        \node[invisible] (v3) at (0.4,0.8) {}
            edge [<-] (v0);
        \node[invisible] (v4) at (-0.4,-0.8) {}
            edge [->] (v0);
        \node[] at (0,-0.6) {$\scriptstyle\cdots$};
        %\node[invisible] (v5) at (0,-0.9) {}
        %    edge [->] (v0);
        \node[invisible] (v6) at (0.4,-0.8) {}
            edge [->] (v0);
    \end{tikzpicture}
\end{align*}
Any term on the right hand side containing at least one vertex not of the type $(1')$ or $(2)-(5)$ is set to zero.
\end{definition}
\begin{proposition}\label{prop:SinqwGC}
    The inclusion $\mathsf{S}\qwGC_k\hookrightarrow\qwGC_k$ is a quasi-isomorphism. 
\end{proposition}
\begin{proof}
    The proof follows the same structure as the proof of proposition \ref{prop:SinwGC}. Let $gr\ \mathsf{S}\qwGC_k$ be the associated graded of the filtration over the number of in-core vertices in a graph. This complex decompose over in-core graphs as
    \begin{equation*}
        gr\ \mathsf{S}\qwGC_k = \bigoplus \mathsf{QinCore}(\gamma)
    \end{equation*}
    where $\mathsf{QinCore}(\gamma)$ is the subcomplex of graphs with in-core graph $\gamma$. This complex further decompose as
    \begin{equation*}
        \mathsf{QinCore}(\gamma)\cong \Big(\bigotimes_{x\in V(\gamma)}q\mathcal{T}^{in}_x\Big)^{\mathrm{Aut}(\gamma)}
    \end{equation*}
    where $q\mathcal{T}_x^{in}$ is the complex of in-core trees attached to the vertex $x$ in $\gamma$. We first note that if $|x|_{in}\geq 1$, then $q\mathcal{T}_x^{in}\cong\mathcal{T}^{in}_x$, which shows cases $(2)-(5)$.
    Suppose that $|x|_{in}=0$. Then $q\mathcal{T}_x^{in}$ contain a subcomplex $\mathcal{C}$ generated by single vertex graphs on the form 
    $\begin{tikzpicture}[shorten >=1pt,node distance=1.2cm,auto,>=angle 90,baseline=-0.1cm]
            \node[biw] (a) at (0,0) {$m$\nodepart{lower}$0$};
            \node[] (up1) at (-0.4,0.8) {}
                edge [<-] (a);
            \node[] (up2) at (0,0.6) {$\scriptstyle\cdots$};
            \node[] (up2) at (0,0.9) {$\scriptstyle \geq 1$};
            \node[] (up3) at (0.4,0.8) {}
                edge [<-] (a);
        \end{tikzpicture}$
    for $m\geq 0$ such that $m+|x|_{out}\geq 3$. These are all non-trivial cycles, and the complex split as $q\mathcal{T}^{in}_x=\mathcal{C}\oplus\mathcal{T}^{in}_x$. The cohomology of $\mathcal{T}^{in}_x$ is given by condition $(1)$, and so this gives the modified condition $(1')$.
\end{proof}
Let $\mathsf{S}\oqwGC_k$ be the subcomplex of $\mathsf{S}\qwGC_k$ of oriented graphs. Then $\mathsf{S}\oqwGC_k$ is a subcomplex of $\owGC_k$.
\begin{corollary}
    The inclusion $\mathsf{S}\owGC_k\hookrightarrow\owGC_k$ is a quasi-isomorphism.
\end{corollary}
\begin{proof}
    See proof of corollary \ref{cor:SinwGC}.
\end{proof}
\begin{definition}
    Let $\mathsf{S}\pwGC_k$ be the subcomplex of $\pwGC_k$ of graphs $\Gamma$ whose vertices $V(\Gamma)$ are independently decorated by the modified bi-weights $\frac{m}{\infty_1}$, $\frac{m}{0}$ and $\frac{0}{\infty_2}$ for $m\in\mathbb{N}$ subject to the conditions $(1')$ and $(4)-(5)$ as in definition \ref{def:swGC}, together with the additional modified conditions
    \begin{itemize}
        \item[$(2')$] If $x\in V(\Gamma)$ is a univalent target, then
        \begin{equation*}
        x=\begin{tikzpicture}[shorten >=1pt,node distance=1.2cm,auto,>=angle 90,baseline=-0.1cm]
            \node[biw] (a) at (0,0) {$m$\nodepart{lower}$\infty_1$};
            \node[] (d) at (0,-1) {}
                edge [->] (a);
        \end{tikzpicture}
        \text{ with }m\geq 1,\ \ 
        x=\begin{tikzpicture}[shorten >=1pt,node distance=1.2cm,auto,>=angle 90,baseline=-0.1cm]
            \node[biw] (a) at (0,0) {$m$\nodepart{lower}$0$};
            \node[] (d) at (0,-1) {}
                edge [->] (a);
        \end{tikzpicture}
        \text{ with }m\geq 2 ,\text{ or }
        x=\begin{tikzpicture}[shorten >=1pt,node distance=1.2cm,auto,>=angle 90,baseline=-0.1cm]
            \node[biw] (a) at (0,0) {$0$\nodepart{lower}$\infty_2$};
            \node[] (d) at (0,-1) {}
                edge [->] (a);
        \end{tikzpicture}
        \end{equation*}
        \item[$(3')$] If $x\in V(\Gamma)$ is a target with at least two in-edges, then
    \begin{equation*}
        x=\begin{tikzpicture}[shorten >=1pt,node distance=1.2cm,auto,>=angle 90,baseline=-0.1cm]
            \node[biw] (a) at (0,0) {$m$\nodepart{lower}$\infty_1$};
            \node[] (up1) at (-0.4,-0.8) {}
                edge [->] (a);
            \node[] (up2) at (0,-0.6) {$\scriptstyle\cdots$};
            \node[] (up2) at (0,-0.9) {$\scriptstyle \geq 2$};
            \node[] (up3) at (0.4,-0.8) {}
                edge [->] (a);
        \end{tikzpicture}
        \text{with } m\geq 0,\text{ or }
        x=\begin{tikzpicture}[shorten >=1pt,node distance=1.2cm,auto,>=angle 90,baseline=-0.1cm]
            \node[biw] (a) at (0,0) {$m$\nodepart{lower}$0$};
            \node[] (up1) at (-0.4,-0.8) {}
                edge [->] (a);
            \node[] (up2) at (0,-0.6) {$\scriptstyle\cdots$};
            \node[] (up2) at (0,-0.9) {$\scriptstyle \geq 2$};
            \node[] (up3) at (0.4,-0.8) {}
                edge [->] (a);
        \end{tikzpicture}
        \text{with } m+|x|_{in}\geq 3
    \end{equation*}
    \end{itemize}
    Note that this complex neither contain any special-in vertices nor create any such vertices under the differential. The differential acts on a graph $\Gamma\in S^{in}_1\qwGC_k$ with vertices of the types $(1')$ and $(2)-(5)$ as $d(\Gamma)=\sum_{x\in V(\Gamma)}d_x(\Gamma)$. The map $d_x$ act on vertices with in-weight $\infty_1$ and $0$ respectively as
\begin{align*}
    d_x\Big(\ 
    \begin{tikzpicture}[shorten >=1pt,>=angle 90,baseline={([yshift=-.5ex]current bounding box.center)}]
        \node[biw] (v0) at (0,0) {$m$\nodepart{lower}$\infty_1$};
        \node[invisible] (v1) at (-0.4,0.8) {}
            edge [<-] (v0);
        \node[] at (0,0.6) {$\scriptstyle\cdots$};
        %\node[invisible] (v2) at (0,0.9) {};
        %    edge [<-] (v0);
        \node[invisible] (v3) at (0.4,0.8) {}
            edge [<-] (v0);
        \node[invisible] (v4) at (-0.4,-0.8) {}
            edge [->] (v0);
        \node[] at (0,-0.6) {$\scriptstyle\cdots$};
        %\node[invisible] (v5) at (0,-0.9) {}
        %    edge [->] (v0);
        \node[invisible] (v6) at (0.4,-0.8) {}
            edge [->] (v0);
    \end{tikzpicture}\ \Big) \ 
    &= \ \sum_{\scriptstyle{m=m_1+m_2}}
    \Bigg(
    \begin{tikzpicture}[shorten >=1pt,>=angle 90,baseline={([yshift=-.5ex]current bounding box.center)}]
        \node[ellipse,
            draw = black,
            minimum width = 3cm, 
            minimum height = 1.8cm,
            dotted] (e) at (0,0) {};
        \node[biw] (vLeft) at (-0.8,0) {$m_1$\nodepart{lower}$\infty_1$};
        \node[biw] (vRight) at (0.8,0) {$m_2$\nodepart{lower}$\infty_1$}
            edge [<-] (vLeft);
        \node[invisible] (v1) at (-0.6,1.4) {}
            edge [<-] (e);
        \node[] at (0,1.2) {$\cdots$};
        %\node[invisible] (v2) at (0,1.5) {}
        %    edge [<-] (e);
        \node[invisible] (v3) at (0.6,1.4) {}
            edge [<-] (e);
        \node[invisible] (v4) at (-0.6,-1.4) {}
            edge [->] (e);
        \node[] at (0,-1.2) {$\cdots$};
        %\node[invisible] (v5) at (0,-1.5) {}
        %    edge [->] (e);
        \node[invisible] (v6) at (0.6,-1.4) {}
            edge [->] (e);
    \end{tikzpicture}
    \ + \
    \begin{tikzpicture}[shorten >=1pt,>=angle 90,baseline={([yshift=-.5ex]current bounding box.center)}]
        \node[ellipse,
            draw = black,
            minimum width = 3cm, 
            minimum height = 1.8cm,
            dotted] (e) at (0,0) {};
        \node[biw] (vLeft) at (-0.8,0) {$m_1$\nodepart{lower}$0$};
        \node[biw] (vRight) at (0.8,0) {$m_2$\nodepart{lower}$\infty_1$}
            edge [<-] (vLeft);
        \node[invisible] (v1) at (-0.6,1.4) {}
            edge [<-] (e);
        \node[] at (0,1.2) {$\cdots$};
        %\node[invisible] (v2) at (0,1.5) {}
        %    edge [<-] (e);
        \node[invisible] (v3) at (0.6,1.4) {}
            edge [<-] (e);
        \node[invisible] (v4) at (-0.6,-1.4) {}
            edge [->] (e);
        \node[] at (0,-1.2) {$\cdots$};
        %\node[invisible] (v5) at (0,-1.5) {}
        %    edge [->] (e);
        \node[invisible] (v6) at (0.6,-1.4) {}
            edge [->] (e);
    \end{tikzpicture}
    \ + \ 
    \begin{tikzpicture}[shorten >=1pt,>=angle 90,baseline={([yshift=-.5ex]current bounding box.center)}]
        \node[ellipse,
            draw = black,
            minimum width = 3cm, 
            minimum height = 1.8cm,
            dotted] (e) at (0,0) {};
        \node[biw] (vLeft) at (-0.8,0) {$m_1$\nodepart{lower}$\infty_1$};
        \node[biw] (vRight) at (0.8,0) {$m_2$\nodepart{lower}$0$}
            edge [<-] (vLeft);
        \node[invisible] (v1) at (-0.6,1.4) {}
            edge [<-] (e);
        \node[] at (0,1.2) {$\cdots$};
        %\node[invisible] (v2) at (0,1.5) {}
        %    edge [<-] (e);
        \node[invisible] (v3) at (0.6,1.4) {}
            edge [<-] (e);
        \node[invisible] (v4) at (-0.6,-1.4) {}
            edge [->] (e);
        \node[] at (0,-1.2) {$\cdots$};
        %\node[invisible] (v5) at (0,-1.5) {}
        %    edge [->] (e);
        \node[invisible] (v6) at (0.6,-1.4) {}
            edge [->] (e);
    \end{tikzpicture}
    \Bigg)\\
    &\qquad \ -\ 
    \begin{tikzpicture}[shorten >=1pt,>=angle 90,baseline={(4ex,-0.5ex)}]
        \node[biw] (v0) at (0,0) {$\scriptstyle m-1$\nodepart{lower}$\infty_1$};
        \node[biw] (new) at (1.5,0.7) {$\infty_1$\nodepart{lower}$\infty_0$}
            edge [<-] (v0);
        \node[invisible] (v1) at (-0.4,0.8) {}
            edge [<-] (v0);
        \node[] at (0,0.6) {$\scriptstyle\cdots$};
        %\node[invisible] (v2) at (0,0.9) {}
        %    edge [<-] (v0);
        \node[invisible] (v3) at (0.4,0.8) {}
            edge [<-] (v0);
        \node[invisible] (v4) at (-0.4,-0.8) {}
            edge [->] (v0);
        \node[] at (0,-0.6) {$\scriptstyle\cdots$};
        %\node[invisible] (v5) at (0,-0.9) {}
        %    edge [->] (v0);
        \node[invisible] (v6) at (0.4,-0.8) {}
            edge [->] (v0);
    \end{tikzpicture} 
    \ - \
    \begin{tikzpicture}[shorten >=1pt,>=angle 90,baseline={(4ex,-0.5ex)}]
        \node[biw] (v0) at (0,0) {$\scriptstyle m-1$\nodepart{lower}$\infty_1$};
        \node[biw] (new) at (1.5,0.7) {$0$\nodepart{lower}$\infty_2$}
            edge [<-] (v0);
        \node[invisible] (v1) at (-0.4,0.8) {}
            edge [<-] (v0);
        \node[] at (0,0.6) {$\scriptstyle\cdots$};
        %\node[invisible] (v2) at (0,0.9) {}
        %    edge [<-] (v0);
        \node[invisible] (v3) at (0.4,0.8) {}
            edge [<-] (v0);
        \node[invisible] (v4) at (-0.4,-0.8) {}
            edge [->] (v0);
        \node[] at (0,-0.6) {$\scriptstyle\cdots$};
        %\node[invisible] (v5) at (0,-0.9) {}
        %    edge [->] (v0);
        \node[invisible] (v6) at (0.4,-0.8) {}
            edge [->] (v0);
    \end{tikzpicture}\\
    &\qquad \ - \
    \begin{tikzpicture}[shorten >=1pt,>=angle 90,baseline={(4ex,-0.5ex)}]
        \node[biw] (v0) at (0,0) {$m$\nodepart{lower}$\infty_1$};
        \node[biw] (new) at (1.5,-0.7) {$\infty_1$\nodepart{lower}$\infty_1$}
            edge [->] (v0);
        \node[invisible] (v1) at (-0.4,0.8) {}
            edge [<-] (v0);
        \node[] at (0,0.6) {$\scriptstyle\cdots$};
        %\node[invisible] (v2) at (0,0.9) {}
        %    edge [<-] (v0);
        \node[invisible] (v3) at (0.4,0.8) {}
            edge [<-] (v0);
        \node[invisible] (v4) at (-0.4,-0.8) {}
            edge [->] (v0);
        \node[] at (0,-0.6) {$\scriptstyle\cdots$};
        %\node[invisible] (v5) at (0,-0.9) {}
        %    edge [->] (v0);
        \node[invisible] (v6) at (0.4,-0.8) {}
            edge [->] (v0);
    \end{tikzpicture}
    \ -\ 
    \begin{tikzpicture}[shorten >=1pt,>=angle 90,baseline={(4ex,-0.5ex)}]
        \node[biw] (v0) at (0,0) {$m$\nodepart{lower}$0$};
        \node[biw] (new) at (1.5,-0.7) {$\infty_1$\nodepart{lower}$\infty_1$}
            edge [->] (v0);
        \node[invisible] (v1) at (-0.4,0.8) {}
            edge [<-] (v0);
        \node[] at (0,0.6) {$\scriptstyle\cdots$};
        %\node[invisible] (v2) at (0,0.9) {}
        %    edge [<-] (v0);
        \node[invisible] (v3) at (0.4,0.8) {}
            edge [<-] (v0);
        \node[invisible] (v4) at (-0.4,-0.8) {}
            edge [->] (v0);
        \node[] at (0,-0.6) {$\scriptstyle\cdots$};
        %\node[invisible] (v5) at (0,-0.9) {}
        %    edge [->] (v0);
        \node[invisible] (v6) at (0.4,-0.8) {}
            edge [->] (v0);
    \end{tikzpicture}\\
    &\qquad \ -\
    \begin{tikzpicture}[shorten >=1pt,>=angle 90,baseline={(4ex,-0.5ex)}]
        \node[biw] (v0) at (0,0) {$m$\nodepart{lower}$\infty_1$};
        \node[biw] (new) at (1.5,-0.7) {$\infty_2$\nodepart{lower}$0$}
            edge [->] (v0);
        \node[invisible] (v1) at (-0.4,0.8) {}
            edge [<-] (v0);
        \node[] at (0,0.6) {$\scriptstyle\cdots$};
        %\node[invisible] (v2) at (0,0.9) {}
        %    edge [<-] (v0);
        \node[invisible] (v3) at (0.4,0.8) {}
            edge [<-] (v0);
        \node[invisible] (v4) at (-0.4,-0.8) {}
            edge [->] (v0);
        \node[] at (0,-0.6) {$\scriptstyle\cdots$};
        %\node[invisible] (v5) at (0,-0.9) {}
        %    edge [->] (v0);
        \node[invisible] (v6) at (0.4,-0.8) {}
            edge [->] (v0);
    \end{tikzpicture}
    \ -\ 
    \begin{tikzpicture}[shorten >=1pt,>=angle 90,baseline={(4ex,-0.5ex)}]
        \node[biw] (v0) at (0,0) {$m$\nodepart{lower}$0$};
        \node[biw] (new) at (1.5,-0.7) {$\infty_2$\nodepart{lower}$0$}
            edge [->] (v0);
        \node[invisible] (v1) at (-0.4,0.8) {}
            edge [<-] (v0);
        \node[] at (0,0.6) {$\scriptstyle\cdots$};
        %\node[invisible] (v2) at (0,0.9) {}
        %    edge [<-] (v0);
        \node[invisible] (v3) at (0.4,0.8) {}
            edge [<-] (v0);
        \node[invisible] (v4) at (-0.4,-0.8) {}
            edge [->] (v0);
        \node[] at (0,-0.6) {$\scriptstyle\cdots$};
        %\node[invisible] (v5) at (0,-0.9) {}
        %    edge [->] (v0);
        \node[invisible] (v6) at (0.4,-0.8) {}
            edge [->] (v0);
    \end{tikzpicture}
     \\
%%%%%%%%%%%%%%%%IN-WEIGHT ZERO%%%%%%%%%%%%%%%%
    d_x\Big(\ 
    \begin{tikzpicture}[shorten >=1pt,>=angle 90,baseline={([yshift=-.5ex]current bounding box.center)}]
        \node[biw] (v0) at (0,0) {$m$\nodepart{lower}$0$};
        \node[invisible] (v1) at (-0.4,0.8) {}
            edge [<-] (v0);
        \node[] at (0,0.6) {$\scriptstyle\cdots$};
        %\node[invisible] (v2) at (0,0.9) {};
        %    edge [<-] (v0);
        \node[invisible] (v3) at (0.4,0.8) {}
            edge [<-] (v0);
        \node[invisible] (v4) at (-0.4,-0.8) {}
            edge [->] (v0);
        \node[] at (0,-0.6) {$\scriptstyle\cdots$};
        %\node[invisible] (v5) at (0,-0.9) {}
        %    edge [->] (v0);
        \node[invisible] (v6) at (0.4,-0.8) {}
            edge [->] (v0);
    \end{tikzpicture}\ \Big) \ 
    &= \ \sum_{\scriptstyle{m=m_1+m_2}}
    \begin{tikzpicture}[shorten >=1pt,>=angle 90,baseline={([yshift=-.5ex]current bounding box.center)}]
        \node[ellipse,
            draw = black,
            minimum width = 3cm, 
            minimum height = 1.8cm,
            dotted] (e) at (0,0) {};
        \node[biw] (vLeft) at (-0.8,0) {$m_1$\nodepart{lower}$0$};
        \node[biw] (vRight) at (0.8,0) {$m_2$\nodepart{lower}$0$}
            edge [<-] (vLeft);
        \node[invisible] (v1) at (-0.6,1.4) {}
            edge [<-] (e);
        \node[] at (0,1.2) {$\cdots$};
        %\node[invisible] (v2) at (0,1.5) {}
        %    edge [<-] (e);
        \node[invisible] (v3) at (0.6,1.4) {}
            edge [<-] (e);
        \node[invisible] (v4) at (-0.6,-1.4) {}
            edge [->] (e);
        \node[] at (0,-1.2) {$\cdots$};
        %\node[invisible] (v5) at (0,-1.5) {}
        %    edge [->] (e);
        \node[invisible] (v6) at (0.6,-1.4) {}
            edge [->] (e);
    \end{tikzpicture}
    \ -\
    \begin{tikzpicture}[shorten >=1pt,>=angle 90,baseline={(4ex,-0.5ex)}]
        \node[biw] (v0) at (0,0) {$\scriptstyle m-1$\nodepart{lower}$0$};
        \node[biw] (new) at (1.5,0.7) {$\infty_1$\nodepart{lower}$\infty_0$}
            edge [<-] (v0);
        \node[invisible] (v1) at (-0.4,0.8) {}
            edge [<-] (v0);
        \node[] at (0,0.6) {$\scriptstyle\cdots$};
        %\node[invisible] (v2) at (0,0.9) {}
        %    edge [<-] (v0);
        \node[invisible] (v3) at (0.4,0.8) {}
            edge [<-] (v0);
        \node[invisible] (v4) at (-0.4,-0.8) {}
            edge [->] (v0);
        \node[] at (0,-0.6) {$\scriptstyle\cdots$};
        %\node[invisible] (v5) at (0,-0.9) {}
        %    edge [->] (v0);
        \node[invisible] (v6) at (0.4,-0.8) {}
            edge [->] (v0);
    \end{tikzpicture}
    \ - \
    \begin{tikzpicture}[shorten >=1pt,>=angle 90,baseline={(4ex,-0.5ex)}]
        \node[biw] (v0) at (0,0) {$\scriptstyle m-1$\nodepart{lower}$\infty_1$};
        \node[biw] (new) at (1.5,0.7) {$0$\nodepart{lower}$\infty_2$}
            edge [<-] (v0);
        \node[invisible] (v1) at (-0.4,0.8) {}
            edge [<-] (v0);
        \node[] at (0,0.6) {$\scriptstyle\cdots$};
        %\node[invisible] (v2) at (0,0.9) {}
        %    edge [<-] (v0);
        \node[invisible] (v3) at (0.4,0.8) {}
            edge [<-] (v0);
        \node[invisible] (v4) at (-0.4,-0.8) {}
            edge [->] (v0);
        \node[] at (0,-0.6) {$\scriptstyle\cdots$};
        %\node[invisible] (v5) at (0,-0.9) {}
        %    edge [->] (v0);
        \node[invisible] (v6) at (0.4,-0.8) {}
            edge [->] (v0);
    \end{tikzpicture}\\
    d_x\Big(\ \begin{tikzpicture}[shorten >=1pt,node distance=1.2cm,auto,>=angle 90,baseline=-0.1cm]
            \node[biw] (a) at (0,0) {$0$\nodepart{lower}$\infty_2$};
            \node[] (d) at (0,-1) {}
                edge [->] (a);
        \end{tikzpicture}\ \Big) \ 
        &= \begin{tikzpicture}[shorten >=1pt,node distance=1.2cm,auto,>=angle 90,baseline=-0.1cm]
            \node[biw] (a) at (0,0) {$0$\nodepart{lower}$\infty_1$};
            \node[biw] (up) at (1.5,0.7) {$0$\nodepart{lower}$\infty_2$}
                edge [<-] (a);
            \node[] (down) at (0,-1) {}
                edge [->] (a);
        \end{tikzpicture}
        \ - \ 
        \begin{tikzpicture}[shorten >=1pt,node distance=1.2cm,auto,>=angle 90,baseline=-0.1cm]
            \node[biw] (a) at (0,0) {$0$\nodepart{lower}$\infty_1$};
            \node[biw] (up) at (1.5,-0.7) {$\infty_1$\nodepart{lower}$\infty_1$}
                edge [->] (a);
            \node[] (down) at (0,-1) {}
                edge [->] (a);
        \end{tikzpicture}
        \ - \
        \begin{tikzpicture}[shorten >=1pt,node distance=1.2cm,auto,>=angle 90,baseline=-0.1cm]
            \node[biw] (a) at (0,0) {$0$\nodepart{lower}$\infty_1$};
            \node[biw] (up) at (1.5,-0.7) {$\infty_2$\nodepart{lower}$0$}
                edge [->] (a);
            \node[] (down) at (0,-1) {}
                edge [->] (a);
        \end{tikzpicture}
\end{align*}
Any term on the right hand side containing at least one vertex not of the type $(1')-(3')$ or $(4)-(5)$ is set to zero.
\end{definition}
\begin{proposition}\label{prop:SinpwGC}
    The inclusion $\mathsf{S}\pwGC_k\hookrightarrow\pwGC_k$ is a quasi-isomorphism.
\end{proposition}

Finally, consider a filtration of $\pwGC_k$ over the number of core vertices in a graph and let $\{S^{in}_r\pwGC_k\}_{r\geq 0}$ be the associated spectral sequence.

\begin{proof}
    Let $gr\ \mathsf{S}\pwGC_k$ be the associated graded of the filtration of $\pwGC_k$ over the number of core-vertices in a graph. This complex decompose over in-core graphs $\gamma$ as
    \begin{equation*}
        gr\ \mathcal{S}\pwGC_k=\bigoplus_{\gamma}\mathsf{PinCore}(\gamma)
    \end{equation*}
    and further
    \begin{equation*}
        \mathsf{PinCore}(\gamma)\cong\Big(\bigotimes_{x\in V(\gamma)}p\mathcal{T}_x^{in}\Big)^{\mathrm{Aut}(\gamma)}
    \end{equation*}
    where $p\mathcal{T}_x^{in}$ is the complex of in-core trees attached to the vertex $x$ in $\gamma$. Recall that $|x|_{out}$, $|x|_{in}$ and $w_x^{out}$ are invariant in $p\mathcal{T}_x^{in}$. By observation, we gather the following:
    \begin{itemize}
        \item If $|x|_{out},|x|_{in}\geq 1$, then $p\mathcal{T}_x^{in}=\mathcal{T}_x^{in}$.
        \item If $|x|_{in}=0$, then $p\mathcal{T}_x^{in}=q\mathcal{T}_x^{in}$.
        \item If $|x|_{out}=0$ and $w_x^{out}\geq1$, then $p\mathcal{T}_x^{in}=\mathcal{T}_x^{in}$.
    \end{itemize}
    These cases correspond to the conditions $(4)-(5)$, $(1')$ and $(2)-(3)$ respectively. One last case remains. Suppose that $|x|_{out}=0$ and $w_x^{out}=0$. We will show that $H(p\mathcal{T}_x^{in})$ is generated by the following one vertex graphs
    
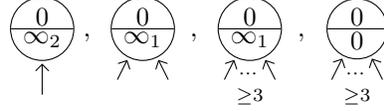
\begin{equation*}
        \begin{tikzpicture}[shorten >=1pt,node distance=1.2cm,auto,>=angle 90,baseline=-0.1cm]
            \node[biw] (a) at (0,0) {$0$\nodepart{lower}$\infty_2$};
            \node[] (d) at (0,-1) {}
                edge [->] (a);
        \end{tikzpicture}
        \ , \ 
        \begin{tikzpicture}[shorten >=1pt,node distance=1.2cm,auto,>=angle 90,baseline=-0.1cm]
            \node[biw] (a) at (0,0) {$0$\nodepart{lower}$\infty_1$};
            \node[] (up1) at (-0.4,-0.8) {}
                edge [->] (a);
            \node[] (up3) at (0.4,-0.8) {}
                edge [->] (a);
        \end{tikzpicture}
        \ , \ 
        \begin{tikzpicture}[shorten >=1pt,node distance=1.2cm,auto,>=angle 90,baseline=-0.1cm]
            \node[biw] (a) at (0,0) {$0$\nodepart{lower}$\infty_1$};
            \node[] (up1) at (-0.4,-0.8) {}
                edge [->] (a);
            \node[] (up2) at (0,-0.6) {$\scriptstyle\cdots$};
            \node[] (up2) at (0,-0.9) {$\scriptstyle \geq 3$};
            \node[] (up3) at (0.4,-0.8) {}
                edge [->] (a);
        \end{tikzpicture}
        \ , \ 
        \begin{tikzpicture}[shorten >=1pt,node distance=1.2cm,auto,>=angle 90,baseline=-0.1cm]
            \node[biw] (a) at (0,0) {$0$\nodepart{lower}$0$};
            \node[] (up1) at (-0.4,-0.8) {}
                edge [->] (a);
            \node[] (up2) at (0,-0.6) {$\scriptstyle\cdots$};
            \node[] (up2) at (0,-0.9) {$\scriptstyle \geq 3$};
            \node[] (up3) at (0.4,-0.8) {}
                edge [->] (a);
        \end{tikzpicture}
    \end{equation*}
    depending on the value of $|x|_{in}$. The proof is the same as lemma 5.2.4 and lemma 5.2.5 in \cite{F}, and so we only give a brief outline here.
    Consider a filtration of $p\mathcal{T}_x^{in}$ over the number univalent special-in vertices, where we consider the graph containing a single vertex as having one such vertex. The associated graded complex split as $gr\ p\mathcal{T}_x^{in}=\bigoplus_{N\geq 1} u_Np\mathcal{T}_x^{in}$ where $u_Np\mathcal{T}_x^{in}$ is the subcomplex of graphs with $N$ univalent special-in vertices.
    We can show that $u_Np\mathcal{T}_x^{in}$ is acyclic for $N\geq2$ by considering a filtration over the number of \textit{branch vertices}. A vertex $y$ is a branch vertex if there are at least two directed paths from a univalent vertex to $y$. Secondly we consider a filtration over the total in-weights of the branch vertices. The in-core vertex is always a branch vertex, and the differential does not depend on the bi-weight of the core-vertex. Hence the proof now is equivalent to that of lemma 5.2.5 in \cite{F}.
    Lastly we note that the one vertex graphs above are non-trivial cycles of $u_1p\mathcal{T}_x^{in}$. By the same methods as in the proof of lemma 5.2.4 we show that the cohomology of $u_1p\mathcal{T}_x^{in}$ is one or two-dimensional, finishing the proof.
\end{proof}
Let $\mathsf{S}\opwGC_k$ be the subcomplex of $\mathsf{S}\pwGC_k$ of oriented graphs. This complex is also a subcomplex of $\opwGC_k$.
\begin{corollary}
    The inclusion $\mathsf{S}\opwGC_k\hookrightarrow\opwGC_k$ is a quasi-isomorphism.
\end{corollary}
\begin{proof}
    See proof of corollary \ref{cor:SinwGC}.
\end{proof}
Let $\mathsf{S}\wGC_k^+$ be the subcomplex of $\mathsf{S}\wGC_k$ of graphs which contain at least one vertex with out-decoration $m\geq 1$ and one vertex with in-decoration $\infty_1$. This is also a subcomplex of $\wGC_k^+$. The following proposition was proven in \cite{F}, but here we give an alternate proof.
\begin{proposition}\label{prop:SinwGC+}
    The inclusion $\mathsf{S}\wGC_k^+\hookrightarrow\wGC_k^+$ is a quasi-isomorphism.
\end{proposition}
\begin{proof}
Consider the commutative diagram of short exact sequences
\begin{equation*}
    \xymatrix{
    & 0 \ar[d] & 0 \ar[d] & 0 \ar[d] & \\
    0 \ar[r] & \mathsf{S}\wGC_k^+ \ar[d] \ar[r] & \mathsf{S}\wGC_k \ar[r] \ar[d] & \mathsf{S}\wGC_k^\sim \ar[d] \ar[r] & 0  \\
    0 \ar[r] & \wGC_k^+ \ar[d] \ar[r] & \wGC_k \ar[r] \ar[d] & \wGC_k^\sim \ar[d] \ar[r] & 0  \\
    0 \ar[r] & \mathsf{Q}_k^+ \ar[d] \ar[r] & \mathsf{Q}_k \ar[r] \ar[d] & \mathsf{Q}_k^\sim \ar[d] \ar[r] & 0  \\
    & 0 & 0 & 0 & 
    }
\end{equation*}
where the new complexes are the appropriate quotient complexes. Note in particular that $\wGC_k^\sim$ is the quotient complex of bi-weighted graphs that do not both have a positive out-weight and positive in-weight at the same time. It decomposes as $\wGC_k^\sim=\wGC_k^{out}\oplus\wGC_k^{in}\oplus\wGC_k^0$ where $\wGC_k^{out}$ is the complex of graphs with at least one vertex with positive out-weight, $\wGC_k^{out}$ the complex of graphs with at least one vertex with positive in-weight, and $\wGC_k^0$ the complex of graphs with no positive out-weights nor in-weights.
The complex $\mathsf{S}\wGC_k^\sim$ is a subcomplex of $\wGC_k^\sim$ of graphs whose vertices are either only decorated by $\frac{0}{\infty_1}$ and $\frac{0}{0}$, or $\frac{m}{0}$ and $\frac{0}{0}$ for $m\geq 1$. It similarly decompose as $\mathsf{S}\wGC_k^\sim=\mathsf{S}\wGC_k^{out}\oplus\mathsf{S}\wGC_k^{in}\oplus\mathsf{S}\wGC_k^0$ where is the subcomplex of graphs with at least one vertex decorated by $\frac{m}{0}$ $m\geq 1$.
By proposition \ref{prop:SinwGC} inclusion $\mathsf{S}\wGC_k\hookrightarrow\wGC_k$ is a quasi-isomorphism, and so $\mathsf{Q}_k$ is acyclic.
If we show that the inclusion $\mathsf{S}\wGC^\sim_k\hookrightarrow\wGC^\sim_k$ is a quasi-isomorphism, we get that both $\mathsf{Q}^\sim_k$ and $\mathsf{Q}^+_k$ are acyclic, and that the inclusion $\mathsf{S}\wGC_k^+\hookrightarrow\wGC_k^+$ is a quasi-isomorphism.
Equivalently, we show that the inclusions $\mathsf{S}\wGC_k^{out}\hookrightarrow\wGC_k^{out}$, $\mathsf{S}\wGC_k^{in}\hookrightarrow\wGC_k^{out}$ and $\mathsf{S}\wGC_k^0\hookrightarrow\wGC_k^0$ are quasi-isomorphisms.
Consider a filtration of these complexes over the number of in-core vertices.
The associated graded $gr\ \wGC_k^{out}$ contain no special-in vertices since any such vertex necessarily has a positive in-weight, and so the differential is trivial and easily seen to be equal to $\mathsf{S}\wGC_k$.
Consider the associated graded of $\wGC_k^{in}\oplus\wGC_k^0$. It decomposes over in-core graphs $\gamma$ as
\begin{equation*}
    gr(\wGC_k^{in}\oplus\wGC_k^0)=\bigoplus_{\gamma}\mathsf{inCore}^{out,0}(\gamma)
\end{equation*}
where $\mathsf{inCore}^{out,0}(\gamma)$ is the complex of graphs with in-core $\gamma$. The complex further decompose as
\begin{equation*}
    \mathsf{inCore}^{out,0}(\gamma)\cong\Big(\bigotimes_{x\in V(\gamma)}\mathcal{T}_x^{in,\sim}\Big)^{\mathrm{Aut}(\gamma)}
\end{equation*}
where $\mathcal{T}_x^{in,\sim}$ is the complex of special-in trees attached to $x$. This complex is isomorphic to $\mathcal{T}_{|x|_{out},|x|_{in},0}^{in}$ as seen in proposition \ref{prop:SinwGC}. Hence $\mathcal{T}_x^{in,\sim}$ is acyclic when $x$ is a target vertex, and otherwise generated by the single vertex graph decorated by $\frac{0}{\infty_1}$ and $\frac{0}{0}$ (when $x$ is generic). This completes the proof. 
\end{proof}
Lastly let $\mathsf{S}\qwGC_k^+$ be the subcomplex of $\mathsf{S}\qwGC_k$ of graphs which contain at least one vertex with out-decoration $m\geq 1$. This is also a subcomplex of $\qwGC_k^+$.
\begin{proposition}
    The inclusion $\mathsf{S}\qwGC_k^+\hookrightarrow\qwGC_k^+$ is a quasi-isomorphism.
\end{proposition}
\begin{proof}
    The proof is analogous of that of proposition \ref{prop:SinwGC+}.
\end{proof}

\subsection{Special out-vertices}
\begin{definition}
    Let $\Gamma$ be a bi-weighted graph. A
    vertex $x$ of $\Gamma$ is a \textit{special out-vertex} if
    \begin{enumerate}
        \item[i)] either $x$ is a univalent vertex with one outgoing edge and in-weight zero, i.e on the form $\begin{tikzpicture}[shorten >=1pt,node distance=1.2cm,auto,>=angle 90,baseline=-0.1cm]
            \node[biw] (a) at (0,0) {$m$\nodepart{lower}$0$};
            \node[] (d) at (0,-1) {}
                edge [->] (a);
        \end{tikzpicture}$
        \item[ii)] or $x$ becomes a univalent vertex of type i) after recursive removal of all special-out vertices of type i) from $\Gamma$.
    \end{enumerate}
Vertices which are not special out-vertices are called \textit{out-core vertices} or just \textit{core vertices}.
\end{definition}
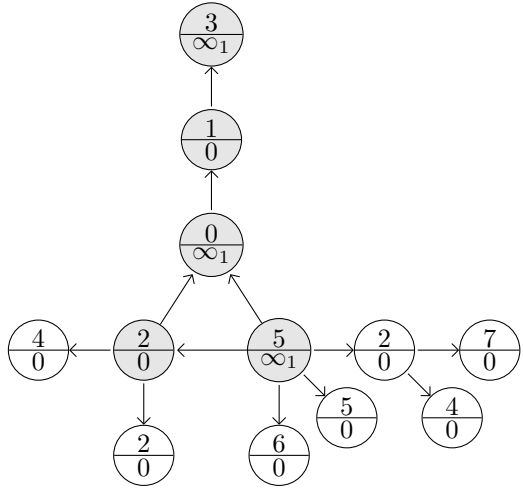
\begin{figure}[h]
    \begin{equation*}
    \begin{tikzpicture}[shorten >=1pt,node distance=1.2cm,auto,>=angle 90,baseline=-0.1cm]
    %%% CORE %%%
        \node[biw,fill=mygray] (up) at (0,0) {$0$\nodepart{lower}$\infty_1$};
        \node[biw,fill=mygray] (left) at (-0.9,-1.4) {$2$\nodepart{lower}$0$}
            edge [->] (up);
        \node[biw,fill=mygray] (right) at (0.9,-1.4) {$5$\nodepart{lower}$\infty_1$}
            edge [->] (up)
            edge [->] (left);
    %%% UP %%%
        \node[biw,fill=mygray] (up1) at (0,1.4) {$1$\nodepart{lower}$0$}
            edge [<-] (up);
        \node[biw,fill=mygray] (up2) at (0,2.8) {$3$\nodepart{lower}$\infty_1$}
            edge [<-] (up1);
    %%% LEFT %%%
        \node[biw] (left1) at (-2.3,-1.4) {$4$\nodepart{lower}$0$}
            edge [<-] (left);
        \node[biw] (left2) at (-0.9,-2.8) {$2$\nodepart{lower}$0$}
            edge [<-] (left);
    %%% RIGHT %%%
        \node[biw] (right1) at (2.3,-1.4) {$2$\nodepart{lower}$0$}
            edge [<-] (right);
        \node[biw] (right2) at (1.8,-2.3) {$5$\nodepart{lower}$0$}
            edge [<-] (right);
        \node[biw] (right3) at (0.9,-2.8) {$6$\nodepart{lower}$0$}
            edge [<-] (right);
        \node[biw] at (3.7,-1.4) {$7$\nodepart{lower}$0$}
            edge [<-] (right1);
        \node[biw] at (3.2,-2.3) {$4$\nodepart{lower}$0$}
            edge [<-] (right1);    
    \end{tikzpicture}
\end{equation*}
    \caption{Example of a graph in $S_1^{in}\OWGC_k$ where out-core vertices are colored gray and the special out-vertices are colored white.}
\end{figure}
\begin{definition}
    Let $\qGC_k$ be the subcomplex of $\mathsf{S}\wGC_k$ of graphs $\Gamma$ whose vertices $V(\Gamma)$ are independently decorated by the bi-weights $\frac{\infty_1}{\infty_1},\frac{\infty_1}{0},\frac{0}{\infty_1}$ and $\frac{0}{0}$ subject to the following conditions:
    \begin{enumerate}
        \item If $x\in V(\gamma)$ is a univalent vertex, then it is decorated by $\frac{\infty_1}{\infty_1}$.
        \item If $x\in V(\gamma)$ is a source, then it is either decorated by $\frac{\infty_1}{\infty_1}$ or $\frac{0}{\infty_1}$.
        \item If $x\in V(\gamma)$ is a target, then it is either decorated by $\frac{\infty_1}{\infty_1}$ or $\frac{\infty_1}{0}$.
        \item If $x\in V(\gamma)$ is a passing vertex, then it is either decorated by $\frac{\infty_1}{\infty_1}$, $\frac{\infty_1}{0}$ or $\frac{0}{\infty_1}$.
        \item If $x\in V(\gamma)$ is a generic vertex, then it is either decorated by $\frac{\infty_1}{\infty_1}$, $\frac{\infty_1}{0}$, $\frac{0}{\infty_1}$ or $\frac{0}{0}$.
    \end{enumerate}
    To describe the differential in an effective manner, we introduce a new notation. Let $\Gamma$ be a graph in $\qGC_d$ and let $x$ be some vertex of $\Gamma$. Set $\Big(\frac{a}{b},\frac{c}{d}\Big)_x$ to be the sum over all possible reattachements of the edges attached to $x$ among two new vertices $x'$ and $x''$ (connected by a single edge going from $x'$ to $x''$) of bi-weight $\frac{a}{b}$ and $\frac{c}{d}$ respectively.
    Any term containing an invalid bi-weight is set to zero. Pictorially, we have
    \begin{equation*}
    \Big(\frac{a}{b},\frac{c}{d}\Big)_x:=
    \begin{tikzpicture}[shorten >=1pt,>=angle 90,baseline={([yshift=-.5ex]current bounding box.center)}]
            \node[ellipse,
            draw = black,
            minimum width = 3cm, 
            minimum height = 1.8cm,
            dotted] (e) at (0,0) {};
        \node[biw] (vLeft) at (-0.8,0) {$a$\nodepart{lower}$b$};
        \node[biw] (vRight) at (0.8,0) {$c$\nodepart{lower}$d$}
            edge [<-] (vLeft);
        \node[invisible] (v1) at (-0.6,1.4) {}
            edge [<-] (e);
        \node[] at (0,1.2) {$\cdots$};
        %\node[invisible] (v2) at (0,1.5) {}
        %    edge [<-] (e);
        \node[invisible] (v3) at (0.6,1.4) {}
            edge [<-] (e);
        \node[invisible] (v4) at (-0.6,-1.4) {}
            edge [->] (e);
        \node[] at (0,-1.2) {$\cdots$};
        %\node[invisible] (v5) at (0,-1.5) {}
        %    edge [->] (e);
        \node[invisible] (v6) at (0.6,-1.4) {}
            edge [->] (e);
    \end{tikzpicture}    
    \end{equation*}
    With this notation, the differential $d$ acts on graphs $\Gamma\in\qGC_k$ as $d(\Gamma)=\sum_{x\in V(\Gamma)}d_x(\Gamma)$. The map $d_x$ act on vertices with the four different bi-weights in the following way:
    \begin{flalign*}
        d_x\Big(\ 
        \begin{tikzpicture}[shorten >=1pt,>=angle 90,baseline={([yshift=-.5ex]current bounding box.center)}]
            \node[biw] (v0) at (0,0) {$\infty_1$\nodepart{lower}$\infty_1$};
            \node[invisible] (v1) at (-0.4,0.8) {}
                edge [<-] (v0);
            \node[] at (0,0.6) {$\scriptstyle\cdots$};
            %\node[invisible] (v2) at (0,0.9) {};
            %    edge [<-] (v0);
            \node[invisible] (v3) at (0.4,0.8) {}
                edge [<-] (v0);
            \node[invisible] (v4) at (-0.4,-0.8) {}
                edge [->] (v0);
            \node[] at (0,-0.6) {$\scriptstyle\cdots$};
            %\node[invisible] (v5) at (0,-0.9) {}
            %    edge [->] (v0);
            \node[invisible] (v6) at (0.4,-0.8) {}
                edge [->] (v0);
        \end{tikzpicture}\ \Big) \
        &=\Big(\frac{\infty_1}{\infty_1},\frac{\infty_1}{\infty_1}\Big)_x
        +\Big(\frac{0}{\infty_1},\frac{\infty_1}{\infty_1}\Big)_x
        +\Big(\frac{\infty_1}{0},\frac{\infty_1}{\infty_1}\Big)_x
        +\Big(\frac{\infty_1}{\infty_1},\frac{0}{\infty_1}\Big)_x
        +\Big(\frac{\infty_1}{\infty_1},\frac{\infty_1}{0}\Big)_x\\
        &+\Big(\frac{0}{0},\frac{\infty_1}{\infty_1}\Big)_x
        +\Big(\frac{\infty_1}{0},\frac{0}{\infty_1}\Big)_x
        +\Big(\frac{0}{\infty_1},\frac{\infty_1}{0}\Big)_x
        +\Big(\frac{\infty_1}{\infty_1},\frac{0}{0}\Big)_x\\
        & - \Bigg( \ \ 
        \begin{tikzpicture}[shorten >=1pt,>=angle 90,baseline={(4ex,-0.5ex)}]
            \node[biw] (v0) at (0,0) {$\infty_1$\nodepart{lower}$\infty_1$};
            \node[biw] (new) at (1.5,0.7) {$\infty_1$\nodepart{lower}$\infty_1$}
                edge [<-] (v0);
            \node[invisible] (v1) at (-0.4,0.8) {}
                edge [<-] (v0);
            \node[] at (0,0.6) {$\scriptstyle\cdots$};
            %\node[invisible] (v2) at (0,0.9) {}
            %    edge [<-] (v0);
            \node[invisible] (v3) at (0.4,0.8) {}
                edge [<-] (v0);
            \node[invisible] (v4) at (-0.4,-0.8) {}
                edge [->] (v0);
            \node[] at (0,-0.6) {$\scriptstyle\cdots$};
            %\node[invisible] (v5) at (0,-0.9) {}
            %    edge [->] (v0);
            \node[invisible] (v6) at (0.4,-0.8) {}
                edge [->] (v0);
        \end{tikzpicture}
        \ +\ 
        \begin{tikzpicture}[shorten >=1pt,>=angle 90,baseline={(4ex,-0.5ex)}]
            \node[biw] (v0) at (0,0) {$0$\nodepart{lower}$\infty_1$};
            \node[biw] (new) at (1.5,0.7) {$\infty_1$\nodepart{lower}$\infty_1$}
                edge [<-] (v0);
            \node[invisible] (v1) at (-0.4,0.8) {}
                edge [<-] (v0);
            \node[] at (0,0.6) {$\scriptstyle\cdots$};
            %\node[invisible] (v2) at (0,0.9) {}
            %    edge [<-] (v0);
            \node[invisible] (v3) at (0.4,0.8) {}
                edge [<-] (v0);
            \node[invisible] (v4) at (-0.4,-0.8) {}
                edge [->] (v0);
            \node[] at (0,-0.6) {$\scriptstyle\cdots$};
            %\node[invisible] (v5) at (0,-0.9) {}
            %    edge [->] (v0);
            \node[invisible] (v6) at (0.4,-0.8) {}
                edge [->] (v0);
        \end{tikzpicture}
        \ +\ 
        \begin{tikzpicture}[shorten >=1pt,>=angle 90,baseline={(4ex,-0.5ex)}]
            \node[biw] (v0) at (0,0) {$\infty_1$\nodepart{lower}$\infty_1$};
            \node[biw] (new) at (1.5,-0.7) {$\infty_1$\nodepart{lower}$\infty_1$}
                edge [->] (v0);
            \node[invisible] (v1) at (-0.4,0.8) {}
                edge [<-] (v0);
            \node[] at (0,0.6) {$\scriptstyle\cdots$};
            %\node[invisible] (v2) at (0,0.9) {}
            %    edge [<-] (v0);
            \node[invisible] (v3) at (0.4,0.8) {}
                edge [<-] (v0);
            \node[invisible] (v4) at (-0.4,-0.8) {}
                edge [->] (v0);
            \node[] at (0,-0.6) {$\scriptstyle\cdots$};
            %\node[invisible] (v5) at (0,-0.9) {}
            %    edge [->] (v0);
            \node[invisible] (v6) at (0.4,-0.8) {}
                edge [->] (v0);
        \end{tikzpicture}
        \ +\ 
        \begin{tikzpicture}[shorten >=1pt,>=angle 90,baseline={(4ex,-0.5ex)}]
            \node[biw] (v0) at (0,0) {$\infty_1$\nodepart{lower}$0$};
            \node[biw] (new) at (1.5,-0.7) {$\infty_1$\nodepart{lower}$\infty_1$}
                edge [->] (v0);
            \node[invisible] (v1) at (-0.4,0.8) {}
                edge [<-] (v0);
            \node[] at (0,0.6) {$\scriptstyle\cdots$};
            %\node[invisible] (v2) at (0,0.9) {}
            %    edge [<-] (v0);
            \node[invisible] (v3) at (0.4,0.8) {}
                edge [<-] (v0);
            \node[invisible] (v4) at (-0.4,-0.8) {}
                edge [->] (v0);
            \node[] at (0,-0.6) {$\scriptstyle\cdots$};
            %\node[invisible] (v5) at (0,-0.9) {}
            %    edge [->] (v0);
            \node[invisible] (v6) at (0.4,-0.8) {}
                edge [->] (v0);
        \end{tikzpicture}
        \ \ \Bigg) &&
        \\
        \\
        d_x\Big(\ 
        \begin{tikzpicture}[shorten >=1pt,>=angle 90,baseline={([yshift=-.5ex]current bounding box.center)}]
            \node[biw] (v0) at (0,0) {$0$\nodepart{lower}$\infty_1$};
            \node[invisible] (v1) at (-0.4,0.8) {}
                edge [<-] (v0);
            \node[] at (0,0.6) {$\scriptstyle\cdots$};
            %\node[invisible] (v2) at (0,0.9) {};
            %    edge [<-] (v0);
            \node[invisible] (v3) at (0.4,0.8) {}
                edge [<-] (v0);
            \node[invisible] (v4) at (-0.4,-0.8) {}
                edge [->] (v0);
            \node[] at (0,-0.6) {$\scriptstyle\cdots$};
            %\node[invisible] (v5) at (0,-0.9) {}
            %    edge [->] (v0);
            \node[invisible] (v6) at (0.4,-0.8) {}
                edge [->] (v0);
        \end{tikzpicture}\ \Big) \
        &=\Big(\frac{0}{\infty_1},\frac{0}{\infty_1}\Big)_x
        +\Big(\frac{0}{0},\frac{0}{\infty_1}\Big)_x
        +\Big(\frac{0}{\infty_1},\frac{0}{0}\Big)_x\\
        & - \Bigg( \ \ 
        \begin{tikzpicture}[shorten >=1pt,>=angle 90,baseline={(4ex,-0.5ex)}]
            \node[biw] (v0) at (0,0) {$0$\nodepart{lower}$\infty_1$};
            \node[biw] (new) at (1.5,-0.7) {$\infty_1$\nodepart{lower}$\infty_1$}
                edge [->] (v0);
            \node[invisible] (v1) at (-0.4,0.8) {}
                edge [<-] (v0);
            \node[] at (0,0.6) {$\scriptstyle\cdots$};
            %\node[invisible] (v2) at (0,0.9) {}
            %    edge [<-] (v0);
            \node[invisible] (v3) at (0.4,0.8) {}
                edge [<-] (v0);
            \node[invisible] (v4) at (-0.4,-0.8) {}
                edge [->] (v0);
            \node[] at (0,-0.6) {$\scriptstyle\cdots$};
            %\node[invisible] (v5) at (0,-0.9) {}
            %    edge [->] (v0);
            \node[invisible] (v6) at (0.4,-0.8) {}
                edge [->] (v0);
        \end{tikzpicture}
        \ +\ 
        \begin{tikzpicture}[shorten >=1pt,>=angle 90,baseline={(4ex,-0.5ex)}]
            \node[biw] (v0) at (0,0) {$0$\nodepart{lower}$0$};
            \node[biw] (new) at (1.5,-0.7) {$\infty_1$\nodepart{lower}$\infty_1$}
                edge [->] (v0);
            \node[invisible] (v1) at (-0.4,0.8) {}
                edge [<-] (v0);
            \node[] at (0,0.6) {$\scriptstyle\cdots$};
            %\node[invisible] (v2) at (0,0.9) {}
            %    edge [<-] (v0);
            \node[invisible] (v3) at (0.4,0.8) {}
                edge [<-] (v0);
            \node[invisible] (v4) at (-0.4,-0.8) {}
                edge [->] (v0);
            \node[] at (0,-0.6) {$\scriptstyle\cdots$};
            %\node[invisible] (v5) at (0,-0.9) {}
            %    edge [->] (v0);
            \node[invisible] (v6) at (0.4,-0.8) {}
                edge [->] (v0);
        \end{tikzpicture}
        \ \ \Bigg) &&
        %\\
        %\\
        \end{flalign*}
    \begin{flalign*}
        d_x\Big(\ 
        \begin{tikzpicture}[shorten >=1pt,>=angle 90,baseline={([yshift=-.5ex]current bounding box.center)}]
            \node[biw] (v0) at (0,0) {$\infty_1$\nodepart{lower}$0$};
            \node[invisible] (v1) at (-0.4,0.8) {}
                edge [<-] (v0);
            \node[] at (0,0.6) {$\scriptstyle\cdots$};
            %\node[invisible] (v2) at (0,0.9) {};
            %    edge [<-] (v0);
            \node[invisible] (v3) at (0.4,0.8) {}
                edge [<-] (v0);
            \node[invisible] (v4) at (-0.4,-0.8) {}
                edge [->] (v0);
            \node[] at (0,-0.6) {$\scriptstyle\cdots$};
            %\node[invisible] (v5) at (0,-0.9) {}
            %    edge [->] (v0);
            \node[invisible] (v6) at (0.4,-0.8) {}
                edge [->] (v0);
        \end{tikzpicture}\ \Big) \
        &=\Big(\frac{\infty_1}{0},\frac{\infty_1}{0}\Big)_x
        +\Big(\frac{0}{0},\frac{\infty_1}{0}\Big)_x
        +\Big(\frac{\infty_1}{0},\frac{0}{0}\Big)_x\\
        & - \Bigg( \ \ 
        \begin{tikzpicture}[shorten >=1pt,>=angle 90,baseline={(4ex,-0.5ex)}]
            \node[biw] (v0) at (0,0) {$\infty_1$\nodepart{lower}$0$};
            \node[biw] (new) at (1.5,0.7) {$\infty_1$\nodepart{lower}$\infty_1$}
                edge [<-] (v0);
            \node[invisible] (v1) at (-0.4,0.8) {}
                edge [<-] (v0);
            \node[] at (0,0.6) {$\scriptstyle\cdots$};
            %\node[invisible] (v2) at (0,0.9) {}
            %    edge [<-] (v0);
            \node[invisible] (v3) at (0.4,0.8) {}
                edge [<-] (v0);
            \node[invisible] (v4) at (-0.4,-0.8) {}
                edge [->] (v0);
            \node[] at (0,-0.6) {$\scriptstyle\cdots$};
            %\node[invisible] (v5) at (0,-0.9) {}
            %    edge [->] (v0);
            \node[invisible] (v6) at (0.4,-0.8) {}
                edge [->] (v0);
        \end{tikzpicture}
        \ +\ 
        \begin{tikzpicture}[shorten >=1pt,>=angle 90,baseline={(4ex,-0.5ex)}]
            \node[biw] (v0) at (0,0) {$0$\nodepart{lower}$0$};
            \node[biw] (new) at (1.5,0.7) {$\infty_1$\nodepart{lower}$\infty_1$}
                edge [<-] (v0);
            \node[invisible] (v1) at (-0.4,0.8) {}
                edge [<-] (v0);
            \node[] at (0,0.6) {$\scriptstyle\cdots$};
            %\node[invisible] (v2) at (0,0.9) {}
            %    edge [<-] (v0);
            \node[invisible] (v3) at (0.4,0.8) {}
                edge [<-] (v0);
            \node[invisible] (v4) at (-0.4,-0.8) {}
                edge [->] (v0);
            \node[] at (0,-0.6) {$\scriptstyle\cdots$};
            %\node[invisible] (v5) at (0,-0.9) {}
            %    edge [->] (v0);
            \node[invisible] (v6) at (0.4,-0.8) {}
                edge [->] (v0);
        \end{tikzpicture}
        \ \ \Bigg) &&
        \\
        \\
        d_x\Big(\ 
        \begin{tikzpicture}[shorten >=1pt,>=angle 90,baseline={([yshift=-.5ex]current bounding box.center)}]
            \node[biw] (v0) at (0,0) {$0$\nodepart{lower}$0$};
            \node[invisible] (v1) at (-0.4,0.8) {}
                edge [<-] (v0);
            \node[] at (0,0.6) {$\scriptstyle\cdots$};
            %\node[invisible] (v2) at (0,0.9) {};
            %    edge [<-] (v0);
            \node[invisible] (v3) at (0.4,0.8) {}
                edge [<-] (v0);
            \node[invisible] (v4) at (-0.4,-0.8) {}
                edge [->] (v0);
            \node[] at (0,-0.6) {$\scriptstyle\cdots$};
            %\node[invisible] (v5) at (0,-0.9) {}
            %    edge [->] (v0);
            \node[invisible] (v6) at (0.4,-0.8) {}
                edge [->] (v0);
        \end{tikzpicture}\ \Big) \
        &=\Big(\frac{0}{0},\frac{0}{0}\Big)_x &&
\end{flalign*}
\end{definition}
\begin{proposition}\label{prop:qGC}
    The inclusion $\qGC_k\hookrightarrow\mathsf{S}\wGC_k$ is a quasi-isomorphism.
\end{proposition}
\begin{proof}
    The proposition was proven in \cite{F}, so here we do only give a brief outline. Consider a filtration on $\mathsf{S}\wGC_k$ over the number of out-core vertices in a graph. The associated graded split over out-core graphs $\gamma$ as
    \begin{equation*}
        gr\ \mathsf{S}\wGC_k = \bigoplus_{\gamma}\mathsf{outCore}(\gamma)
    \end{equation*}
    where $\mathsf{outCore}(\gamma)$ is the complex of graphs with out-core $\gamma$. This complex further decompose as
    \begin{equation*}
        \mathsf{outCore}(\gamma)\cong\Big(\bigotimes_{x\in V(\gamma)\mathcal{T}_x^{out}}\Big)
    \end{equation*}
    where $\mathcal{T}_x^{out}$ is the complex of special-out trees attached to $x$. For two vertices $x,y\in V(\gamma)$, we have that $\mathcal{T}^{out}_x\cong\mathcal{T}_y^{out}$ if $|x|_{out}=|y|_{out}$, $|x|_{in}=|y|_{in}$ and $w_x^{in}=w_y^{in}$, the latter being invariant under the differential. So there is an isomorphism $\mathcal{T}_x^{out}\cong\mathcal{T}_{|x|_{out},|x|_{in},w_x^{in}}$ where $\mathcal{T}_{a,b,c}^{out}$ is a general complex of special-out vertices to a vertex with $a$ outgoing edges, $b$ incoming edges, and whose core-vertex has in-weight $c$.
    The only non-trivial cycles of $\mathcal{T}_{a,b,c}^{out}$ are the decorations of vertices stated in the definition of $\qwGC_k$.
\end{proof}
Let $\oqGC_k$ be the subcomplex of $\qGC_k$ of oriented graphs. This is also a subcomplex of $\mathsf{S}\owGC_k$.
\begin{proposition}
    The inclusion $\oqGC_k\hookrightarrow\qGC_k$ is a quasi-isomorphism.
\end{proposition}
\begin{proof}
    See proof of corollary \ref{cor:SinwGC}.
\end{proof}
\begin{definition}
    Let $\qqGC_k$ be the subcomplex of $\mathsf{S}\qwGC_k$ of graphs $\Gamma$ whose vertices $V(\Gamma)$ are independently decorated by the bi-weights $\frac{\infty_1}{\infty_1},\frac{\infty_1}{0},\frac{0}{\infty_1}$, $\frac{0}{0}$ and $\frac{\infty_2}{0}$ subject to the following conditions:
    \begin{enumerate}
        \item[$(1')$] If $x\in V(\gamma)$ is a univalent source, then it is decorated by $\frac{\infty_1}{\infty_1}$ or $\frac{\infty_2}{0}$.
        \item[$(1'')$] If $x\in V(\gamma)$ is a univalent target, then it is decorated by $\frac{\infty_1}{\infty_1}$.
        \item[$(2')$] If $x\in V(\gamma)$ is a bivalent source, then it is either decorated by $\frac{\infty_1}{\infty_1}$ or $\frac{0}{\infty_1}$.
        \item[$(2'')$] If $x\in V(\gamma)$ is an at least trivalent source, then it is either decorated by $\frac{\infty_1}{\infty_1}$, $\frac{0}{\infty_1}$ or $\frac{0}{0}$.
        \item[$(3)$] If $x\in V(\gamma)$ is a target, then it is either decorated by $\frac{\infty_1}{\infty_1}$ or $\frac{\infty_1}{0}$.
        \item[$(4)$] If $x\in V(\gamma)$ is a passing vertex, then it is either decorated by $\frac{\infty_1}{\infty_1}$, $\frac{\infty_1}{0}$ or $\frac{0}{\infty_1}$.
        \item[$(5)$] If $x\in V(\gamma)$ is a generic vertex, then it is either decorated by $\frac{\infty_1}{\infty_1}$, $\frac{\infty_1}{0}$, $\frac{0}{\infty_1}$ or $\frac{0}{0}$.
    \end{enumerate}
    The differential $d$ acts similarly as for the complex $\qGC_k$ but where new decorations are allowed. If $\Gamma\in\qqGC_k$, then $d(\Gamma)=\sum_{x\in V(\Gamma)}d_x(\Gamma)=\sum_{x\in V(\Gamma)}\delta_x(\Gamma)-\delta'_x(\Gamma)-\delta''_x(\Gamma)$. We describe below how the $\delta-$maps act on vertices with different bi-weights. A univalent source decorated by $\frac{\infty_1}{0}$ in any of the formulas is interpreted as decorated by $\frac{\infty_2}{0}$.
    \begin{flalign*}
        %Term 1
        \delta_x\Big(\ 
        \begin{tikzpicture}[shorten >=1pt,>=angle 90,baseline={([yshift=-.5ex]current bounding box.center)}]
            \node[biw] (v0) at (0,0) {$\infty_1$\nodepart{lower}$\infty_1$};
            \node[invisible] (v1) at (-0.4,0.8) {}
                edge [<-] (v0);
            \node[] at (0,0.6) {$\scriptstyle\cdots$};
            %\node[invisible] (v2) at (0,0.9) {};
            %    edge [<-] (v0);
            \node[invisible] (v3) at (0.4,0.8) {}
                edge [<-] (v0);
            \node[invisible] (v4) at (-0.4,-0.8) {}
                edge [->] (v0);
            \node[] at (0,-0.6) {$\scriptstyle\cdots$};
            %\node[invisible] (v5) at (0,-0.9) {}
            %    edge [->] (v0);
            \node[invisible] (v6) at (0.4,-0.8) {}
                edge [->] (v0);
        \end{tikzpicture}\ \Big) \
        &=\Big(\frac{\infty_1}{\infty_1},\frac{\infty_1}{\infty_1}\Big)_x
        +\Big(\frac{0}{\infty_1},\frac{\infty_1}{\infty_1}\Big)_x
        +\Big(\frac{\infty_1}{0},\frac{\infty_1}{\infty_1}\Big)_x
        +\Big(\frac{\infty_1}{\infty_1},\frac{0}{\infty_1}\Big)_x
        +\Big(\frac{\infty_1}{\infty_1},\frac{\infty_1}{0}\Big)_x\\
        &+\Big(\frac{0}{0},\frac{\infty_1}{\infty_1}\Big)_x
        +\Big(\frac{\infty_1}{0},\frac{0}{\infty_1}\Big)_x
        +\Big(\frac{0}{\infty_1},\frac{\infty_1}{0}\Big)_x
        +\Big(\frac{\infty_1}{\infty_1},\frac{0}{0}\Big)_x\\
        %Term 2
        \delta'_x\Big(\ 
\ \Big) \ = 0
\end{flalign*}
\end{definition}
\begin{proposition}
    The inclusion $\qqGC_k\hookrightarrow\mathsf{S}\qwGC_k$ is a quasi-isomorphism.
\end{proposition}
\begin{proof}
    Consider a filtration of $\mathsf{S}\qwGC_k$ over the number of out-core vertices in a graph. Then the associated graded split over out-core graphs
    \begin{equation*}
        gr \ \mathsf{S}\qwGC_k=\bigoplus_{\gamma} \mathsf{outQCore}(\gamma)
    \end{equation*}
    where $\mathsf{outQCore}(\gamma)$ is the complex of graphs with out-core $\gamma$. This complex further split as
    \begin{equation*}
        \mathsf{outQCore}(\gamma)\cong \Big(\bigotimes_{x\in V(\gamma)}q\mathcal{T}_x^{out}\Big)^{\mathrm{Aut}(\gamma)}
    \end{equation*}
    where $q\mathcal{T}_x^{out}$ is the complex of out-tree graphs attached to $x$. We note that $q\mathcal{T}_x^{out}=\mathcal{T}_x^{out}$ when $x$ is not a source or has in-weight $\infty_1$, and so these result follow from proposition \ref{prop:qGC}.
    Assume $|x|_{in}=0$ and $w_x^{in}=0$ and consider the three cases $|x|_{out}=1$, $|x|_{out}=2$ and $|x|_{out}\geq 3$. By the same arguments as in proposition \ref{prop:SinpwGC}, we compute the cohomology of $q\mathcal{T}_x^{out}$ and see that it agrees with the proposition.
\end{proof}
Let $\oqqGC_k$ be the subcomplex of $\qqGC_k$ of oriented graphs. This is also a subcomplex of $\mathsf{S}\oqwGC_k$.
\begin{proposition}
    The inclusion $\oqqGC_k\hookrightarrow\mathsf{S}\oqwGC_k$ is a quasi-isomorphism.
\end{proposition}
\begin{proof}
    See proof of corollary \ref{cor:SinwGC}.
\end{proof}
\begin{definition}
    Let $\pqGC_k$ be the subcomplex of $\mathsf{S}\pwGC_k$ of graphs $\Gamma$ whose vertices $V(\Gamma)$ are independently decorated by the bi-weights $\frac{\infty_1}{\infty_1},\frac{\infty_1}{0},\frac{0}{\infty_1}$, $\frac{0}{0}$, $\frac{\infty_2}{0}$ and $\frac{0}{\infty_2}$ subject to the following conditions:
    \begin{enumerate}
        \item[1'] If $x\in V(\gamma)$ is a univalent source, then it is decorated by $\frac{\infty_1}{\infty_1}$ or $\frac{\infty_2}{0}$.
        \item[1''] If $x\in V(\gamma)$ is a univalent target, then it is decorated by $\frac{\infty_1}{\infty_1}$ or $\frac{0}{\infty_2}$.
        \item[2'] If $x\in V(\gamma)$ is a bivalent source, then it is either decorated by $\frac{\infty_1}{\infty_1}$ or $\frac{0}{\infty_1}$.
        \item[2''] If $x\in V(\gamma)$ is an at least trivalent source, then it is either decorated by $\frac{\infty_1}{\infty_1}$, $\frac{0}{\infty_1}$ or $\frac{0}{0}$.
        \item[3'] If $x\in V(\gamma)$ is a bivalent target, then it is either decorated by $\frac{\infty_1}{\infty_1}$ or $\frac{\infty_1}{0}$.
        \item[3'] If $x\in V(\gamma)$ is an at least trivlanet target, then it is either decorated by $\frac{\infty_1}{\infty_1}$, $\frac{\infty_1}{0}$ or $\frac{0}{0}$.
        \item If $x\in V(\gamma)$ is a passing vertex, then it is either decorated by $\frac{\infty_1}{\infty_1}$, $\frac{\infty_1}{0}$ or $\frac{0}{\infty_1}$.
        \item If $x\in V(\gamma)$ is a generic vertex, then it is either decorated by $\frac{\infty_1}{\infty_1}$, $\frac{\infty_1}{0}$, $\frac{0}{\infty_1}$ or $\frac{0}{0}$.
    \end{enumerate}
    The differential $d$ acts similarly as for the complex $\qqGC_k$ but where new decorations are allowed. If $\Gamma\in\pqGC_k$, then $d(\Gamma)=\sum_{x\in V(\Gamma)}d_x(\Gamma)=\sum_{x\in V(\Gamma)}\delta_x(\Gamma)-\delta'_x(\Gamma)-\delta''_x(\Gamma)$. We describe below how the $\delta-$maps act on vertices with different bi-weights. A univalent source decorated by $\frac{\infty_1}{0}$ in any of the formulas is interpreted as decorated by $\frac{\infty_2}{0}$. Similarly a univalent target decorated by $\frac{0}{\infty_1}$ is interpreted as the decoration $\frac{0}{\infty_2}$.
    \begin{flalign*}
        %Term 1
        \delta_x\Big(\ 
        \begin{tikzpicture}[shorten >=1pt,>=angle 90,baseline={([yshift=-.5ex]current bounding box.center)}]
            \node[biw] (v0) at (0,0) {$\infty_1$\nodepart{lower}$\infty_1$};
            \node[invisible] (v1) at (-0.4,0.8) {}
                edge [<-] (v0);
            \node[] at (0,0.6) {$\scriptstyle\cdots$};
            %\node[invisible] (v2) at (0,0.9) {};
            %    edge [<-] (v0);
            \node[invisible] (v3) at (0.4,0.8) {}
                edge [<-] (v0);
            \node[invisible] (v4) at (-0.4,-0.8) {}
                edge [->] (v0);
            \node[] at (0,-0.6) {$\scriptstyle\cdots$};
            %\node[invisible] (v5) at (0,-0.9) {}
            %    edge [->] (v0);
            \node[invisible] (v6) at (0.4,-0.8) {}
                edge [->] (v0);
        \end{tikzpicture}\ \Big) \
        &=\Big(\frac{\infty_1}{\infty_1},\frac{\infty_1}{\infty_1}\Big)_x
        +\Big(\frac{0}{\infty_1},\frac{\infty_1}{\infty_1}\Big)_x
        +\Big(\frac{\infty_1}{0},\frac{\infty_1}{\infty_1}\Big)_x
        +\Big(\frac{\infty_1}{\infty_1},\frac{0}{\infty_1}\Big)_x
        +\Big(\frac{\infty_1}{\infty_1},\frac{\infty_1}{0}\Big)_x\\
        &+\Big(\frac{0}{0},\frac{\infty_1}{\infty_1}\Big)_x
        +\Big(\frac{\infty_1}{0},\frac{0}{\infty_1}\Big)_x
        +\Big(\frac{0}{\infty_1},\frac{\infty_1}{0}\Big)_x
        +\Big(\frac{\infty_1}{\infty_1},\frac{0}{0}\Big)_x\\
        %Term 2
        \delta'_x\Big(\ 
\ \Big) \ = 0
\end{flalign*}
\end{definition}
\begin{proposition}
    The inclusion $\pqGC_k\hookrightarrow\mathsf{S}\pwGC_k$ is a quasi-isomorphism.
\end{proposition}
\begin{proof}
    Consider the filtration of $\mathsf{S}\pwGC_k$ over the number of out-core vertices. The associated graded complex split over out-core graphs as
    \begin{equation*}
        gr\ \mathsf{S}\pwGC_k = \bigoplus_{\gamma}\mathsf{outPCore}(\gamma)
    \end{equation*}
    where $\mathsf{outPCore}(\gamma)$ is the complex of graphs with out-core $\gamma$. This complex further decomopose as
    \begin{equation*}
        \mathsf{outPCore}(\gamma)\cong \Big(\bigotimes_{x\in V(\gamma)}p\mathcal{T}_x^{out}\Big)^{\mathrm{Aut}(\gamma)}
    \end{equation*}
    where $p\mathcal{T}_x^{out}$ is the complex of special-out trees attached to $x$. If $x$ is a passing or generic vertex, then $p\mathcal{T}_x^{out}=\mathcal{T}_x^{out}$. If $x$ is a source, then $p\mathcal{T}_x^{out}=q\mathcal{T}_x^{out}$.
    The last case to consider is when $|x|_{out}=0$. In the three cases $|x|_{out}=1$, $|x|_{out}=2$ and $|x|_{out}\geq 3$ we can show with the same arguments as in proposition \ref{prop:SinwGC} that the cohomology of $p\mathcal{T}_x^{out}$ agrees with the proposition statement.
\end{proof}
\begin{proposition}
    The inclusion $\opqGC_k\hookrightarrow\mathsf{S}\opqGC_k$ is a quasi-isomorphism.
\end{proposition}
\begin{proof}
    See proof of corollary \ref{cor:SinwGC}.
\end{proof}
Let $\qGC_k^+$ be the subcomplex of $\mathsf{S}\wGC_k$ of graphs which have either at least one vertex with bi-weight $\frac{\infty_1}{\infty_1}$, or a pair of vertices decorated by $\frac{\infty_1}{0}$ and $\frac{0}{\infty_1}$ respectively. This is also a subcomplex of $\mathsf{S}\wGC_k^+$. The following proposition was proven in \cite{F}.
\begin{proposition}
    The inclusion $\qGC_k^+\hookrightarrow\mathsf{S}\wGC_k^+$ is a quasi-isomorphism.
\end{proposition}
Let $\qqGC_k^+$ be the subcomplex of $\mathsf{S}\qwGC_k$ of graphs with at least one vertex decorated by $\frac{\infty_1}{\infty_1}$, or a vertex decorated by $\frac{\infty_1}{0}$ (or $\frac{\infty_2}{0}$). This is also a subcomplex of $\wGC_k^+$.
\begin{proposition}
    The inclusion $\qqGC_k^+\hookrightarrow\mathsf{S}\qwGC_k$ is a quasi-isomorphism.
\end{proposition}
\begin{proof}
    The proof is similar to that of proposition \ref{prop:SinwGC+}.
\end{proof}
\section{Further reductions of the derivation complexes and proofs of the Main theorems}
\subsection{$\frac{0}{0}$-decorations}
The complex $\qGC_k$ split as
\begin{align*}
    \qGC_k=\qGC_k^0\oplus\qGC_k^*
\end{align*}
where $\qGC_k^0$ is the complex of graphs whose all vertices are decorated by $\frac{0}{0}$, and $\qGC_k^*$ is the complex of graphs with at least one vertex with positive out-weight or in-weight. Similarly $\qqGC_k=\qqGC_k^0\oplus\qqGC_k^*$ and $\pqGC_k=\pqGC^0_k\oplus\pqGC_k^*$. In \cite{F} we showed that $\qGC_k^0\cong\dGC_k/\dGC_k^{s+t}$.
\begin{proposition}
    Let $\dGC_k^{2}\subset\dGC_k$ be the subcomplex of graphs with at least one bivalent vertex and let $\dGC_k^{2+t}$ be the complex of graphs with at least one bivalent vertex or one target vertex. Then $\qqGC_k^0\cong\dGC_k/\dGC_k^{2+t}$ and $\pqGC_k\cong\dGC_k/\dGC_k^2$.
\end{proposition}
\begin{proof}
    This is an easy observation.
\end{proof}
Let $\qGC_k^{0,*}\subseteq\qGC_k^*$ be the subcomplex of graphs with at least one vertex decorated by $\frac{0}{0}$ and $\tGC_k^*:=\qGC_k^*/\qGC_k^{0,*}$. The complex $\qGC_k^{0,*}$ is acyclic, and hence the projection $\qGC_k^*\rightarrow\tGC^{*}_k$ is a quasi-isomorphism \cite{F}. Let $\qqGC_k^{0,*}$ and $\pqGC_k^{0,*}$ be the corresponding subcomplexes and $\qtGC^*_k$ and $\ptGC_k^*$ the corresponding quotients.
\begin{proposition}
    The projections $\qqGC_k^*\rightarrow\qtGC_k^*$ and $\pqGC_k^*\rightarrow\ptGC_k^*$ are quasi-isomophisms.
\end{proposition}
\begin{proof}
    It is enough to show that $\qqGC_k^{0,*}$ and $\pqGC_k^{0,*}$ are acyclic. The argument is equivalent to the proof of proposition 6.1.1 in \cite{F}, showing that $\qGC_k^{0,*}$ is acyclic.
\end{proof}
There are similar propositions for $\qqGC_k^+$ and the oriented subcomplexes with the same proofs. Denote their corresponding quotient complex of graphs with three decorations $\frac{\infty_1}{\infty_1}$, $\frac{\infty_1}{0}$ and $\frac{0}{\infty_1}$ by $\qtGC_k^+$, $\mathsf{otGC}_k$, $\mathsf{otQGC}_k$ and $\mathsf{otPGC}_k$.
\subsection{Monodecorated graphs}
The following graphs are in $\qtGC_k^*$. Let $\Gamma$ be an undecorated graph without univalent vertices. Let $\Gamma(\frac{\infty_1}{0})$ denote the graph $\Gamma$ where all vertices are decorated by $\frac{\infty_1}{0}$. Similarly define $\Gamma(\frac{0}{\infty_1})$. Let $\Gamma(\omega)$ be sum of all possible decorations of $\Gamma$ such that at least one vertex is decorated by $\frac{\infty_1}{\infty_1}$ or a pair of vertices is decorated by $\frac{\infty_1}{0}$ and $\frac{0}{\infty_1}$. Let $\mathcal{C}^{\geq2}$ denote the subspace (not subcomplex) of graphs on the form $\Gamma(\frac{\infty_1}{0})$, $\Gamma(\frac{0}{\infty_1})$ and $\Gamma(\omega)$ for $\Gamma$ with no univlanet vertices.
Recall that the differential decompose as $d=d_s+d_u$ where $d_u$ increase the number of univalent vertices of a graph by one, and $d_s$ split vertices without creating univalent ones. It is easy to see that $d_u(\Gamma(\omega)+\Gamma(\frac{\infty_1}{0})+\Gamma(\frac{0}{\infty_1}))=0$. Further note that $\Gamma(\frac{0}{\infty_1})=0$ if $\Gamma$ contain a target vertex. Hence we gather that
\begin{equation*}
    d_u\Gamma(\omega)=
    \begin{cases}
        -d_u\Gamma(\frac{\infty_1}{0}) & \text{ if there is a target vertex in $\Gamma$.}\\
        -d_u\Gamma(\frac{\infty_1}{0})-d_u\Gamma(\frac{0}{\infty_1}) & \text{ if there is no target vertex in $\Gamma$.}
    \end{cases}
\end{equation*}
Let $\mathcal{C}^1(\frac{\infty_1}{0})$ be the subcomplex of $\tGC_k^*$ of graphs only decorated by $\frac{\infty_1}{0}$ and where the only univalent vertices are attached to at non-antenna vertices and are of the special type 
$\begin{tikzpicture}[shorten >=1pt,node distance=1.2cm,auto,>=angle 90,baseline=-0.1cm]
            \node[biw] (a) at (0,0) {$\infty_1$\nodepart{lower}$\infty_1$};
            \node[] (d) at (0,-1) {}
                edge [->] (a);
\end{tikzpicture}
+ (-1)^{k+1}
\begin{tikzpicture}[shorten >=1pt,node distance=1.2cm,auto,>=angle 90,baseline=-0.1cm]
            \node[biw] (a) at (0,0) {$\infty_2$\nodepart{lower}$0$};
            \node[] (d) at (0,-1) {}
                edge [<-] (a);
\end{tikzpicture}$. Let $\mathcal{C}^1(\frac{0}{\infty_1})$ be the similar subcomplex of $\tGC_k^*$ of graphs only decorated by $\frac{0}{\infty_1}$ with univalent vertices of the type $\begin{tikzpicture}[shorten >=1pt,node distance=1.2cm,auto,>=angle 90,baseline=-0.1cm]
            \node[biw] (a) at (0,0) {$\infty_1$\nodepart{lower}$\infty_1$};
            \node[] (d) at (0,-1) {}
                edge [<-] (a);
\end{tikzpicture}
+
\begin{tikzpicture}[shorten >=1pt,node distance=1.2cm,auto,>=angle 90,baseline=-0.1cm]
            \node[biw] (a) at (0,0) {$\infty_2$\nodepart{lower}$0$};
            \node[] (d) at (0,-1) {}
                edge [<-] (a);
\end{tikzpicture}$. It is easy to see that the space $\mathcal{C}^{\geq2}\oplus\mathcal{C}^1(\frac{\infty_1}{0})\oplus\mathcal{C}^1(\frac{0}{\infty_1})$ is a subspace of $\tGC_k^*$. Denote this subcomplex by $\qmGC_k$.
\begin{proposition}\label{prop:qmGC}
    The inclusion $\qmGC_k\hookrightarrow\qtGC_k^*$ is a quasi-isomorphism.
\end{proposition}

\tikzstyle{mblack}=[circle,fill=black!,inner sep=0pt,minimum size=1mm]
\newcommand{\wEd}[2]{
    \begin{tikzpicture}[shorten <=1pt,>=stealth,semithick,baseline=-0.1cm
    %{([yshift=-.5ex]current bounding box.center)}
    ]
    \node[mblack] (l) at (0,0) {};
    \node[mblack] (r) at (0.9,0) {};
    % End edges
    \node[] at (-0.2,0.3) {}
        edge [-] (l);
    \node[] at (-0.35,0) {}
        edge [-] (l);
    \node[] at (-0.2,-0.3) {}
        edge [-] (l);
    \draw [#2,snake=snake,
segment amplitude=0.6mm,
segment length=2.2mm,
line after snake=1.3mm
] (l) to (r);
    \node[] (m) at (0.45,0.2) {$\scriptstyle #1$};
    \end{tikzpicture}
}

\newcommand{\sBi}[0]{
    \begin{tikzpicture}[shorten <=1pt,>=stealth,semithick,baseline=-0.1cm
    %{([yshift=-.5ex]current bounding box.center)}
    ]
        \node[mblack] (l) at (-0.6,0) {};
        \node[mblack] (r) at (0.6,0) {};
        \node[mblack] (m) at (0,0) {}
            edge [->] (r)
            edge [->] (l);
        \node[] (ll) at (-1.2,0) {}
            edge[->] (l);
        \node[] (dots) at (-1.3,0) {$\cdots$};
        \node[mblack] (end) at (-2,0){};
        \node[] at (-1.4,0) {}
            edge [-] (end);
        % End edges
        \node[] at (-2.2,0.3) {}
        edge [-] (end);
        \node[] at (-2.35,0) {}
            edge [-] (end);
        \node[] at (-2.2,-0.3) {}
            edge [-] (end);
        \draw[snake=brace] (-2,0.2) -- (0.6,0.2);
        \node[] at (-0.7,0.5) {$\scriptstyle m$};
    \end{tikzpicture}
}

\newcommand{\sEd}[1]{
    \begin{tikzpicture}[shorten <=1pt,>=stealth,semithick,baseline=-0.1cm
    %{([yshift=-.5ex]current bounding box.center)}
    ]
        \node[mblack] (r) at (0.6,0) {};
        \node[mblack] (l) at (0,0) {}
            edge [#1] (r);
        % End edges
        \node[] at (-0.2,0.3) {}
        edge [-] (l);
        \node[] at (-0.35,0) {}
            edge [-] (l);
        \node[] at (-0.2,-0.3) {}
            edge [-] (l);
    \end{tikzpicture}
}

\begin{proof}
    First consider the subcomplex $\qtGC_k^{1,\frac{\infty_1}{0},\frac{0}{\infty_1}}\subseteq \qtGC_k^*$ of graphs with at least one univalent vertex and a pair of non-univalent vertices decorated by $\frac{\infty_1}{0}$ and $\frac{0}{\infty_1}$ respectively. One can show that this complex is acyclic by considering a filtration over the number of non-passing vertices. Hence the quotient complex $\mathsf{s}_1\qtGC_k$ is quasi-isomorphic to $\qtGC_k^*$. One also notes that $\qmGC_k$ is also a subcomplex of this quotient complex.
    Next consider the subcomplex $\mathsf{s}_2\qtGC_k$ of $\mathsf{s}_1\qtGC_k$ of graphs with non-univalent vertices as in $\qmGC_k$, and all other types of graphs with at least one univalent vertex as in $\mathsf{s}_1\qtGC_k$. The quotient complex consists of graphs with no univalent vertices, and where at least two vertices have different decorations. This complex is acyclic, seen by consider a filtration over the number of non-passing vertices.
    Next, consider the subcomplex $\mathsf{s}_3\qtGC_k$ of $\mathsf{s}_2\qtGC_k$ containing the same graphs having no univalent vertices, and additionally graphs with at least one univalent vertex but no non-univalent vertices decorated by $\frac{\infty_1}{\infty_1}$. The quotient complex is generated by graphs with at least one univalent vertex, and where all non-univalent vertices are decorated by $\frac{\infty_1}{\infty_1}$. This complex is acyclic, seen by considering a filtration over the number of non-passing vertices.
    Let $\mathsf{s}_4\qtGC_k$ be the subcomplex of $\mathsf{s}_3\qtGC_k$ containing the same graphs having no univalent vertices, and additionally graphs with univalent vertices on the following forms: Let $\Gamma$ be a graph without univalent vertices. Then we consider graphs $\Gamma(\frac{\infty_1}{0})$ where at least one vertex has the univalent vertex 
    $\begin{tikzpicture}[shorten >=1pt,node distance=1.2cm,auto,>=angle 90,baseline=-0.1cm]
            \node[biw] (a) at (0,0) {$\infty_1$\nodepart{lower}$\infty_1$};
            \node[] (d) at (0,-1) {}
                edge [->] (a);
    \end{tikzpicture}
    + (-1)^{k+1}
    \begin{tikzpicture}[shorten >=1pt,node distance=1.2cm,auto,>=angle 90,baseline=-0.1cm]
                \node[biw] (a) at (0,0) {$\infty_2$\nodepart{lower}$0$};
                \node[] (d) at (0,-1) {}
                    edge [<-] (a);
    \end{tikzpicture}$
    attached. Similarly we also consider graphs $\Gamma(\frac{0}{\infty_1})$ where at least one vertex has the univalent vertex
    $\begin{tikzpicture}[shorten >=1pt,node distance=1.2cm,auto,>=angle 90,baseline=-0.1cm]
            \node[biw] (a) at (0,0) {$\infty_1$\nodepart{lower}$\infty_1$};
            \node[] (d) at (0,-1) {}
                edge [<-] (a);
    \end{tikzpicture}
    +
    \begin{tikzpicture}[shorten >=1pt,node distance=1.2cm,auto,>=angle 90,baseline=-0.1cm]
                \node[biw] (a) at (0,0) {$\infty_2$\nodepart{lower}$0$};
                \node[] (d) at (0,-1) {}
                    edge [<-] (a);
    \end{tikzpicture}$
    attached. Let $\mathsf{s}_4\qtGC_k^1$ be the corresponding complex of graphs with at least one univalent vertex and similarly for $\mathsf{s}_3\qtGC_k$. We get the commutative diagram
    \begin{equation*}
        \xymatrix{& 0 \ar[d] & 0 \ar[d] & 0 \ar[d] &\\
            0 \ar[r] & \mathsf{s}_4\qtGC_k^1 \ar[r] \ar[d] & \mathsf{s}_4\qtGC_k \ar[r] \ar[d] & \mathsf{s}_4\qtGC_k^{\geq2} \ar[r] \ar[d] & 0 \\
            0 \ar[r] & \mathsf{s}_3\qtGC_k^1 \ar[r] \ar[d] & \mathsf{s}_3\qtGC_k \ar[r] \ar[d] & \mathsf{s}_3\qtGC_k^{\geq2} \ar[r] \ar[d] & 0 \\
            0 \ar[r] & A \ar[r] \ar[d] & B \ar[r] \ar[d] & 0 \ar[r] \ar[d] & 0\\
            & 0 & 0 & 0 &}
    \end{equation*}
    By showing that the inclusion $\mathsf{s}_4\qtGC_k^1\hookrightarrow\mathsf{s}_3\qtGC_k^1$ is a quasi-isomorphism, we also show that the inclusion for the full complexes is a quasi-isomorphism.
    The complex $\mathsf{s}_3\qtGC_k^1$ split as $\mathsf{s}_3\qtGC_k^1(\frac{\infty_1}{0})\oplus\mathsf{s}_3\qtGC_k^1(\frac{0}{\infty_1})$ where each complex contains the graphs where all non-univalent vertices are decorated by $\frac{\infty_1}{0}$ and $\frac{0}{\infty_1}$ respectively.
    We consider a vertex an \textit{antenna-vertex} if it is either univalent, or becomes univalent after iterative removal of univalent vertices in a graph. Consider a filtration of both complexes over the number of non-bivalent non-antenna vertices.
    In the associated graded the differential acts on by creating bivalent antenna vertices. All graphs of $gr\ \mathsf{s}_4\qtGC_k^1$ are non-trivial cycles, and so we need to show that the cohomology of $gr\ \mathsf{s}_3\qtGC_k^1$ is generated by the same graphs.
    Consider first a filtration over the number of non-passing vertices. We see that $H(gr\ \mathsf{s}_3\qtGC_k^1(\frac{0}{\infty_1}))$ is generated exactly by the graphs in $\mathsf{s}_4\qtGC_k^1(\frac{0}{\infty_1})$.
    We further get that $H(gr\ \mathsf{s}_3\qtGC_k^1(\frac{\infty_1}{0}))$ is generated by graphs with no passing antenna-vertices and univalent vertices of two types, either $\begin{tikzpicture}[shorten >=1pt,node distance=1.2cm,auto,>=angle 90,baseline=-0.1cm]
                \node[biw] (a) at (0,0) {$\infty_1$\nodepart{lower}$\infty_1$};
                \node[] (d) at (0,-1) {}
                    edge [->] (a);
    \end{tikzpicture}$
    or
    $\begin{tikzpicture}[shorten >=1pt,node distance=1.2cm,auto,>=angle 90,baseline=-0.1cm]
                \node[biw] (a) at (0,0) {$\infty_2$\nodepart{lower}$0$};
                \node[] (d) at (0,-1) {}
                    edge [<-] (a);
    \end{tikzpicture}$. Note that there is only one possible decoration for a univalent vertex, depending on if it is a source or a target.
    On the next page of the spectral sequence, the antenna-vertices form trees, where chains of bivalent vertices are composed of edges alternating directions. Call such an edge a \textit{zig-zag} edge. The case when zig-zag edges are adjacent to two at least trivalent vertices and how the differential acts on them has been studied in \cite{Z2} (there called skeleton edges). The more important case here is the case when a zig-zag edge is adjacent to a univalent vertex. They are on the form $\wEd{m}{->}:=\sBi$ where $m$ is the number of normal edges in the zig-zag edge and its direction is the direction of the last edge attached to the univalent vertex. The differential acts on zig-zag edges as
    \begin{equation*}
    \begin{tabular}{lcccccc}
         $d\big( \wEd{2m+1}{->} \big) $& $=$ & $(-1)^{k+1}$ & $\wEd{2m+2}{->}$ & $+$ & $(-1)^{k+1}$ & $\wEd{2m+2}{-<}$ \\
         $d\big( \wEd{2m+1}{-<} \big)$ & $=$ & $-$ & $\wEd{2m+2}{->}$ & $-$ & & $\wEd{2m+2}{-<}$\\
         $d\big( \wEd{2m}{->} \big) $& $=$ & & $\wEd{2m+1}{->}$ & $+$ & $(-1)^{k+1}$ & $\wEd{2m+1}{-<}$ \\
         $d\big( \wEd{2m}{-<} \big)$ & $=$ & $-$ & $\wEd{2m+1}{->}$ & $+$ & $(-1)^{k}$ & $\wEd{2m+1}{-<}$
    \end{tabular}
    \end{equation*}
    We see that its cohomology is generated by $\sEd{->}+(-1)^{k+1}\sEd{<-}$.
    The complex decomposes over graphs with the same skeleton graphs, i.e the resulting graph obtained by replacing sequences of bivalent vertices with one single edge. These complexes decompose as a tensor complexes over the edges of the skeleton graphs, similar to the compositions in proposition \ref{prop:SinwGC}, we get that the cohomology of $gr \ \mathsf{s}_3\qtGC_k^1(\frac{\infty_1}{0}))$ is generated the desired one.
\end{proof}
Similarly define $\qmGC_k^+$ as a subcomplex of $\qmGC_k^*$, but where graphs without univalent vertices on the form $\Gamma(\frac{0}{\infty_1})$ are excluded.
\begin{corollary}
    The inclusion $\qmGC_k^+\hookrightarrow \qtGC_k^+$ is a quasi-isomorphism.
\end{corollary}
\begin{proof}
    By considering the short exact sequence induced by this inclusion, we get a quotient complex that is equal to the induced quotient by the inclusion $\qmGC_k^*\hookrightarrow \qtGC_k^*$. By the proposition, this complex is acyclic, giving the corollary. 
\end{proof}
Let $\mathsf{oqmGC}_k\subseteq\qmGC_k$ and $\mathsf{opmGC}_k\subseteq\pmGC_k$ be the subcomplexes of oriented graphs.
\begin{proposition}
    The inclusions $\mathsf{oqmGC}_k\hookrightarrow\mathsf{oqtGC}_k$ and $\mathsf{opmGC}_k\hookrightarrow\mathsf{optGC}_k$ are quasi-isomorphisms.
\end{proposition}
\begin{proof}
    The proof is the same as proposition \ref{prop:qmGC}, noting that all the arguments are independent of the orientation of a graph.
\end{proof}
\subsection{Cohomology results}
Let $\mathcal{C}(\frac{\infty_1}{0})$ and $\mathcal{C}(\frac{0}{\infty_1})$ be the subcomplexes of $\qmGC_k^*$ of graphs whose non-univalent vertices are decorated by $\frac{\infty_1}{0}$ and $\frac{0}{\infty_1}$ respectively.
\begin{proposition}\label{prop:qmGC(omega)}
    Consider the short exact sequence
    \begin{equation*}
        \xymatrix{ 0 \ar[r] & \mathcal{C}(\frac{\infty_1}{0})\oplus\mathcal{C}(\frac{0}{\infty_1}) \ar[r] & \qmGC_k^* \ar[r] & \qmGC_k(\omega) \ar[r] & 0}.
    \end{equation*}
    Then
    \begin{enumerate}
        \item The complexes $\mathcal{C}(\frac{\infty_1}{0})$ and $\mathcal{C}(\frac{0}{\infty_1})$ are acyclic.
        \item The complex $\qmGC_k(\omega)$ is isomorphic to $\dGC_k$.
    \end{enumerate}
    In particular there is a quasi-isomorphism $\qmGC_k^*\rightarrow \dGC_k$.
\end{proposition}
\begin{proof}
    The proof is the same as the proof of proposition 6.3.1 in \cite{F}.
\end{proof}
\begin{theorem}[Main theorem III]\label{thm:qmGC*}
    Let $\dGC_k\rightarrow \qwGC_k^*$ be the map where a graph $\Gamma$ is mapped to the sum of all possible bi-weights to put on $\Gamma$ excluding the decoration with only $\frac{0}{0}$. Then this map is a quasi-isomorphism.
\end{theorem}
\begin{proof}
    This follows by noting that the restrictions of this map to the complexes of the zig-zag $\qwGC_k^*\leftarrow\qqGC_k^*\rightarrow\qtGC_k^*\leftarrow\qmGC_k^*\rightarrow\qmGC_k(\omega)$ are quasi-isomorphisms.
\end{proof}
\begin{proposition}
    Let $\mathsf{oqmGC}_k(\omega)$ be the subcomplex of $\qmGC_k(\omega)$ of oriented graphs. Then the induced projection $\mathsf{oqmGC}_k\rightarrow\mathsf{oqmGC}_k(\omega)$ is a quasi-isomorphism. Furthermore $\mathsf{oqmGC}_k(\omega)\cong\oGC_k$.
\end{proposition}
\begin{proof}
    The same arguments as in proposition \ref{prop:qmGC(omega)}.
\end{proof}
\begin{theorem}[Main theorem I]
    The restriction $\oGC_k\rightarrow\oqwGC_k$ of the map in \ref{thm:qmGC*} is a quasi-isomorphism.
\end{theorem}
\begin{proof}
    See \ref{thm:qmGC*}.
\end{proof}
Let $\mathcal{C}(\frac{\infty_1}{0})$ be the subcomplex of $\qmGC_k^+$ of graphs where all non-univalent vertices are decorated with $\frac{\infty_1}{0}$, and let $\mathcal{Q}:=\qmGC_k^+/\mathcal{C}(\frac{\infty_1}{0})$.
\begin{proposition}\label{prop:redqmGC}
    The complex $\mathcal{C}(\frac{\infty_1}{0})$ is acyclic. In particular the projection $\qmGC_k^+\rightarrow\mathcal{Q}$ is a quasi-isomorphism.
\end{proposition}
\begin{proof}
    The proof follows by considering a filtration over the number of non-univalent vertices.
\end{proof}
\begin{proposition}
    Let $\mathcal{Q}^t$ be the subcomplex of $\mathcal{Q}$ of graphs with at least one target vertex when excluding antenna-vertices. Then the inclusion is a quasi-isomorphism. Furthermore the complex $\mathcal{Q}^t$ is isomorphic to $\dGC_k^t$.
\end{proposition}
\begin{proof}
    One shows that the quotient $\mathcal{Q}/\mathcal{Q}^t$ is acyclic by the same argument as in proposition \ref{prop:redqmGC}. The second part follows by inspection.
\end{proof}
\begin{theorem}[Main theorem II]
    The restriction $\dGC_k^t\rightarrow\qwGC_k^+$ of the map from theorem \ref{thm:qmGC*} is a quasi-isomorphism.
\end{theorem}
\begin{proof}
    See \ref{thm:qmGC*}.
\end{proof}
Let $\mathcal{C}$ be the subcomplex of $\ptGC_k^*$ of graphs with at least one univalent vertex, and graphs with without univalent on the form $\Gamma(\frac{\infty_1}{0})$ and $\Gamma(\frac{0}{\infty_1})$.
\begin{proposition}\label{prop:ptGC(gamma)}
    Consider the short exact sequence
    \begin{equation*}
        \xymatrix{0 \ar[r] & \mathcal{C} \ar[r] & \ptGC_k^* \ar[r] & \ptGC_k(\omega) \ar[r] & 0 }
    \end{equation*}
    where $\ptGC_k(\omega)$ is the complex of graphs without univalent vertices on the form $\Gamma(\gamma)$. Then projection $\ptGC_k^*\rightarrow \ptGC_k(\omega)$ is quasi-isomorphism. Furthermore the complex $\ptGC_k(\omega)$ is isomorphic to $\dGC_k$.
\end{proposition}
\begin{proof}
    It is enough to show that $\mathcal{C}$ is acyclic.
    The proof is similar of that of proposition \ref{prop:qmGC}. First consider the subcomplex $\mathsf{s}_1\qtGC_k$ of graphs with at least one non-univalent vertex decorated by $\frac{\infty_1}{0}$ or $\frac{0}{\infty_1}$. The quotient complex is acyclic. Secondly consider the subcomplex of graphs with at least a pair of non-univalent vertices are decorated by $\frac{\infty_1}{0}$ and $\frac{0}{\infty_1}$ respectively. This complex is acyclic, and let $\mathsf{s}_2\ptGC_k$ be the corresponding quotient complex. It splits as $\mathsf{C}(\frac{\infty_1}{0})\oplus\mathcal{C}(\frac{0}{\infty_1})$ of graphs where all non-univalent vertices are decorated by $\frac{\infty_1}{0}$ and $\frac{0}{\infty_1}$ respectively. Consider the subcomplex of $\mathcal{C}(\frac{\infty_1}{0})$ of graphs with at least one univalent vertex 
    $\begin{tikzpicture}[shorten >=1pt,node distance=1.2cm,auto,>=angle 90,baseline=-0.1cm]
                \node[biw] (a) at (0,0) {$0$\nodepart{lower}$\infty_2$};
                \node[] (d) at (0,-1) {}
                    edge [->] (a);
    \end{tikzpicture}$. This complex is acyclic. The quotient complex have only two types of univalent vertices on the form 
    $\begin{tikzpicture}[shorten >=1pt,node distance=1.2cm,auto,>=angle 90,baseline=-0.1cm]
                \node[biw] (a) at (0,0) {$\infty_1$\nodepart{lower}$\infty_1$};
                \node[] (d) at (0,-1) {}
                    edge [->] (a);
    \end{tikzpicture}$
    and 
    $\begin{tikzpicture}[shorten >=1pt,node distance=1.2cm,auto,>=angle 90,baseline=-0.1cm]
                \node[biw] (a) at (0,0) {$\infty_2$\nodepart{lower}$0$};
                \node[] (d) at (0,-1) {}
                    edge [<-] (a);
    \end{tikzpicture}$. By considering a filtration over the number of non-bivalent non-antenna vertices and, we can show that this complex is acyclic on the first page. The proof is analogous showing that $\mathcal{C}(\frac{0}{\infty_1})$ is acyclic.
\end{proof}
\begin{theorem}[Main theorem V]
    Let $\dGC_k\rightarrow \pwGC_k^*$ be the map where a graph $\Gamma$ is mapped to the sum of all possible bi-weights to put on $\Gamma$, excluding the graph only decorated by $\frac{0}{0}$. Then this map is a quasi-isomorphism.
\end{theorem}
\begin{proof}
    See \ref{thm:qmGC*}.
\end{proof}
\begin{proposition}
    Let $\mathsf{otPGC}_k(\omega)$ be the subcomplex of $\ptGC_k(\omega)$ of oriented graphs. Then the induced projection $\mathsf{optGC}_k\rightarrow\mathsf{optGC}_k(\omega)$ is a quasi-isomorphism. Furthermore $\mathsf{otPGC}_k(\omega)\cong\oGC_k$.
\end{proposition}
\begin{proof}
    Same as the proof of proposition \ref{prop:ptGC(gamma)}.
\end{proof}
\begin{theorem}[Main theorem IV]
    Let $\oGC_k\rightarrow \pOWGC_k^*$ be the restriction of the map in theorem \ref{thm:qmGC*} to oriented graphs. Then this map is a quasi-isomorphism.
\end{theorem}
\begin{proof}
    See \ref{thm:qmGC*}.
\end{proof}


\begin{thebibliography}{10}
\bibitem[AM]{AM} A. Andersson, S.A. Merkulov, \textit{From deformation theory of wheeled props to classification of Kontsevich formality maps},  Int. Math. Res. Not. IMRN (2022), no. 12, 9275–9307
\bibitem[AWZ]{AWZ} A. Andersson, T. Willwacher, M.\ Živković, \textit{Oriented hairy graphs and moduli spaces of curves}, preprint arXiv:2005.00439 (2020)
\bibitem[C]{C} M. Chas, \textit{Combinatorial Lie Bialgebras of curves on surfaces}, Topology \textbf{43} (2004), no. 3, 543-568
\bibitem[CFL]{CFL} K. Cieliebak, K. Fukaya and J. Latschev, \textit{Homological algebra related to surfaces with boundary},  Quantum Topol. \textbf{11} (2020), no. 4, 691–837
\bibitem[D1]{D1} V. Drinfeld, \textit{Hamiltionian structures on Lie groups, Lie bialgebras and the geometric meaning of the classical Yang-Baxter equations}, Soviet Math. dokl. \textbf{27} (1983) 68-71
\bibitem[D2]{D2} V. Drinfeld, \textit{On some unsolved problems in quantum group theory}, Lecture Notes in Math., \textbf{1510} (1992), 1-8
\bibitem[D3]{D3} V. Drinfeld, \textit{On quasitriangular quasi-Hopf algebras and a group closely connected with $Gal(\bar{Q}/Q)$}, Leningrad Math. J. \textbf{2}, No. 4 (1991) 829-860
\bibitem[EH]{EH} B. Enriquez and G. Halbout, \textit{Quantization of QUASI-LIE bialgebras}, Journal of the American Mathematical Society. \textbf{23}, (2008).
\bibitem[ES]{ES} P. Etingof and O. Schiffmann, Lectures on Quantum Groups, International Press 2002
\bibitem[F]{F} O. Frost, \textit{Deformation theory of the wheeled properad of strongly homotopy Lie bialgebras and graph complexes}, preprint at arXiv:2212.03739 (2022)
\bibitem[G]{G} J. Granåker, \textit{Quantum BV-manifolds and quasi-Lie bialgebras}, Differential Geometry and its Applications. \textbf{28}. 194-204, (2010)
\bibitem[KM]{KM} A. Khoroshkin and S.A. Merkulov, \textit{On deformation quantization of quadratic Poisson structures}, Commun. Math. Phys. 404, No. 2, 597--628 (2023)
\bibitem[Ko]{Ko} M.\ Kontsevich, unpublished.
\bibitem[Kr]{Kr} O. Kravchenko, \textit{Strongly Homotopy Lie Bialgebras and Lie Quasi-bialgebras}, Letters in Mathematical Physics. \textbf{81} (2006) 
\bibitem[LV]{LV} J.-L.\ Loday and B.\ Vallette, \textit{Algebraic Operads}, Number 346 in Grundlehren der mathematischen Wissenschaften.
Springer, Berlin, (2012)
\bibitem[MaVo]{MaVo} M.\ Markl and A. A.\ Voronov, {\em PROPped up graph cohomology}. In: ``Algebra, Arithmetic
and Geometry - Manin Festschrift" (eds. Yu.\ Tschinkel and Yu.\ Zarhin),Vol.\  II, Progr. Math.\ vol.\ 270, Birkhauser (2010) pp. 249-281
\bibitem[MMS]{MMS} M. Markl, S. Merkulov and S. Shadrin, \textit{Wheeled props and the master equation}, preprint math.AG/0610683, J. Pure
and Appl. Algebra \textbf{213} (2009), 496-535
\bibitem[M1]{M1} S.A. Merkulov, \textit{Prop profile of Poisson geometry}, Commun. Math. Phys. \textbf{262} (2006), 117-135.
\bibitem[M2]{M2} S.A. Merkulov, \textit{Graph complexes with loops and wheels}, in ''Algebra, Arithmetic and Geometry - Manin Festschrift'' (eds. Yu. Tschinkel and Yu. Zarhin), Progress in Mathematics, Birkhauser (2010), pp. 311-354.
\bibitem[M3]{M3} S.A. Merkulov, \textit{Twisting of properads},  J. Pure Appl. Algebra 227, No. 10, Article ID 107388, 47 p. (2023)

\bibitem[M4]{M4} S.A Merkulov, \textit{Gravity prop and moduli spaces}$\mathcal{M}_{g,n}$, preprint at arXiv:2108.10644 (2021).
\bibitem[MW1]{MW1} S.A. Merkulov and T. Willwacher, \textit{Deformation theory of Lie bialgebra properads}, Geometry and physics. Vol. I, 219–247, Oxford Univ. Press, Oxford, 2018
\bibitem[MW2]{MW2} S.A. Merkulov and T. Willwacher, \textit{Props of ribbon graphs, involutive Lie bialgebras and moduli spaces of curves}, preprint arXiv:1511.07808 (2015) 51pp.
\bibitem[MW3]{MW3} S.A. Merkulov and T. Willwacher, \textit{Classification of universal formality maps for quantizations of Lie bialgebras}, Compos. Math. \textbf{156} (2020), no. 10, 2111–2148
\bibitem[Sc]{Sc} T. Schedler, \textit{A Hopf algebra quantizing a necklace Lie algebra canonically associated to a quiver.} Intern. Math. Res. Notices (2005), 725-760.
\bibitem[Tu]{Tu} V.G. Turaev, \textit{Skein quantization of Poisson algebras of loops on surfaces}, Ann. Sci. Ecole Norm. Sup (4) 24, no. 6, (1991) 635-704.
\bibitem[V]{V} B.\ Vallette, \textit{A Koszul duality for props}, Trans. Amer. Math. Soc. \textbf{359} (2007), no. 10, 4865–4943
\bibitem[W1]{W1} T.\ Willwacher, \textit{M. Kontsevich's graph complex and the Grothendieck-Teichmueller Lie algebra}, Invent. Math \textbf{200(3)} (2015), 671-760.
\bibitem[W2]{W2} T. \ Willwacher, \textit{The oriented graph complexes}, Commun. Math. Phys. \textbf{334(3)} (2015), 1649-1666.
\bibitem[We]{We} Weibel, C. (1994) \emph{An Introduction to Homological Algebra}, (Cambridge Studies in Advanced Mathematics). Cambridge: Cambridge University Press. doi:10.1017/CBO9781139644136
\bibitem[Z1]{Z1} M.\ Zivkovic, \textit{Multi-directed graph complexes and quasi-isomorphisms between them I: oriented graphs}, High. Struct. \textbf{4} (2020), no. 1, 266–283
\bibitem[Z2]{Z2} M.\ Zivkovic, \textit{Multi-directed graph complexes and quasi-isomorphisms between them II: Sourced graphs
}, Int. Math. Res. Not. IMRN 2021, no. 2, 948–1004
\bibitem[Z3]{Z3} M.\ Zivkovic, \textit{Sourced and targeted graph complexes}, to appear (2022)

\end{thebibliography}
\end{document}